\documentclass[12pt]{amsart}
\usepackage[margin=1.0in]{geometry} 
\usepackage{esint,amsthm,amsmath,amssymb,amsfonts,graphicx,color,comment,enumerate,psfrag,caption,bbm,bm,hyperref,enumitem}
\usepackage{mathtools}

\usepackage{tikz,pgf}
\usepackage{pgfplots}
\usetikzlibrary{calc}
\usetikzlibrary{patterns}

\allowdisplaybreaks

\DeclareMathOperator\supp{supp}

\newtheorem{lemma}{Lemma}[section]

\numberwithin{equation}{section}
\newtheorem{theorem}{Theorem}[section]
\newtheorem{proposition}[theorem]{Proposition}

\begin{document}
\title{Bilinear Bochner-Riesz Means for Grushin Operators}

\author[S. Bagchi, Md N. Molla, J. Singh]
{Sayan Bagchi \and Md Nurul Molla \and Joydwip Singh} 

\address[S. Bagchi]{Department of Mathematics and Statistics, Indian Institute of Science Education and Research Kolkata, Mohanpur--741246, West Bengal, India.}
\email{sayan.bagchi@iiserkol.ac.in}

\address[Md N. Molla]{Department of Mathematics and Statistics, Indian Institute of Science Education and Research Kolkata, Mohanpur--741246, West Bengal, India; Department of General Sciences, BITS Pilani Dubai Campus, International Academic City, Dubai, 345055, UAE}
\email{nurul.pdf@iiserkol.ac.in \& nurul@dubai.bits-pilani.ac.in}

\address[J. Singh]{Department of Mathematics and Statistics, Indian Institute of Science Education and Research Kolkata, Mohanpur--741246, West Bengal, India.}
\email{js20rs078@iiserkol.ac.in}



\begin{abstract}
This paper is devoted to the study of $L^{p_1} \times L^{p_2}$ to $L^{p}$ boundedness of the bilinear Bochner-Riesz mean $\mathcal{B}^{\alpha}$ associated with the Grushin operator $\mathcal{L} = -\Delta_{x'} - |x'|^2 \Delta_{x''}$ on $\mathbb{R}^{d_1} \times \mathbb{R}^{d_2}$. Our result almost resembles the corresponding Euclidean results, where the Euclidean dimension in the smoothness threshold is replaced by the topological dimension $d$ of the underlying space, except at few cases.
\end{abstract}

\maketitle

\section{Introduction}

The study of Bochner-Riesz means, a central topic in harmonic analysis, plays a key role in understanding several aspects of the field. Not only do they offer an approach to provide a framework for the Fourier inversion formula in the $L^p$ setting, but they also give insights into the theory of Fourier multipliers and the development of the Fourier restriction theory. Recall that the Bochner-Riesz means $S^{\alpha}$ of order $\alpha \geq 0$ in the Euclidean spaces $\mathbb{R}^n$ is a Fourier multiplier operator defined by
\begin{align*}
    S^{\alpha}(f)(x) = \int_{\mathbb{R}^n} \left(1-|\xi|^2\right)_{+}^{\alpha} \widehat{f}(\xi)\ e^{2 \pi i x \cdot \xi} \ d\xi,
\end{align*}
where $(r)_{+} = \max\{r, 0\}$ for $r \in \mathbb{R}$, $f \in \mathcal{S}(\mathbb{R}^n)$, the space of all Schwartz class functions in $\mathbb{R}^n$, and $\widehat{f}$ is the Fourier transform of $f$. A classical problem in the theory of Bochner-Riesz means is to find the range of $\alpha$ for which the $L^p$-boundedness of $S^{\alpha}  $ holds for $1\leq p \leq \infty$. The famous Bochner-Riesz conjecture states that  for $1\leq p \leq \infty$ and $p\neq 2$,  $S^{\alpha}$ is bounded on $L^p(\mathbb{R}^n)$ if and only if $\alpha> \alpha(p) = \max\Big\{n|\frac{1}{p}-\frac{1}{2}|-\frac{1}{2}, 0 \Big\}$. For $n=2$, the conjecture was settled by Carleson and Sj\"olin \cite{Carleson_Sjolin_Multiplier_Problem_on_Disc_1972}. However, for $n \geq 3$ the conjecture still remains open, with only partial progress having been made so far.  For a detailed historical background and recent progress on Bochner-Riesz conjecture, we refer to  \cite{Tao_Recent_Progress_Restriction_conjecture_2004}, \cite{Bourgain_Guth_Oscillatory_Integral_2011}, \cite{Lee_Improved_Bounds_Bochner_Riesz_Maximal_2004}, \cite{Bourgain_Besicovitch_type_maximal_1991}, \cite{Tao_Vargas_Bilinear_approach_2000} and references therein.

The Bilinear Bochner-Riesz means can be viewed as a natural bilinear extension of classical (linear) Bochner-Riesz mean $S^{\alpha}$. For $f, g\in \mathcal{S}(\mathbb{R}^n)$ and for $\alpha \geq 0$, the bilinear Bochner-Riesz mean, denoted by $B^{\alpha}$, is defined as
\begin{align*}
    B^{\alpha}(f,g)(x) = \int_{\mathbb{R}^n} \int_{\mathbb{R}^n} \left(1-|\xi|^2-|\eta|^2 \right)_{+}^{\alpha} \widehat{f}(\xi)\, \widehat{g}(\eta)\, e^{2 \pi i x \cdot (\xi + \eta)} \, d\xi \,  d\eta .
\end{align*}
This operator is also closely related to the convergence of the product of two $n$-dimensional Fourier series; see \cite{Bernicot_Grafakos_Song_Yan_Bilinear_Bochner_Riesz_2015} for further details. As in the linear setting, the bilinear Bochner-Riesz problem concerns determining the range of $\alpha$ for which $B^{\alpha}$ is bounded from $L^{p_1}(\mathbb{R}^n) \times L^{p_2} (\mathbb{R}^n)$ to  $L^{p}(\mathbb{R}^n)$ for $1\leq p_1, p_2 \leq \infty$ with $1/p=1/p_1 + 1/p_2$. In the one-dimensional case ($n=1$), this problem was studied by \cite[Theorem 4.1]{Bernicot_Grafakos_Song_Yan_Bilinear_Bochner_Riesz_2015}, \cite{Grafakos_Li_Disc_Multiplier_2006} and \cite[Theorem 2.2]{Jotsaroop_Shrivastava_Maximal_Bochner_Riesz_2022}. For dimensions $n\geq 2$ and $\alpha=0$, Diestel and Grafakos \cite{Diestel_Grafakos_Ball_multiplier_problem_2007}, proved that if exactly one of $p_1, p_2, p'$ is less than $2$, then  $B^{\alpha}$ is not bounded from $L^{p_1}(\mathbb{R}^n) \times L^{p_2}(\mathbb{R}^n)$ to $L^{p}(\mathbb{R}^n)$, where $p'$ is the conjugate exponent of $p$. For $n\geq 2$ and $\alpha>0$, the work of \cite{Bernicot_Grafakos_Song_Yan_Bilinear_Bochner_Riesz_2015} marks the beginning of the investigation of this problem and established a series of positive results along with some necessary conditions as well. Shortly afterward, subsequent research led to improvements of these results in two directions. For the range $2\leq p_1,p_2\leq \infty$, Jeong, Lee and Vargas \cite{Jeong_Lee_Vargas_Bilinear_Bochner_Riesz_2018} obtained new results by improving the lower bounds on the index  $\alpha$. Their approach relies on a new decomposition of the bilinear Bochner-Riesz operator into a product of square functions in the pointwise sense, see \cite[Section 3]{Jeong_Lee_Vargas_Bilinear_Bochner_Riesz_2018} for more details. On the other hand,  Liu and Wang \cite{Liu_Wang_Bilinear_Bochner_Riesz_Non_Banach_2020} significantly improved the result of \cite{Bernicot_Grafakos_Song_Yan_Bilinear_Bochner_Riesz_2015} for $p<1$  by obtaining lower smoothness thresholds for $\alpha$. Below, we highlight some key results from \cite{Bernicot_Grafakos_Song_Yan_Bilinear_Bochner_Riesz_2015} and \cite{Liu_Wang_Bilinear_Bochner_Riesz_Non_Banach_2020}.

\begin{theorem}\cite[Proposition 4.10, 4.11 and 4.2]{Bernicot_Grafakos_Song_Yan_Bilinear_Bochner_Riesz_2015}, \cite[Theorem 1.1]{Liu_Wang_Bilinear_Bochner_Riesz_Non_Banach_2020}
\label{Theorem: Euclidean bilinear Bochner-Riesz for Grafakos}
Let $n \geq 2$ and $1\leq p_1, p_2 \leq \infty$ with $1/p = 1/p_1 +1/p_2$. Then $B^{\alpha}$ is bounded from $L^{p_1}(\mathbb{R}^n) \times L^{p_2}(\mathbb{R}^n)$ to $L^{p}(\mathbb{R}^n)$ if $p_1, p_2, p$ and $\alpha$ satisfy one of the following conditions:
\begin{enumerate}
    \item (Region I) $2 \leq p_1, p_2 < \infty$, $1\leq p \leq 2$ and $\alpha> (n-1)(1-\frac{1}{p})$.
    \item (Region II) $2 \leq p_1, p_2, p < \infty$ and $\alpha> \frac{n-1}{2} + n(\frac{1}{2}-\frac{1}{p})$.
    \item (Region III) $2 \leq p_2 < \infty$, $1\leq p_1, p < 2$ and $\alpha> n(\frac{1}{2}-\frac{1}{p_2})-(1-\frac{1}{p})$.
    \item (Region III) $2 \leq p_1 < \infty$, $1\leq p_2, p < 2$ and $\alpha> n(\frac{1}{2}-\frac{1}{p_1})-(1-\frac{1}{p})$.
    \item (Region IV) $1\leq p_1 \leq 2 \leq p_2 \leq \infty$, $0<p<1$ and $\alpha> n(\frac{1}{p_1}-\frac{1}{2})$.
    \item (Region IV) $1\leq p_2 \leq 2 \leq p_1 \leq \infty$, $0<p<1$ and $\alpha> n(\frac{1}{p_2}-\frac{1}{2})$.
    \item (Region V) $1\leq p_1 \leq p_2 \leq 2$ and $\alpha> n(\frac{1}{p}-1)-(\frac{1}{p_2}-\frac{1}{2})$.
    \item (Region V) $1\leq p_2 \leq p_1 \leq 2$ and $\alpha> n(\frac{1}{p}-1)-(\frac{1}{p_1}-\frac{1}{2})$. 
\end{enumerate}
\end{theorem}

\vspace{-0.7cm}
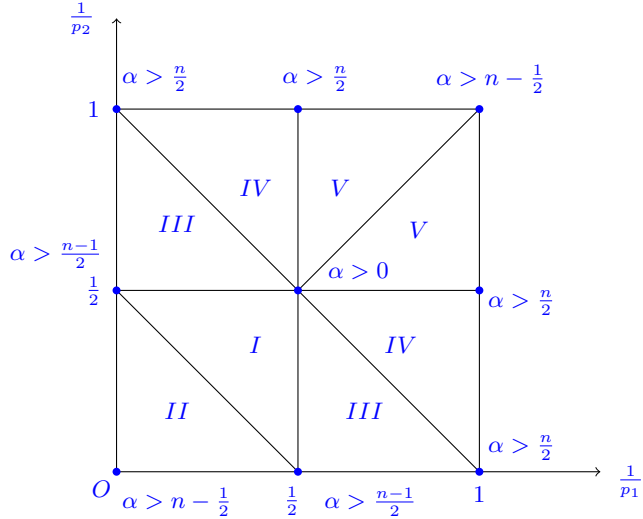
\begin{figure}[!ht]
\begin{centering}
\definecolor{qqqqff}{rgb}{0,0,1}
\begin{tikzpicture}[line cap=round,line join=round,x=0.8cm,y=0.8cm]
%\clip(-8.08,-9.17) rectangle (25.07,9.24);
\draw (6,6)-- (0,6);
\draw (6,0)-- (6,6);
\draw (0,3)-- (6,3);
\draw [->] (0,0) -- (8,0);
\draw [->] (0,0) -- (0,7.5);
\draw (3,6)-- (3,3);
\draw (0,6)-- (6,0);
\draw (3,3)-- (3,0);
\draw (3,3)-- (6,6);
\draw (0,3)-- (3,0);
\begin{scriptsize}
\fill [color=qqqqff] (0,0) circle (1.5pt);
\draw[color=qqqqff] (1,-.5) node {$\alpha>n-\frac{1}{2}$};
% \fill [color=qqqqff] (8,0) circle (0pt);
% \draw[color=qqqqff] (8.21,0.33) node {$B$};
% \fill [color=qqqqff] (0,8) circle (0pt);
% \draw[color=qqqqff] (0.21,8.33) node {$C$};
\fill [color=qqqqff] (0,6) circle (1.5pt);
\draw[color=qqqqff] (0.65,6.5) node {$\alpha > \frac{n}{2}$};
\fill [color=qqqqff] (6,6) circle (1.5pt);
\draw[color=qqqqff] (6.19,6.5) node {$\alpha>n-\frac{1}{2}$};
\fill [color=qqqqff] (6,0) circle (1.5pt);
\draw[color=qqqqff] (6.7,0.4) node {$\alpha>\frac{n}{2}$};
\fill [color=qqqqff] (3,0) circle (1.5pt);
\draw[color=qqqqff] (4.2,-0.5) node {$\alpha>\frac{n-1}{2}$};
\fill [color=qqqqff] (3,6) circle (1.5pt);
\draw[color=qqqqff] (3.3,6.5) node {$\alpha>\frac{n}{2}$};
\fill [color=qqqqff] (0,3) circle (1.5pt);
\draw[color=qqqqff] (-1,3.6) node {$\alpha>\frac{n-1}{2}$};
\fill [color=qqqqff] (3,3) circle (1.5pt);
\draw[color=qqqqff] (4,3.33) node {$\alpha>0$};
\fill [color=qqqqff] (6,3) circle (1.5pt);
\draw[color=qqqqff] (6.7,2.8) node {$\alpha>\frac{n}{2}$};
\fill [color=qqqqff] (6,-0.5) circle (0pt);
\draw[color=qqqqff] (6.0,-0.38) node {$1$};
\fill [color=qqqqff] (3,-0.5) circle (0pt);
\draw[color=qqqqff] (2.9,-0.48) node {$\frac{1}{2}$};
\fill [color=qqqqff] (-0.25,-0.42) circle (0pt);
\draw[color=qqqqff] (-0.25,-0.28) node {$O$};
\fill [color=qqqqff] (-0.5,3) circle (0pt);
\draw[color=qqqqff] (-0.38,3.0) node {$\frac{1}{2}$};
\fill [color=qqqqff] (-0.5,6) circle (0pt);
\draw[color=qqqqff] (-0.38,6.0) node {$1$};
\fill [color=qqqqff] (8.5,-0.5) circle (0pt);
\draw[color=qqqqff] (8.5,-0.16) node {$\frac{1}{p_1}$};
\fill [color=qqqqff] (-0.5,8.5) circle (0pt);
\draw[color=qqqqff] (-0.58,7.5) node {$\frac{1}{p_2}$};
\draw[color=qqqqff] (2.3,2.1) node {$I$};
\draw[color=qqqqff] (1.0,1.0) node {$II$};
\draw[color=qqqqff] (4.1,1.0) node {$III$};
\draw[color=qqqqff] (1.0,4.1) node {$III$};
\draw[color=qqqqff] (4.7,2.1) node {$IV$};
\draw[color=qqqqff] (2.3,4.7) node {$IV$};
\draw[color=qqqqff] (3.7,4.7) node {$V$};
\draw[color=qqqqff] (5,4) node {$V$};
\end{scriptsize}
\end{tikzpicture}
        \caption{Here $O=(0,0)$, and $\alpha>\alpha(p_1, p_2)$ represents that $B^{\alpha}$ is bounded on $L^{p_1}(\mathbb{R}^n) \times L^{p_2}(\mathbb{R}^n) \to L^p(\mathbb{R}^n)$ for $\alpha$ bigger that $\alpha(p_1, p_2)$, see Theorem \ref{Theorem: Euclidean bilinear Bochner-Riesz for Grafakos}.}
\end{centering}        
\end{figure}

Since the Euclidean Laplacian $-\Delta$ is positive, essentially self-adjoint operator, one can see that via the spectral resolution of $-\Delta$, the Bochner-Riesz mean is given by
\begin{align*}
  S^{\alpha} := (I-(-\Delta))_{+}^{\alpha}.  
\end{align*}
This perspective naturally allows one to associate the notion of Bochner-Riesz means to other positive, self-adjoint operators as well. In this context,  we are particularly interested in the Grushin operator $\mathcal{L}$, defined on $\mathbb{R}^d := \mathbb{R}^{d_1} \times \mathbb{R}^{d_2}$, where $d_1, d_2 \geq 1$, and the space equipped with the standard Lebesgue measure $dx$. The operator $\mathcal{L}$ is given by
\begin{align*}
    \mathcal{L} 
    &= -\Delta_{x'} - |x'|^2 \Delta_{x''},
\end{align*}
where $x=(x',x'')\in \mathbb{R}^{d_1} \times \mathbb{R}^{d_2}$, while $\Delta_{x'}$, $\Delta_{x''}$ are the Laplacian on $\mathbb{R}^{d_1}$, $\mathbb{R}^{d_2}$ respectively and $|x'|$ denotes the Euclidean norm of $x'$. The Grushin operator $\mathcal{L}$ is positive and essentially self-adjoint on $L^2(\mathbb{R}^d)$. Although $\mathcal{L}$ fails to be elliptic on the plane $x' = 0$, it is nevertheless hypoelliptic, due to a result of H\"ormander \cite{Hormander_Hypoellptic_second_order_differential_equatuion_1967}.
The spectral decomposition of $\mathcal{L}$ is well known (see \cite{Bagchi_Garg_Grushin_2024}). For $f \in \mathcal{S}(\mathbb{R}^d)$, let
\begin{align*}
    \mathcal{F}_2f(x', \lambda) := f^{\lambda}(x') = \int_{\mathbb{R}^{d_2}} f(x',x'') e^{-i \lambda \cdot x''} \ dx''
\end{align*}
denote the Fourier transform of $f$ in the second variable $x''$. Set $[k] := 2k+d_1$. For $\alpha>0 $, the Bochner-Riesz mean associated with $\mathcal{L}$ is defined by
\begin{align}
\label{Bochner-Riesz for Grushin}
    S^{\alpha}(\mathcal{L}) f(x) &= \frac{1}{(2\pi)^{d_2}} \int_{\mathbb{R}^{d_2}} e^{i \lambda \cdot x''} \sum_{k=0}^{\infty} \left(1-[k]|\lambda| \right)_{+}^{\alpha} P_k^{\lambda}f^{\lambda}(x') \ d\lambda ,
\end{align}
where $P_k^{\lambda}$ denotes the orthogonal projection onto the eigenspace of scaled Hermite operator $H(\lambda):= (-\Delta_{x'} + |x'|^2 |\lambda|^2)$ for $\lambda \neq 0$, corresponding to the eigenvalue $[k]|\lambda|$.

Given that $\mathcal{L}$ is a second order, subelliptic differential operator, the notion of ``control distance'' $\varrho$ associated with $\mathcal{L}$ can be introduced, 
and as illustrated in \cite{Robinson_Sikora_Degenerate_Elliptic_Operator_2008}, \cite{Martini_Sikora_Grushin_Weighted_Plancherel_2012}, $(\mathbb{R}^d, \varrho, dx)$ becomes a doubling metric-measure space with homogeneous dimension $Q= d_1 + 2d_2$. Also note that $d= d_1 + d_2$ is called the topological dimension of $\mathbb{R}^d$. In view of a general theorem in the context of doubling metric-measure space \cite{Duong_Ouhabaz_Sikora_Sharp_Multiplier_2002}, it follows that $S^{\alpha}(\mathcal{L})$ is bounded on $L^p(\mathbb{R}^d)$ for $1\leq p \leq \infty$, whenever $\alpha > Q/2$ (see also \cite{Martini_Sikora_Grushin_Weighted_Plancherel_2012}). For subelliptic operators, the $L^p$-boundedness of Bochner-Riesz multipliers or more generally, spectral multipliers with smoothness conditions formulated in terms of the homogeneous dimension $Q$, in many cases, are not sharp. As a result, a considerable amount of attention has been devoted to determining the optimal smoothness threshold for spectral multiplier theorems on various non-Euclidean settings. The line of investigation began with the groundbreaking work of M\"uller and Stein \cite{Muller_Stein_Spectral_Multiplier-Heisenberg_1994} and independently by Hebisch \cite{Hebisch_Spectral_Multiplier_Heisenberg_1993}, in the context of Heisenberg (-type) groups and subsequently, significant amount of research has been carried out in many different settings; see, for instance \cite{Cowling_Sikora_spectral_mult_sublaplacian_2001}, \cite{Cowling_Klima_Sikora_Kohn_Laplacian_On_Sphere_2011}, \cite{Marini_Multiplier_polynomial_Growth_2012}, \cite{Martini_Sikora_Grushin_Weighted_Plancherel_2012}, \cite{Martini_Spectral_Multiplier_Heisenberg_Reiter_2015}, \cite{Ahrens_Cowling_Martini_Muller_Quaternionic_Sphere_2020} and references therein. Towards this direction, let us highlight some of the progress made for the operator $S^{\alpha}(\mathcal{L})$. In \cite{Martini_Sikora_Grushin_Weighted_Plancherel_2012}, Martini and Sikora proved that if $\alpha> \max\{d_1+d_2, 2d_2\}/2-1/2$, then $S^{\alpha}(\mathcal{L})$ is bounded on $L^p(\mathbb{R}^d)$ for all $1\leq p \leq \infty$. It is important to remark that for $d_1 \geq d_2$, the smoothness order simplifies to $\alpha>(d-1)/2$, and this result is sharp (see \cite{Martini_Sikora_Grushin_Weighted_Plancherel_2012}). Subsequently, Martini and M\"uller \cite{Martini_Muller_Sharp_Multiplier_Grushin_2014} removed the restriction $d_1 \geq d_2$ by showing that the condition $\alpha>(d-1)/2$ is indeed sharp for $L^p$-boundedness of $S^{\alpha}(\mathcal{L})$ for all $d_1, d_2 \geq 1$. However, the scenario is quiet different when $0< \alpha \leq \frac{d-1}{2}$. One cannot expect $S^{\alpha}(\mathcal{L})$ to be bounded for all $p \in [1, \infty]$. In this context,  Chen and Ouhabaz \cite{Chen_Ouhabaz_Bochner_Riesz_Grushin_2016} proved a $p$-specific boundedness result for $\mathcal{L}$ with smoothness indices $\alpha> \max\{d_1+d_2, 2d_2\}|1/p-1/2|-1/2 $. This result was recently  improved by Niedorf in \cite{Niedorf_Bochner_Riesz_Grushin_2022}, who proved that for $1 \leq  p \leq \min{\{2d_1/(d_1 + 2), (2d_2 + 2)/(d_2 + 3)\}}$ and $\alpha >d|1/p-1/2|-1/2$, the Bochner–Riesz mean $S^{\alpha}(\mathcal{L})$ is bounded on $L^p(\mathbb{R}^d )$.  

This linear phenomenon motivates us to study corresponding boundedness problem for the bilinear Bochner-Riesz operator associated with Grushin operator $\mathcal{L}$. Before proceeding further, let us define $\mathcal{E}(\mathbb{R}^d) := \cup_{l \in \mathbb{N}} \mathcal{E}_l(\mathbb{R}^d)$, where each $\mathcal{E}_l(\mathbb{R}^d)$ is defined as follows
\begin{align*}
    \mathcal{E}_l(\mathbb{R}^d) = \{ f \in L^2(\mathbb{R}^d) : f^{\lambda}(x')= \sum_{|\mu|\leq l} C(\lambda, \mu) \Phi_{\mu}^{\lambda}(x'), \  \text{for a bounded function}\  C(\lambda, \mu)\\ \text{which is compactly supported in}\ \lambda\text{-variable} \}.
\end{align*}
It was shown in \cite{Bagchi_Garg_Grushin_2024} that $\mathcal{E}(\mathbb{R}^d)$ is dense in $L^2(\mathbb{R}^d)$. In particular, one can show that the operator $S^{\alpha}(\mathcal{L})$ (see \eqref{Bochner-Riesz for Grushin}) is densely defined on $L^2(\mathbb{R}^d)$. For $f,g \in \mathcal{E}(\mathbb{R}^d)$ and $\alpha, R>0$, the bilinear Bochner-Riesz operator associated with the Grushin operator $\mathcal{L}$ is defined by
\begin{align}
\label{Bilinear Bochner-Riesz operator definition with R}
    \mathcal{B}_R^{\alpha}(f,g)(x) &= \frac{1}{(2\pi)^{2 d_2}} \int_{\mathbb{R}^{d_2}} \int_{\mathbb{R}^{d_2}} e^{i (\lambda_1 + \lambda_2) \cdot x''} \sum_{k_1, k_2=0}^{\infty} \left(1-\frac{[k_1]|\lambda_1|+[k_2]|\lambda_2|}{R} \right)_{+}^{\alpha} \\
    &\nonumber \hspace{7cm} P_{k_1}^{\lambda_1}f^{\lambda_1}(x') P_{k_2}^{\lambda_2}g^{\lambda_2}(x') \  d\lambda_1 \  d\lambda_2 .
\end{align}

In this paper, our goal is to obtain analogous result as of Theorem \ref{Theorem: Euclidean bilinear Bochner-Riesz for Grafakos}, with smoothness threshold $\alpha(p_1, p_2)$ possibly expressed in terms of the topological dimension $d$, replacing the Euclidean dimension $n$, such that whenever $\alpha> \alpha(p_1, p_2)$, there exists a positive constant $C>0$, uniformly over $R>0$, for which
\begin{align*}
    \|\mathcal{B}^{\alpha}_R(f,g)\|_{L^p(\mathbb{R}^d)} &\leq C \, \|f\|_{L^{p_1}(\mathbb{R}^d)} \|g\|_{L^{p_2}(\mathbb{R}^d)},
\end{align*}
holds for $f, g \in \mathcal{E}(\mathbb{R}^d)$ with $1\leq p_1, p_2 \leq \infty$ and $1/p = 1/p_1+1/p_2$. Note that on $\mathbb{R}^d$, we have the family of non-isotropic dilation $\{\delta_t\}_{t>0}$, defined by $\delta_t(x', x'')= (tx', t^2x'')$. Using these  dilation, one can easily check that
\begin{align*}
    \mathcal{B}_{t^{-2}R}^{\alpha}(f,g)(x) = \delta_{t^{-1}} \circ \mathcal{B}_R^{\alpha} (\delta_t f, \delta_t g)(x) .
\end{align*}
In view of the above relation, it is enough to consider the case $R=1$. When $R=1$, we simply denote $\mathcal{B}_1^{\alpha} =: \mathcal{B}^{\alpha}$.

Set $D=\max\{d_1+d_2, 2d_2\}$ and $\mathfrak{D} = \min\{D, d_1+d_2+1\}$. The following is our first main result concerning the boundedness of $\mathcal{B}^{\alpha}$.

\begin{theorem}
\label{Bilinear Bochner-Riesz main theorem}
Let $1\leq p_1, p_2 \leq \infty$ with $1/p = 1/p_1 +1/p_2$. Then $\mathcal{B}^{\alpha}$ is bounded from $L^{p_1}(\mathbb{R}^d) \times L^{p_2}(\mathbb{R}^d)$ to $L^{p}(\mathbb{R}^d)$ provided that $p_1, p_2, p$ and $\alpha> \alpha(p_1, p_2)$ satisfy one of the following conditions:
\begin{enumerate}
    \item (Region I) $2 \leq p_1, p_2 < \infty$, $1\leq p \leq 2$ and $\alpha(p_1, p_2)= (d-1)(1-\frac{1}{p})$.
    \item (Region II) $2 \leq p_1, p_2, p < \infty$ and $\alpha(p_1, p_2)= \frac{d-1}{2} + d(\frac{1}{2}-\frac{1}{p})$.
    \item (Region III) $2 \leq p_2 \leq \infty$, $1\leq p_1, p \leq 2$ and $\alpha(p_1, p_2)= Q(\frac{1}{p_1}-\frac{1}{2})+(d-1)(1-\frac{1}{p})$.
    \item (Region III) $2 \leq p_1 \leq \infty$, $1\leq p_2, p \leq 2$ and $\alpha(p_1, p_2)= Q(\frac{1}{p_2}-\frac{1}{2})+(d-1)(1-\frac{1}{p})$.
    \item (Region IV) $1\leq p_1 \leq 2 \leq p_2 \leq \infty$, $0<p\leq 1$ and $\alpha(p_1, p_2)= \mathfrak{D}(\frac{1}{p}-1)+Q(\frac{1}{2}-\frac{1}{p_2})$.
    \item (Region IV) $1\leq p_2 \leq 2 \leq p_1 \leq \infty$, $0<p \leq 1$ and $\alpha(p_1, p_2)= \mathfrak{D}(\frac{1}{p}-1)+Q(\frac{1}{2}-\frac{1}{p_1})$.
    \item (Region V) $1\leq p_1, p_2 \leq 2$ and $\alpha(p_1, p_2)= \mathfrak{D}(\frac{1}{p}-1)$.
\end{enumerate}
\end{theorem}

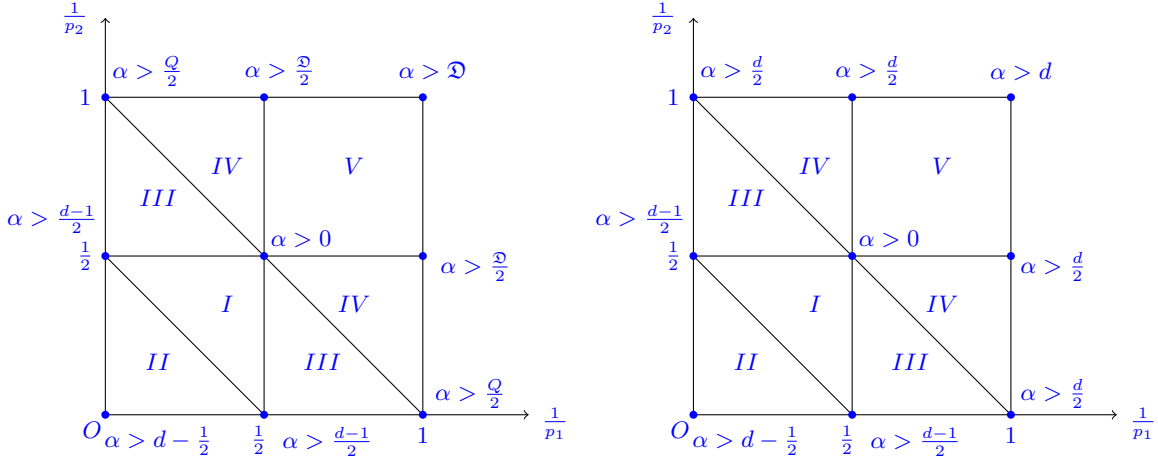
\begin{figure}[!ht]
\begin{centering}
\definecolor{qqqqff}{rgb}{0,0,1}
\begin{tikzpicture}[line cap=round,line join=round,x=0.7cm,y=0.7cm]
%\clip(-8.08,-9.17) rectangle (25.07,9.24);
\draw (6,6)-- (0,6);
\draw (6,0)-- (6,6);
\draw (0,3)-- (6,3);
\draw [->] (0,0) -- (8,0);
\draw [->] (0,0) -- (0,7.5);
\draw (3,6)-- (3,3);
\draw (0,6)-- (6,0);
\draw (3,3)-- (3,0);
\draw (0,3)-- (3,0);
%\draw (0,0)-- (3,3);
\begin{scriptsize}
\fill [color=qqqqff] (0,0) circle (1.5pt);
\draw[color=qqqqff] (1,-.5) node {$\alpha>d-\frac{1}{2}$};
% \fill [color=qqqqff] (8,0) circle (0pt);
% \draw[color=qqqqff] (8.21,0.33) node {$B$};
% \fill [color=qqqqff] (0,8) circle (0pt);
% \draw[color=qqqqff] (0.21,8.33) node {$C$};
\fill [color=qqqqff] (0,6) circle (1.5pt);
\draw[color=qqqqff] (0.8,6.5) node {$\alpha > \frac{Q}{2}$};
\fill [color=qqqqff] (6,6) circle (1.5pt);
\draw[color=qqqqff] (6.19,6.5) node {$\alpha>\mathfrak{D}$};
\fill [color=qqqqff] (6,0) circle (1.5pt);
\draw[color=qqqqff] (6.9,0.4) node {$\alpha>\frac{Q}{2}$};
\fill [color=qqqqff] (3,0) circle (1.5pt);
\draw[color=qqqqff] (4.2,-0.5) node {$\alpha>\frac{d-1}{2}$};
\fill [color=qqqqff] (3,6) circle (1.5pt);
\draw[color=qqqqff] (3.3,6.5) node {$\alpha>\frac{\mathfrak{D}}{2}$};
\fill [color=qqqqff] (0,3) circle (1.5pt);
\draw[color=qqqqff] (-1,3.7) node {$\alpha>\frac{d-1}{2}$};
\fill [color=qqqqff] (3,3) circle (1.5pt);
\draw[color=qqqqff] (3.7,3.33) node {$\alpha>0$};
\fill [color=qqqqff] (6,3) circle (1.5pt);
\draw[color=qqqqff] (7,2.8) node {$\alpha>\frac{\mathfrak{D}}{2}$};
\fill [color=qqqqff] (6,-0.5) circle (0pt);
\draw[color=qqqqff] (6.0,-0.38) node {$1$};
\fill [color=qqqqff] (3,-0.5) circle (0pt);
\draw[color=qqqqff] (2.9,-0.48) node {$\frac{1}{2}$};
\fill [color=qqqqff] (-0.25,-0.42) circle (0pt);
\draw[color=qqqqff] (-0.25,-0.28) node {$O$};
\fill [color=qqqqff] (-0.5,3) circle (0pt);
\draw[color=qqqqff] (-0.38,3.0) node {$\frac{1}{2}$};
\fill [color=qqqqff] (-0.5,6) circle (0pt);
\draw[color=qqqqff] (-0.38,6.0) node {$1$};
\fill [color=qqqqff] (8.5,-0.5) circle (0pt);
\draw[color=qqqqff] (8.5,-0.16) node {$\frac{1}{p_1}$};
\fill [color=qqqqff] (-0.5,8.5) circle (0pt);
\draw[color=qqqqff] (-0.58,7.5) node {$\frac{1}{p_2}$};
\draw[color=qqqqff] (2.3,2.1) node {$I$};
\draw[color=qqqqff] (1.0,1.0) node {$II$};
\draw[color=qqqqff] (4.1,1.0) node {$III$};
\draw[color=qqqqff] (1.0,4.1) node {$III$};
\draw[color=qqqqff] (4.7,2.1) node {$IV$};
\draw[color=qqqqff] (2.3,4.7) node {$IV$};
\draw[color=qqqqff] (4.7,4.7) node {$V$};
\end{scriptsize}
\end{tikzpicture}
\begin{tikzpicture}[line cap=round,line join=round,x=0.7cm,y=0.7cm]
%\clip(-8.08,-9.17) rectangle (25.07,9.24);
\draw (6,6)-- (0,6);
\draw (6,0)-- (6,6);
\draw (0,3)-- (6,3);
\draw [->] (0,0) -- (8,0);
\draw [->] (0,0) -- (0,7.5);
\draw (3,6)-- (3,3);
\draw (0,6)-- (6,0);
\draw (3,3)-- (3,0);
\draw (0,3)-- (3,0);
%\draw (0,0)-- (3,3);
\begin{scriptsize}
\fill [color=qqqqff] (0,0) circle (1.5pt);
\draw[color=qqqqff] (1,-.5) node {$\alpha>d-\frac{1}{2}$};
% \fill [color=qqqqff] (8,0) circle (0pt);
% \draw[color=qqqqff] (8.21,0.33) node {$B$};
% \fill [color=qqqqff] (0,8) circle (0pt);
% \draw[color=qqqqff] (0.21,8.33) node {$C$};
\fill [color=qqqqff] (0,6) circle (1.5pt);
\draw[color=qqqqff] (0.75,6.5) node {$\alpha > \frac{d}{2}$};
\fill [color=qqqqff] (6,6) circle (1.5pt);
\draw[color=qqqqff] (6.19,6.5) node {$\alpha>d$};
\fill [color=qqqqff] (6,0) circle (1.5pt);
\draw[color=qqqqff] (6.8,0.4) node {$\alpha>\frac{d}{2}$};
\fill [color=qqqqff] (3,0) circle (1.5pt);
\draw[color=qqqqff] (4.2,-0.5) node {$\alpha>\frac{d-1}{2}$};
\fill [color=qqqqff] (3,6) circle (1.5pt);
\draw[color=qqqqff] (3.3,6.5) node {$\alpha>\frac{d}{2}$};
\fill [color=qqqqff] (0,3) circle (1.5pt);
\draw[color=qqqqff] (-1,3.7) node {$\alpha>\frac{d-1}{2}$};
\fill [color=qqqqff] (3,3) circle (1.5pt);
\draw[color=qqqqff] (3.7,3.33) node {$\alpha>0$};
\fill [color=qqqqff] (6,3) circle (1.5pt);
\draw[color=qqqqff] (6.8,2.8) node {$\alpha>\frac{d}{2}$};
\fill [color=qqqqff] (6,-0.5) circle (0pt);
\draw[color=qqqqff] (6.0,-0.38) node {$1$};
\fill [color=qqqqff] (3,-0.5) circle (0pt);
\draw[color=qqqqff] (2.9,-0.48) node {$\frac{1}{2}$};
\fill [color=qqqqff] (-0.25,-0.42) circle (0pt);
\draw[color=qqqqff] (-0.25,-0.28) node {$O$};
\fill [color=qqqqff] (-0.5,3) circle (0pt);
\draw[color=qqqqff] (-0.38,3.0) node {$\frac{1}{2}$};
\fill [color=qqqqff] (-0.5,6) circle (0pt);
\draw[color=qqqqff] (-0.38,6.0) node {$1$};
\fill [color=qqqqff] (8.5,-0.5) circle (0pt);
\draw[color=qqqqff] (8.5,-0.16) node {$\frac{1}{p_1}$};
\fill [color=qqqqff] (-0.5,8.5) circle (0pt);
\draw[color=qqqqff] (-0.58,7.5) node {$\frac{1}{p_2}$};
\draw[color=qqqqff] (2.3,2.1) node {$I$};
\draw[color=qqqqff] (1.0,1.0) node {$II$};
\draw[color=qqqqff] (4.1,1.0) node {$III$};
\draw[color=qqqqff] (1.0,4.1) node {$III$};
\draw[color=qqqqff] (4.7,2.1) node {$IV$};
\draw[color=qqqqff] (2.3,4.7) node {$IV$};
\draw[color=qqqqff] (4.7,4.7) node {$V$};
\end{scriptsize}
\end{tikzpicture}
        \caption{Here $O=(0,0)$, and $\alpha>\alpha(p_1, p_2)$ represents that $\mathcal{B}^{\alpha}$ is bounded on $L^{p_1}(\mathbb{R}^d) \times L^{p_2}(\mathbb{R}^d) \to L^p(\mathbb{R}^d)$ for any $\alpha>\alpha(p_1, p_2)$. The left side picture described by Theorem \ref{Bilinear Bochner-Riesz main theorem}, while right side picture described by Theorem \ref{Theorem: Bilinear Bochner-Riesz with Fourier transform away from origin}.}
\end{centering}        
\end{figure}

A few words are in order to emphasize the key contributions of this result. In fact, we have proved the boundedness of $\mathcal{B}^{\alpha}$ at some specific points and then Theorem \ref{Bilinear Bochner-Riesz main theorem} is obtained by a bilinear interpolation from \cite[Section 4.3]{Bernicot_Grafakos_Song_Yan_Bilinear_Bochner_Riesz_2015}. For the points $(p_1,p_2,p)=(\infty, \infty, \infty)$, $(2,\infty,2)$ and $(\infty,2,2)$, that is for the regions $I$ and $II$, our result becomes an exact analogue of the corresponding  Euclidean result (see Theorem \ref{Theorem: Euclidean bilinear Bochner-Riesz for Grafakos}), where the Euclidean dimension $n$ in the smoothness threshold $\alpha(p_1, p_2)$ is replaced by the topological dimension $d$ of the underlying space $\mathbb{R}^d$. First, observe that when $d_1 \geq d_2$, the parameter $D$ equals the topological dimension $d$  of $\mathbb{R}^d$. As a consequence, for $d_1 \geq d_2$ at $(1,1,1/2)$, $(1,2,2/3)$ and by symmetry at $(2,1,2/3)$, the smoothness threshold in our result requires $\alpha>d$ and $\alpha>d/2$ respectively. These results are clearly better than those obtained for M\'etivier groups in \cite[Theorem 1.2]{Bagchi_Molla_Singh_Bilinear_Metivier_2025}, where the results were proved for $\alpha>d+1$ and $\alpha>(d+1)/2$ respectively. But for $ d_2 \geq d_1$, in case of Grushin operator, the smoothness parameter in our result coincides with result one obtained for M\'etivier groups \cite[Theorem 1.2]{Bagchi_Molla_Singh_Bilinear_Metivier_2025}. It is worth mentioning that at $(1,1,1/2)$ the Euclidean result holds for $\alpha>n-1/2$ (see Theorem \ref{Theorem: Euclidean bilinear Bochner-Riesz for Grafakos}). These differences arise because of the absence of an explicit kernel representation in Grushin setting, unlike in the Euclidean setting, where the kernel can be expressed in terms of Bessel function (see \cite[Proposition 4.2 (i)]{Bernicot_Grafakos_Song_Yan_Bilinear_Bochner_Riesz_2015}). 

% On the other hand, in the regions $III$ and $IV$, the smoothness parameter depends on $D$, $\mathfrak{D}$ as well as on $Q$. This is because, in these regions, the result is obtained by interpolating the points $(p_1, p_2,p)=(2,2,1)$, $(1,2,2/3)$, $(1, \infty,1)$, and $(2,\infty,2)$, together with the use of a  duality argument. 
% At this point, it is worth mentioning that in Euclidean setup, the boundedness at $(1,1,1/2)$ is obtained from the explicit bilinear kernel expression, which is given by in terms of the Bessel functions (see \cite[Proposition 4.2 (i)]{Bernicot_Grafakos_Song_Yan_Bilinear_Bochner_Riesz_2015}). But, in Grushing setting,  such explicit kernel expression for $\mathcal{B}^{\alpha}$ is not well known.

The above result can be improved further if some additional support conditions are imposed on the input functions. Under these assumptions, one can recover the precise analogue of Theorem \ref{Theorem: Euclidean bilinear Bochner-Riesz for Grafakos}, with smoothness parameter $\alpha(p_1, p_2)$ expressed in terms of topological dimension $d$, replacing the Euclidean dimension $n$. The exception only occurs at the point $(1,1,1/2)$, where there is a loss of $1/2$ in the required smoothness.

\begin{theorem}
\label{Theorem: Bilinear Bochner-Riesz with Fourier transform away from origin}
Let $1\leq p_1, p_2 \leq \infty$  and $p \geq 1$ with $1/p = 1/p_1 +1/p_2$. Then $\mathcal{B}^{\alpha}$ is bounded from $L^{p_1}(\mathbb{R}^d) \times L^{p_2}(\mathbb{R}^d)$ to $L^{p}(\mathbb{R}^d)$ if $p_1, p_2, p$ and $\alpha> \alpha(p_1, p_2)$ satisfy one of the following conditions:
\begin{enumerate}
    \item (Region III) $2 \leq p_2 \leq \infty$, $1\leq p_1, p \leq 2$ and $\alpha(p_1, p_2)= d(\frac{1}{2}-\frac{1}{p_2})-(1-\frac{1}{p})$, if $\supp \mathcal{F}_2 g(z', \cdot) \subseteq \{|\lambda_2| \geq \kappa_2 \}$ for some $\kappa_2>0$ and every $ z' \in \mathbb{R}^{d_1}$.
    \item (Region III) $2 \leq p_1 \leq \infty$, $1\leq p_2, p \leq 2$ and $\alpha(p_1, p_2)= d(\frac{1}{2}-\frac{1}{p_1})-(1-\frac{1}{p})$, if $\supp \mathcal{F}_2 f(y', \cdot) \subseteq \{|\lambda_1| \geq \kappa_1 \}$ for some $\kappa_1>0$ and every $y' \in \mathbb{R}^{d_1}$.
    \item (Region IV) $1\leq p_1 \leq 2 \leq p_2 \leq \infty$, $0<p\leq 1$ and $\alpha(p_1, p_2)= d(\frac{1}{p_1}-\frac{1}{2})$, if $\supp \mathcal{F}_2 g(z', \cdot) \subseteq \{|\lambda_2| \geq \kappa_2 \}$ for some $\kappa_2>0$ and every $z' \in \mathbb{R}^{d_1}$.
    \item (Region IV) $1\leq p_2 \leq 2 \leq p_1 \leq \infty$, $0<p \leq 1$ and $\alpha(p_1, p_2)= d(\frac{1}{p_2}-\frac{1}{2})$, if $\supp \mathcal{F}_2 f(y', \cdot) \subseteq \{|\lambda_1| \geq \kappa_1 \}$ for some $\kappa_1>0$ and every $y' \in \mathbb{R}^{d_1}$.
    \item (Region V) $1\leq p_1, p_2 \leq 2$ and $\alpha(p_1, p_2)= d(\frac{1}{p}-1)$, if $\supp \mathcal{F}_2 f(y', \cdot) \subseteq \{|\lambda_1| \geq \kappa_1 \}$ and $\supp \mathcal{F}_2 g(z', \cdot) \subseteq \{|\lambda_2| \geq \kappa_2 \}$ for some $\kappa_1, \kappa_2>0$ and every $y', z' \in \mathbb{R}^{d_1}$.
\end{enumerate}
\end{theorem}

Our focus now shifted to the relevant mixed norm estimates. Note that from Theorem \ref{Bilinear Bochner-Riesz main theorem}, at the point $(p_1, p_2,p)=(1,\infty,1)$, the boundedness of $\mathcal{B}^{\alpha}$ holds for $\alpha>Q/2$. If one consider the mixed norm, then these results can be further improved. For $0<p,q< \infty$, let us define the mixed norm of a measurable function $h$ on $\mathbb{R}^d$  by setting
\begin{align*}
    \|h\|_{L_{x''}^{p} L_{x'}^{q}(\mathbb{R}^d)} &:= \Big(\int_{\mathbb{R}^{d_1}} \Big( \int_{\mathbb{R}^{d_2}} |h(x',x'')|^p \ dx'' \Big)^{q/p} \ dx' \Big)^{1/q} ,
\end{align*}
with obvious modification if one of $p, q$ is $\infty$.

The next result summarizes the key findings in this setting.
\begin{theorem}
\label{Theorem: Mixed norm estimate for first layer}
    If $\alpha>(d+1)/2$, then
    \begin{align*}
        \|\mathcal{B}^{\alpha}(f,g)\|_{L_{x''}^{2/3} L_{x'}^{1}} \leq C \|f\|_{L_{x''}^1 L_{x'}^1} \|g\|_{L_{x''}^2 L_{x'}^{\infty}} .
    \end{align*}
\end{theorem}

In order to establish our results, several technical challenges arise because of the underlying geometry of the Grushin operators. First, the absence of group structure which prevents us to use the translation invariance techniques, commonly employed in Euclidean analysis. Second, the ball volume (see (\ref{Estimate : Ball volume})) depends not only on radius but also on the center, which means uniform ball volume estimates are not available  and each case must be handled separately. Another complexity occur from  the sub-ellipticity nature of the Grushin operator, which leads to mismatch between the the homogeneous dimension $Q$ and the topological dimension $d$ of the underlying space $\mathbb{R}^d$. Consequently, if one tries to prove the boundedness of $\mathcal{B}^{\alpha}$ using the existing Euclidean techniques, the best one can expect is that $\mathcal{B}^{\alpha}$ is bounded from $L^{p_1} \times L^{p_2}$ to $L^p$ for $\alpha> \alpha(p_1, p_2)$,  where $\alpha(p_1, p_2)$ is expressed in term of the homogeneous dimension $Q$. However, to obtain boundedness of $\mathcal{B}^{\alpha}$ with $\alpha(p_1, p_2)$ defined in terms of $d$, requires a different approach. Therefore similarly as in \cite{Martini_Sikora_Grushin_Weighted_Plancherel_2012}, \cite{Martini_Muller_Sharp_Multiplier_Grushin_2014} one tries to use appropriate weight and corresponding \emph{weighted Plancherel estimates}. In Grushin setup, there two types of weights, in terms of $|x'|$ (first-layer weight) and in terms of $|x''|$ (second-layer weight). If one use first-layer weight, then this only give the threshold in terms of the topological dimension if $d_1 \geq d_2$, while second-layer weight behave well for weighted Plancherel estimate (see \cite{Martini_Muller_Sharp_Multiplier_Grushin_2014}), but they are not easy to work with always, since certain sub-elliptic estimates does not hold for these weights, but true for first-layer weights (see \cite{Niedorf_Bochner_Riesz_Grushin_2022}, \cite{Chen_Ouhabaz_Bochner_Riesz_Grushin_2016}, \cite{Hebisch_Spectral_Multiplier_Heisenberg_1993}). Therefore in the Banach region ($p>1$), we show that one can actually get the smoothness threshold in terms of the topological dimension, replacing the Euclidean dimension in Theorem \ref{Theorem: Euclidean bilinear Bochner-Riesz for Grafakos}. While for the non-Banach region we use \emph{weighted restriction type estimates} (with respect to first-layer), developed in \cite{Chen_Ouhabaz_Bochner_Riesz_Grushin_2016}, to achieve the required smoothness threshold expressed in terms of $d$ for $d_1 \geq d_2$. To establish the boundedness result of $\mathcal{B}^{\alpha}$, with Euclidean dimension $n$ in $\alpha(p_1, p_2)$ possibly replaced with $d$ for arbitrary $d_1, d_2$, we use ideas from \cite{Bagchi_Molla_Singh_Bilinear_Metivier_2025}, recently developed in case of M\'etivier groups and from \cite{Niedorf_Bochner_Riesz_Grushin_2022}, studied $p$-specific type boundedness result for the linear Bochner-Riesz operators.

The paper is structured as follows. Section \ref{section 2} introduces the necessary preliminaries related to the Grushin setting, including the integral estimates for the distance and weight functions. Next, we derive some kernel estimates corresponding to both linear and bilinear multipliers associated with the Grushin operator, which are presented in detail in Section \ref{section 3}. The proof of the main result, Theorem \ref{Bilinear Bochner-Riesz main theorem}, is presented in  Section \ref{Section: proof of banach case}. There we make a claim for some specific points. Sections \ref{section 5} through \ref{section 9} are dedicated to proving the claim for these specific points. Theorem \ref{Theorem: Bilinear Bochner-Riesz with Fourier transform away from origin} is proved in  Section \ref{Section: Proof of theorem for restricted f and g}. Finally, in Section \ref{section 11}, we establish mixed-norm estimates for bilinear Bochner-Riesz means and concludes with the proof of Theorem \ref{Theorem: Mixed norm estimate for first layer}.
   
Throughout the article we  use standard notation. We use letter $C$ to indicate a  positive constant independent of the main parameters, but may vary from line to line. While writing estimates, we shall use the notation $f \lesssim g$ to indicate $f \leq Cg$ for some $C > 0$, and whenever $f \lesssim g\lesssim f$ , we shall write $f \sim g$. We sometime write $f \lesssim_{\epsilon} g$ to denote $f \leq C g$ where the constant $C$ may depend on the implicit constant $\epsilon$. For a Lebesgue measurable subset $E$ of $\mathbb{R}^d$, we denote by $\chi_E$ the characteristic function of the set $E$. We also denote $A^c$ to be the complement of the set $A$. Let $\Bar{B}$ denote the closure of a ball $B$. For any function $G$ on $\mathbb{R}$, define $\delta_R G(\eta) = G(R \eta) $ for $R>0$. Let $\alpha_1, \alpha_2 \in \mathbb{N}^{d_1}$, then inequalities between multi-indices, $\alpha_1 \leq \alpha_2$ is understood componentwise.

\section{Preliminaries}\label{section 2}
The Grushin operator has been well studied in literature. Let us mention some preliminary details, more details can be found in \cite{Martini_Sikora_Grushin_Weighted_Plancherel_2012}, \cite{Martini_Muller_Sharp_Multiplier_Grushin_2014}, \cite{Chen_Ouhabaz_Bochner_Riesz_Grushin_2016}, \cite{Niedorf_Bochner_Riesz_Grushin_2022}, \cite{Bagchi_Garg_Grushin_2024}, and the references therein. 

Let $\varrho$ denote the control distance associated with the Grushin operator $\mathcal{L}$. Then for all $x,y \in \mathbb{R}^d$,
\begin{align}
\label{Distance equivalent expression}
    \varrho(x,y) \sim |x'-y'| + \left\{ \begin{array}{ll}
       \frac{|x''-y''|}{|x'|+|y'|}  & \ \text{if}\ |x''-y''|^{1/2} \leq |x'|+ |y'| \\
      |x''-y''|^{1/2}   & \ \text{if}\  |x''-y''|^{1/2} \geq |x'|+ |y'| .
    \end{array} \right.
\end{align}
Associated to the distance $\varrho$, we define $ B^{\varrho}(x,r) := \{y \in \mathbb{R}^d : \varrho(x,y) < r \}$, the $\varrho$-ball with center at $x \in \mathbb{R}^d$ and radius $r> 0$. We sometimes omit the superscript $\varrho$ from $B^{\varrho}(x,r)$ when there is no risk of confusion. If there is nothing in the superscript, we always mean the ball is taken with respect to the distance $\varrho$. Consequently, if $|B(x,r)|$ denotes the Lebesgue measure of $\varrho$-ball, then 
\begin{align}
\label{Estimate : Ball volume}
    |B(x,r)| \sim r^{d_1 + d_2} \max\{r, |x'|\}^{d_2}.
\end{align}
The above observation implies that for $\kappa \geq 0$, we have
\begin{align}
\label{Doubling property of balls}
    |B(x,\kappa r)| \leq C (1+\kappa)^Q |B(x,r)|.
\end{align}
Thus,  $\mathbb{R}^d$ equipped with the distance $\varrho$ and the standard Lebesgue measure, becomes a doubling metric measure space with homogeneous dimension $Q:= d_1 + 2d_2$. We call $d= d_1 + d_2$ to be the topological dimension of $\mathbb{R}^d$.

The following lemma, which will be very useful later in our proofs, can be proved in the same way as in \cite[Proposition 2.1]{Niedorf_Bochner_Riesz_Grushin_2022}.
\begin{lemma}
\label{Lemma: Decomposition of balls for small radius}
If $|x'| \leq K r$ for some $K>0$, then there exists a constant $C>0$ such that
\begin{align*}
    B((x', x''),r) \subseteq B^{|\cdot|}(x', r) \times B^{|\cdot|}(x'', C r^2) ,
\end{align*}    
\end{lemma}

We now consider few integral estimates involving the weight and distance functions, to be used frequently throughout the article.

\begin{lemma}
\label{Lemma: Integral of weight over ball}
Suppose $0\leq \gamma < d_1 $. Then for any $r>0$, we have
\begin{align*}
    \int_{B(a, r)} \frac{dx}{|x'|^{\gamma}} &\leq C r^{d_1+ d_2} \max\{4 r, |a'|\}^{d_2-\gamma} .
\end{align*}
\end{lemma}

\begin{proof}
Case-I : $|a'| > 4r$. Since $|x'-a'| \leq \varrho(x,a) \leq r$, one easily see that $|x'| \geq \frac{3}{4} |a'|$. Therefore, using \eqref{Estimate : Ball volume} yields
\begin{align*}
    \int_{B(a, r)} \frac{dx}{|x'|^{\gamma}} &\leq C |a'|^{-\gamma} |B(a, r)| \\
    %&\leq C |a'|^{-\gamma} r^{d_1 + d_2} |a'|^{d_2} \\
    &\leq C r^{d_1 + d_2} |a'|^{d_2 -\gamma} .
\end{align*}

Case-II : $|a'| \leq 4r$. In this case, using Lemma \ref{Lemma: Decomposition of balls for small radius} and for $0\leq \gamma <d_1$, we get
\begin{align*}
    \int_{B(a, r)} \frac{dx}{|x'|^{\gamma}} &\leq C \int_{|x'-a'|< r} \int_{|x''-a''| < Cr^2} \frac{dx' \, dx''}{|x'|^{\gamma}} \\
    &\leq C r^{d_1 + 2d_2 - \gamma} .
\end{align*}
 Combining both the estimate completes the proof of the lemma.
\end{proof}

\begin{lemma}
\label{Lemma: Integral of weight over ball in second layer}
Suppose $|a'| \leq K r$ for some $K>0$ and $0\leq \gamma < d_2 $. Then for any $r>0$,
\begin{align*}
    \int_{B(a, r)} \frac{dy}{|x''-y''|^{\gamma}} &\leq C \, r^{d_1+ 2(d_2-\gamma)} .
\end{align*}
\end{lemma}

\begin{proof}
Application of Lemma \ref{Lemma: Decomposition of balls for small radius} yields
\begin{align}
\label{Integration of weight with second layer}
    \int_{B(a, r)} \frac{dy}{|x''-y''|^{\gamma}} &\leq C \int_{|y'-a'|< r} dy' \int_{|y''-a''| < C r^2} \frac{dy''}{|x''-y''|^{\gamma}} .
\end{align}
Ley us divide the proof into following two cases.

Case I : $|x''-a''| > 2C r^2$. In this case using $|y''-a''|<C r^2$, we can see $|x''-y''|>\tfrac{1}{2}|x''-a''|$. Therefore we get
\begin{align*}
    \int_{|y''-a''| < C r^2} \frac{dy''}{|x''-y''|^{\gamma}} &\leq C |x''-a''|^{-\gamma} r^{2 d_2} \leq C r^{2(d_2-\gamma)} .
\end{align*}

Case II : $|x''-a''| \leq 2 C r^2$. In this case again $|y''-a''|<C r^2$ implies $|x''-y''| \leq 3 C r^2$. Using this fact and for $0\leq \gamma<d_2$ we see that
\begin{align*}
    \int_{|y''-a''| < C r^2} \frac{dy''}{|x''-y''|^{\gamma}} &\leq \int_{|x''-y''| < 3 C r^2} \frac{dy''}{|x''-y''|^{\gamma}} \leq C r^{2(d_2-\gamma)} .
\end{align*}

Combining case I, case II, and plugging them into \eqref{Integration of weight with second layer}, we can conclude the proof of the lemma.
\end{proof}

\begin{lemma}
\label{lemma: outside distance}
    Let $R> 0$. Then for any $N> Q$, where $Q= d_1 +2d_2$, we have
    \begin{align*}
        \int_{\varrho(x, y) \geq R} \frac{|f(y)| \, dy}{\big( 1 + \varrho(x, y)\big)^N} \leq C R^{-N + d_1 + d_2}\, \max\{R, |x'|\}^{d_2} \mathcal{M}f(x) ,
    \end{align*}
    where $\mathcal{M}$ denotes the Hardy-Littlewood maximal function. In particular, taking $f \equiv 1$ for $N>Q$, we get
    \begin{align*}
        \int_{\varrho(x, y) \geq R} \frac{dy}{\big( 1 + \varrho(x, y) \big)^N} \leq C R^{-N + d_1 + d_2}\, \max\{R, |x'|\}^{d_2}.
    \end{align*}
\end{lemma}
\begin{proof}
Decomposing the integral into annulus and using the fact \eqref{Doubling property of balls} we have
\begin{align*}
  \int_{\varrho(x, y) \geq R} \frac{|f(y)| \, dy}{\big( 1 + \varrho(x, y)\big)^N} &= \sum_{k = 0} ^\infty \int_{2^k R \leq \varrho(x, y) < 2^{k+ 1} R}  \frac{|f(y)| \, dy}{\big( 1 + \varrho(x, y)\big)^N} \\
  & \leq C \sum_{k = 0} ^\infty \frac{|B(x, 2^{k+1}R)|}{(2^k R)^N} \frac{1}{|B(x, 2^{k+1}R)|} \int_{\varrho(x, y) < 2^{k+ 1} R}  |f(y)| \, dy \\
  & \leq C R^{-N + d_1 + d_2} \max\{R, |x'|\}^{d_2} \mathcal{M}f(x) ,
\end{align*}
provided $N>Q$. 
\end{proof}

The following lemmas will be useful later in the proof of Theorem \ref{Theorem: Mixed norm estimate for first layer}. Let $\mathcal{M}^{|\cdot|_1}$ and $\mathcal{M}^{|\cdot|_2}$ denote the Hardy-Littlewood maximal function on $\mathbb{R}^{d_1}$ and $\mathbb{R}^{d_2}$ respectively. For $R>0$, we set $[R, |x'|] := R \max\{R, |x'|\}$.
\begin{lemma}
\label{Lemma: Integral of distance with second Hardy-Littlewood}
Let $R>0$. Then for any $\epsilon_1 >0$ and $N>d_2$,
\begin{align*}
    \int_{|x'-y'|>R} \int_{|x''-y''|>[R,|x'|]} \frac{|g(z)| \, dz'' \, dz'} {\big(1 + \varrho(x, z) \big)^N } &\leq C |B(x,1)| \int_{|x'-z'|>R} \frac{\mathcal{M}^{|\cdot|_2}(g(z', \cdot))(x'')}{(1 +  |x'-z'| )^{N-2d_2-\epsilon_1}} \, dz' .
\end{align*}
    
\end{lemma}

\begin{proof}
From \eqref{Distance equivalent expression} and for any $\epsilon_1>0$ we have
\begin{align*}
    & \int_{|x'-y'|>R} \int_{|x''-y''|>[R,|x'|]}  \frac{|g(z)| } {\big(1 + \varrho(x, z) \big)^N } \, dz'' \, dz' \\
    &\leq C \int_{|x'-y'|>R} \int_{|x''-y''|>[R,|x'|]} \frac{|g(z)| \, dz'' \, dz'}{(1 + |x'-z'| )^{N-d_2-\epsilon_1} (1+\frac{|x''-z''|}{|x'|+|z'|})^{d_2+\epsilon_1}}  \\
    &\leq C \int_{|x'-y'|>R} \frac{(|x'|+|z'|)^{d_2}}{(1 + |x'-z'| )^{N-d_2-\epsilon_1}}  \int_{\mathbb{R}^{d_2}} \frac{|g(z)|}{(1+|x''-z''|)^{d_2+\epsilon_1}} \, dz'' \, dz' \\
    &\leq C \int_{|x'-z'|>R} \frac{(|x'|+|z'|)^{d_2}}{(1 + |x'-z'| )^{N-d_2-\epsilon_1}} \mathcal{M}^{|\cdot|_2}(g(z', \cdot))(x'') \, dz' \\
    &\leq C \max\{1, |x'|\}^{d_2} \int_{|x'-z'|>R} \frac{\mathcal{M}^{|\cdot|_2}(g(z', \cdot))(x'')}{(1 +  |x'-z'| )^{N-2d_2-\epsilon_1}} \, dz' ,
\end{align*}
provided $N>d_2$.
    
\end{proof}

\begin{lemma}
\label{Lemma: Integral of weight for general h}
Let $R>0$. Then for any $N> d_2$ we have
\begin{align*}
    \int_{|x'-y'|\leq R} \int_{|x''-y''|>[R,|x'|]} \frac{|h(x)| \, dx'' \, dx'}{(1+\varrho(x,y))^N} &\leq C R^{-(N-2d_2)} |B(y,1)| \int_{|x'-y'|\leq R} \mathcal{M}^{|\cdot|_2}(h(x', \cdot))(y'') dx'.
\end{align*}
    
\end{lemma}

\begin{proof}
From \eqref{Distance equivalent expression}, similarly as in Lemma \ref{lemma: outside distance}, for $N>d_2$ we obtain
\begin{align*}
    & \int_{|x'-y'|\leq R} \int_{|x''-y''|>[R,|x'|]} \frac{|h(x)| }{(1+\varrho(x,y))^N}\, dx'' \, dx' \\
    &\leq C \int_{|x'-y'|\leq R} \Big( \int_{|x''-y''|>[R,|x'|]} \frac{|h(x)| }{(1+\frac{|x''-y''|}{|x'|+|y'|})^N}\, dx'' \Big)\, dx' \\
    &\leq C \int_{|x'-y'|\leq R} (|x'|+|y'|)^{d_2} \int_{|x''-y''|>[R,|x'|]/(|x'|+|y'|)} \frac{|h(x)|}{(1+|x''-y''|)^{N}}\, dx'' \, dx' \\
    &\leq C R^{-(N-d_2)} \int_{|x'-y'|\leq R} (|x'|+|y'|)^{d_2} \Big(\frac{|x'|+|y'|}{\max\{R,|x'|\}} \Big)^{N-d_2} \mathcal{M}^{|\cdot|_2}(h(x', \cdot))(y'') \, dx' \\
    &\leq C  R^{-(N-2d_2)} \max\{1,|y'|\}^{d_2} \int_{|x'-y'|\leq R} \mathcal{M}^{|\cdot|_2}(h(x', \cdot))(y'') \, dx' .
\end{align*}
This completes the proof of the lemma.    
\end{proof}

\section{Kernel estimates}\label{section 3}
In this section, we prove several kernel estimates for both linear and bilinear multipliers associated with the Grushin operators, which will be main ingredients in our proofs. Denote $T:= (-\Delta_{x''})^{1/2}$ for $x'' \in \mathbb{R}^{d_2}$. Note that $\mathcal{L}$ and $T$ commute strongly, so they admit a joint functional calculus on $L^2(\mathbb{R}^d)$ (see \cite{Martini_Sikora_Grushin_Weighted_Plancherel_2012}). This allows us to define the operator $G(\mathcal{L}, T)$ for every bounded Borel functions $G : \mathbb{R} \times \mathbb{R} \rightarrow \mathbb{C}$. Also let us set $T_1 = T \otimes I$ and $T_2 = I \otimes T$.

For $\mu \in \mathbb{N}^{d_1}$, the $\mu$-th Hermite function on $\mathbb{R}^{d_1}$ is given by $ \Phi_{\mu}(x') := \prod_{j=1}^{d_1} h_{\mu_j}(x'_j)$, where for $l \in \mathbb{N}$, the $l$-th Hermite function on $\mathbb{R}$ is given by $h_l(t) := (-1)^l (2^l l! \sqrt{\pi})^{-1/2} e^{t^2/2} \left(\frac{d}{dt} \right)^l (e^{-t^2}) $. Hermite functions are the eigenfunctions of the Hermite operator $H=-\Delta+|x|^2$ on $\mathbb{R}^{d_1}$ with eigenvalue $(2|\mu|+d_1)$. For $\lambda \neq 0$, we also define the scaled Hermite functions on $\mathbb{R}^{d_1}$ by $\Phi_{\mu}^{\lambda}(x') := |\lambda|^{d_1/4} \Phi_{\mu}(|\lambda|^{1/2}x')$. These functions form an orthonormal basis of $L^2(\mathbb{R}^{d_1})$ and serve as eigenfunctions of the scaled Hermite operator $H(\lambda):= (-\Delta_{x'} + |x'|^2 |\lambda|^2)$, that is, they satisfies
\begin{align}
\label{Eigen functions of scaled Hermite operator}
    H(\lambda) \Phi_{\mu}^{\lambda}(x') &= (2|\mu|+d_1)|\lambda| \Phi_{\mu}^{\lambda}(x') .
\end{align}
Recall that $[k]:=(2k+d_1)$. The projection $P_k^{\lambda}$ associated to the eigenvalue $[k]|\lambda|$ onto the eigenspace is given by $P_k^{\lambda}f(x') = \sum_{|\mu|=k} \langle f, \Phi_{\mu}^{\lambda} \rangle \Phi_{\mu}^{\lambda}(x')$.

Set $m(\eta_1, \eta_2)=(1-\eta_1-\eta_2)_{+}^{\alpha}$. Let $\mathcal{L}_1 = \mathcal{L} \otimes I$ and $\mathcal{L}_2 = I \otimes \mathcal{L}$. Then the operators $\mathcal{L}_1$ and $\mathcal{L}_2$ commute strongly (see \cite[Lemma 7.24]{Konrad_Unbounded_Selfadjoint_operator_2012}). Thus bivariate spectral theorem (see \cite[Theorem 5.21]{Konrad_Unbounded_Selfadjoint_operator_2012}) allows us to define the operator
\begin{align*}
    & m(\mathcal{L}_1, \mathcal{L}_2)(f \otimes g)(x_1,x_2) \\
    &= \frac{1}{(2\pi)^{2 d_2}} \int_{\mathbb{R}^{d_2}} \int_{\mathbb{R}^{d_2}} e^{i (\lambda_1 \cdot x_1'' + \lambda_2 \cdot x_2'')} \sum_{k_1, k_2=0}^{\infty} m([k_1]|\lambda_1|, [k_2]|\lambda_2|) P_{k_1}^{\lambda_1}f^{\lambda_1}(x_1') P_{k_2}^{\lambda_2} g^{\lambda_2}(x_2') \ d\lambda_1 \ d\lambda_2 .
\end{align*}

For $f,g \in \mathcal{E}(\mathbb{R}^d)$, one can see that the operator is well-defined on $\mathbb{R}^d \times \mathbb{R}^d$. Moreover, by Lebesgue dominated convergence theorem, $m(\mathcal{L}_1, \mathcal{L}_2)(f \otimes g)(x_1,x_2)$ is in fact continuous on $\mathbb{R}^d \times \mathbb{R}^d$. In particular, taking $x_1=x_2=x$, we see that $m(\mathcal{L}_1, \mathcal{L}_2)(f \otimes g)(x,x)$ coincides with the bilinear Bochner-Riesz operator, that is
\begin{align}
\label{Bilinear Bochner-Riesz and bivariate spectral equality}
    \mathcal{B}^{\alpha}(f,g)(x) &= m(\mathcal{L}_1, \mathcal{L}_2)(f \otimes g)(x,x) .
\end{align}

It is known  (see \cite[Proposition 3]{Martini_Sikora_Grushin_Weighted_Plancherel_2012}) that the integral kernel of the operator $\exp(-t\mathcal{L})$ satisfies Gaussian type upper bound; that is, there exists $b>0$ such that for all $t>0$ and $x,y \in \mathbb{R}^d$, we have
\begin{align}
\label{Heat kernel bound for Grushin}
    |\mathcal{K}_{\exp(-t\mathcal{L})}(x,y)| &\leq C |B(x, t^{1/2})|^{-1} \exp\big(-b \tfrac{\varrho(x,y)^2}{t} \big) .
\end{align}
So that for $t_1, t_2>0$, the operator $\exp(-t_1\mathcal{L}_1-t_2 \mathcal{L}_2)$ has $L^{\infty}$-bilinear kernel given by
\begin{align*}
    \mathcal{K}_{\exp(-t_1\mathcal{L}_1-t_2 \mathcal{L}_2)}(x, y, z) &= \mathcal{K}_{\exp(-t_1\mathcal{L})}(x,y) \mathcal{K}_{\exp(-t_2\mathcal{L})}(x,z) .
\end{align*}
Set $G(\eta_{1},\eta_{2}) = m(\eta_{1},\eta_{2})\exp(\eta_{1}+\eta_{2}) $. An application of Fourier inversion formula implies
\begin{align*}
    m(\eta_{1},\eta_{2}) &= \frac{1}{4 \pi^2}\int_{\mathbb{R}^2} \widehat{G}(\tau_1, \tau_2) \exp((i \tau_1-1) \eta_{1}) \exp((i \tau_2-1) \eta_{2}) \ d\tau_1 \ d\tau_2 .
\end{align*}
Then from \eqref{Bilinear Bochner-Riesz and bivariate spectral equality}, for $f,g \in \mathcal{E}(\mathbb{R}^d)$ we can write
\begin{align}
\label{Bilinear Bochner-Riesz operator in terms of kernel}
    \mathcal{B}^{\alpha}(f,g)(x) &= \frac{1}{4 \pi^2}\int_{\mathbb{R}^2} \widehat{G}(\tau_1, \tau_2) \exp((i \tau_1-1) \mathcal{L})f(x) \exp((i \tau_2-1) \mathcal{L})g(x) \ d\tau_1 \ d\tau_2 \\
    &\nonumber=: \int_{\mathbb{R}^d} \int_{\mathbb{R}^d} \mathcal{K}^{\alpha}(x, y, z) f(y) g(z) \ dy \ dz ,
\end{align}
where
\begin{align}
\label{Kernel expression in terms of heat kernel}
    \mathcal{K}^{\alpha}(x, y, z) &= \frac{1}{4 \pi^2}\int_{\mathbb{R}^2} \widehat{G}(\tau_1, \tau_2) \mathcal{K}_{\exp((i \tau_1-1) \mathcal{L})}(x,y) \mathcal{K}_{\exp((i \tau_2-1) \mathcal{L})}(x,z) \,d\tau_1 \, d\tau_2 .
\end{align}
Using \cite[Lemma 4.1]{Duong_Ouhabaz_Sikora_Sharp_Multiplier_2002}, one can show that $\mathcal{K}^{\alpha}(x, \cdot, \cdot) \in L^2(\mathbb{R}^d \times \mathbb{R}^d)$ for almost every $x \in \mathbb{R}^d$.

Choose two non-negative, increasing sequence of Borel functions $\{\zeta_{n_1}\}_{n_1 \in \mathbb{N}}$ and $\{\zeta_{n_2}\}_{n_2 \in \mathbb{N}}$ on $\mathbb{R}$ such that they are compactly supported in $\mathbb{R} \setminus \{0\}$ and converging pointwise on $\mathbb{R} \setminus \{0\}$ to the constant function $1$. Set
\begin{align}
\label{Introducing cutoff in the kernel expression}
    \mathcal{K}_{m_{n_1, n_2}(\mathcal{L}_1, \mathcal{L}_2, T_1, T_2)}(x, \cdot, \cdot) &:= \zeta_{n_1}(T_1) \zeta_{n_2}(T_2)(\mathcal{K}_{m(\mathcal{L}_1, \mathcal{L}_2)} (x, \cdot, \cdot)) ,
\end{align}
where $m(\eta_1, \eta_2)=(1-\eta_1-\eta_2)_{+}^{\alpha}$.

Then $\mathcal{K}_{m_{n_1, n_2}(\mathcal{L}_1, \mathcal{L}_2, T_1, T_2)}(x, \cdot, \cdot) \to \mathcal{K}_{m(\mathcal{L}_1, \mathcal{L}_2)} (x, \cdot, \cdot)$ in $L^2(\mathbb{R}^d \times \mathbb{R}^d)$ for almost all $x \in \mathbb{R}^d$. Moreover as in \cite[Proposition 5]{Martini_Sikora_Grushin_Weighted_Plancherel_2012} we can write
\begin{align}
\label{Explicit Kernel expression}
    &\mathcal{K}_{m_{n_1, n_2}(\mathcal{L}_1, \mathcal{L}_2, T_1, T_2)}(x, y, z) = \frac{1}{(2\pi)^{2d_2}} \int_{\mathbb{R}^{d_2}} \int_{\mathbb{R}^{d_2}} \sum_{k_1, k_2 = 0} ^{\infty} m_{n_1, n_2}([k_1]|\lambda_1|, [k_2]|\lambda_2|, |\lambda_1|, |\lambda_2|) \\
    &\nonumber \hspace{1cm} \Big(\sum_{|\mu_1|=k_1} \Phi_{\mu_1}^{\lambda_1}(x') \Phi_{\mu_1}^{\lambda_1}(y') \Big) \Big(\sum_{|\mu_2|=k_2} \Phi_{\mu_2}^{\lambda_2}(x') \Phi_{\mu_2}^{\lambda_2}(z') \Big) e^{i(x''-y'') \cdot \lambda_1} e^{i(x''-z'') \cdot \lambda_2} \ d\lambda_1 \ d\lambda_2 ,
\end{align}
where $m_{n_1, n_2}([k_1]|\lambda_1|, [k_2]|\lambda_2|, |\lambda_1|, |\lambda_2|) = m([k_1]|\lambda_1|, [k_2]|\lambda_2| ) \zeta_{n_1}(|\lambda_1|) \zeta_{n_2}(|\lambda_2|) $.

Now we will discuss various weighted Plancherel type estimates. Let us define $|\mathbf{P}|$ to be the operator of multiplication by $|x'|$.
\begin{proposition}\cite[Proposition 10]{Martini_Sikora_Grushin_Weighted_Plancherel_2012}
\label{Weighted Plancherel Martini Sikora}
Let $0\leq \gamma <d_2/2$. Then for all bounded compactly supported Borel functions $F: \mathbb{R} \to \mathbb{C}$,
\begin{align*}
    \| |\mathbf{P}|^{\gamma} \mathcal{K}_{F(\mathcal{L})}(\cdot, y) \|_{L^2(\mathbb{R}^d)}^2 &\leq C_{\gamma} \int_{0}^{\infty} |F(\eta)|^2 \, \eta^{d/2} \min\{\eta^{d_2/2-\gamma}, |y'|^{2\gamma-d_2} \} \frac{d\eta}{\eta} ,
\end{align*}
for almost all $y \in \mathbb{R}^d$.
\end{proposition}

From the above proposition, one can immediately deduce the following result, see also \cite[Theorem 3.4]{Chen_Ouhabaz_Bochner_Riesz_Grushin_2016}.
\begin{proposition}
\label{Weighted restriction estimate}
Let $F$ be a Borel function supported in $[0,1]$. Then, for $0\leq \gamma< d_2/2$, we have
\begin{align*}
    \| |\mathbf{P}|^{\gamma} F(\mathcal{L}) f\|_{L^2(\mathbb{R}^d)} &\leq C_{\gamma} \|F\|_{L^2(\mathbb{R})} \|f\|_{L^1(\mathbb{R}^d)}.
\end{align*}
 Moreover, if $|a'|>4r$, then 
\begin{align*}
    \| |\mathbf{P}|^{\gamma} F(\mathcal{L}) (\chi_{B(a,r)}f)\|_{L^2(\mathbb{R}^d)} &\leq C_{\gamma} |a'|^{\gamma-d_2/2} \|F\|_{L^2(\mathbb{R})} \|f\|_{L^1(\mathbb{R}^d)} ,
\end{align*}
where $\chi_{B(a,r)}$ denotes the characteristic function of the ball $B(a,r)$ of $\mathbb{R}^d$.
\end{proposition}

\begin{proof}
An application of Minkowski's integral inequality yields
\begin{align*}
    \| |\mathbf{P}|^{\gamma} F(\mathcal{L}) f\|_{L^2(\mathbb{R}^d)} &\leq \int_{\mathbb{R}^d} |f(y)| \Big( \int_{\mathbb{R}^d} |x'|^{2\gamma} |\mathcal{K}_{F(\mathcal{L})}(x, y)|^2 \ dx \Big)^{1/2} dy .
\end{align*}
Therefore, using Proposition \ref{Weighted Plancherel Martini Sikora}, for all $\gamma \in [0, d_2/2)$ we get
\begin{align*}
    \| |\mathbf{P}|^{\gamma} F(\mathcal{L}) f\|_{L^2(\mathbb{R}^d)} &\leq C_{\gamma} \int_{\mathbb{R}^d} |f(y)| \ dy  \Big( \int_{0}^{1} |F(\eta)|^2 \eta^{d/2-1} \eta^{d_2/2-\gamma} \  d\eta \Big)^{1/2} \\
    & \leq C_{\gamma} \|F\|_{L^2(\mathbb{R})} \|f\|_{L^1(\mathbb{R}^d)}.
\end{align*}
Now, if $|a'|>4r$ and $y \in B(a,r)$, then $|y'| \geq \frac{3}{4}|a'|$. Hence, Proposition \ref{Weighted Plancherel Martini Sikora} gives
\begin{align*}
    \| |\mathbf{P}|^{\gamma} F(\mathcal{L}) \chi_{B(a,r)}f\|_{L^2(\mathbb{R}^d)} &\leq C_{\gamma} \int_{B(a,r)} |f(y)| \Big( \int_{0}^{1} |F(\eta)|^2 \eta^{d/2-1} |y'|^{2\gamma-d_2} \ d\eta \Big)^{1/2} \ dy \\
    & \leq C_{\gamma} |a'|^{\gamma-d_2/2} \|F\|_{L^2(\mathbb{R})} \|f\|_{L^1(\mathbb{R}^d)}.
\end{align*}
This completes the proof of the proposition.
\end{proof}

The following proposition is a generalization of \cite[Proposition 10]{Martini_Sikora_Grushin_Weighted_Plancherel_2012} to the bilinear setup.
\begin{proposition}
\label{Bilinear weighted Plancherel}
For all bounded Borel compactly supported functions $G: \mathbb{R}^2 \to \mathbb{C}$ we have
\begin{align*}
    & \|\mathcal{K}_{G(\mathcal{L}_1, \mathcal{L}_2)} (x, \cdot, \cdot)\|_{L^2(\mathbb{R}^d \times \mathbb{R}^d)}^2 \\
    &\leq C \int_{0}^{\infty} \int_{0}^{\infty} |G(\eta_1, \eta_2)|^2\, \eta_1^{d/2} \min\{\eta_1^{d_2/2}, |x'|^{-d_2} \}\, \eta_2^{d/2} \min\{\eta_2^{d_2/2}, |x'|^{-d_2} \} \ \frac{d\eta_1}{\eta_1} \ \frac{d\eta_2}{\eta_2} ,
\end{align*}
for almost all $x \in \mathbb{R}^d$.
\end{proposition}

\begin{proof}
Similarly as in \eqref{Introducing cutoff in the kernel expression}, if we define
\begin{align}
\label{Introducing cutoff in kernel in L2 estimate}
    \mathcal{K}_{G_{n_1, n_2}(\mathcal{L}_1, \mathcal{L}_2, T_1, T_2)}(x, \cdot, \cdot) &:= \zeta_{n_1}(T_1) \zeta_{n_2}(T_2)(\mathcal{K}_{G(\mathcal{L}_1, \mathcal{L}_2)} (x, \cdot, \cdot))
\end{align}
then $\mathcal{K}_{G_{n_1, n_2}(\mathcal{L}_1, \mathcal{L}_2, T_1, T_2)}(x, \cdot, \cdot) \to \mathcal{K}_{G(\mathcal{L}_1, \mathcal{L}_2)} (x, \cdot, \cdot)$ in $L^2(\mathbb{R}^d \times \mathbb{R}^d)$ for almost all $x \in \mathbb{R}^d$. Therefore because of Fatou's lemma, enough to prove the estimate for $\mathcal{K}_{G_{n_1, n_2}(\mathcal{L}_1, \mathcal{L}_2, T_1, T_2)}(x, \cdot, \cdot)$.

Recall that $\Phi_{\mu}^{\lambda}(x') = |\lambda|^{d_1/4} \Phi_{\mu}(|\lambda|^{1/2}x')$. Let us set $H_{k_i}(x') = \sum_{|\mu|=k} |\Phi_{\mu}(x')|^2$. Therefore from \eqref{Explicit Kernel expression} and using orthogonality, we obtain
\begin{align*}
    & \|\mathcal{K}_{G_{n_1, n_2}(\mathcal{L}_1, \mathcal{L}_2, T_1, T_2)}(x, \cdot, \cdot)\|_{L^2(\mathbb{R}^d \times \mathbb{R}^d)}^2 \\
    &= \frac{1}{(2\pi)^{2d_2}} \int_{\mathbb{R}^{d_2}} \int_{\mathbb{R}^{d_2}} \sum_{k_1, k_2 = 0}^{\infty} |G_{n_1, n_2}([k_1]|\lambda_1|, [k_2]|\lambda_2|, |\lambda_1|, |\lambda_2|)|^2 \prod_{i=1,2} \Big( \sum_{|\mu_i|=k_i} |\Phi_{\mu_i}^{\lambda_i}(x')|^2 \Big) \, d\lambda_1 \, d\lambda_2 \\
    % &= \frac{1}{(2\pi)^{2d_2}} \int_{\mathbb{R}^{d_2}} \int_{\mathbb{R}^{d_2}} \sum_{k_1, k_2 = 0}^{\infty} |G([k_1]|\lambda_1|, [k_2]|\lambda_2|)|^2 |\lambda_1|^{d_1/2} \sum_{|\mu_1|=k_1} |\Phi_{\mu_1}(|\lambda_1|^{1/2}x')|^2 \\
    % &\hspace{8cm} \times |\lambda_2|^{d_1/2} \sum_{|\mu_2|=k_2} |\Phi_{\mu_2}(|\lambda_2|^{1/2}x')|^2 \ d\lambda_1 \ d\lambda_2 \\
    &= \frac{1}{(2\pi)^{2d_2}} \int_{\mathbb{R}^{d_2}} \int_{\mathbb{R}^{d_2}} \sum_{k_1, k_2 = 0}^{\infty} |G_{n_1, n_2}([k_1]|\lambda_1|, [k_2]|\lambda_2|,|\lambda_1|, |\lambda_2|)|^2 \\
    &\hspace{8cm} \prod_{i=1,2} \Big( |\lambda_i|^{d_1/2} H_{k_i}(|\lambda_i|^{1/2}x') \Big) \ d\lambda_1 \ d\lambda_2 .
\end{align*} 
After a change of variables, the above expression can be dominated by
\begin{align*}
    & C \int_{0}^{\infty} \int_{0}^{\infty} \sum_{k_1, k_2 = 0}^{\infty} |G_{n_1, n_2}(\eta_1, \eta_2, \eta_1/[k_1], \eta_2/[k_2])|^2 \prod_{i=1,2} \left[\frac{\eta_i^{Q/2}}{[k_i]^{Q/2}} H_{k_i}\left(\frac{\eta_i^{1/2}x'}{[k_i]^{1/2}}\right) \right] \, \frac{d\eta_1}{\eta_1} \, \frac{d\eta_2}{\eta_2} .
\end{align*}
The rest of the argument follows the same lines as Proposition 10 in \cite{Martini_Sikora_Grushin_Weighted_Plancherel_2012}, with the aid of Lemma 9 also from \cite{Martini_Sikora_Grushin_Weighted_Plancherel_2012}.

\end{proof}

Suppose $\Theta : \mathbb{R} \to [0,1]$ be a compactly supported smooth function supported in $[1/2, 2]$ such that 
\begin{align}
\label{Definition of chi}
    \sum_{M \in \mathbb{Z}} \Theta_M(\tau) =1 \quad \text{with}\quad \Theta_M(\tau) = \Theta(2^{M} \tau).
\end{align}

Let $G : \mathbb{R} \times \mathbb{R} \to \mathbb{C}$ is a bounded Borel function supported in $ [0, 1] \times [0,1]$. Define the function $G_{M_1, M_2} : \mathbb{R} \times \mathbb{R} \times \mathbb{R} \times \mathbb{R} \to \mathbb{C}$ by setting
\begin{align}
\label{Introducing cutoff in bilinear multiplier}
    G_{M_1, M_2}(\eta_1, \eta_2, \tau_1, \tau_2) &:= G(\eta_1, \eta_2) \Theta_{M_1}(\tau_1) \Theta_{M_2}(\tau_2) .
\end{align}

For $s_1, s_2\geq 0$, let us define the Sobolev space of product type $L^2_{s_1,s_2}(\mathbb{R}^2)$ consists of all $G \in \mathcal{S}'(\mathbb{R}^2)$ such that
\begin{align*}
    \|G\|_{L^2_{s_1,s_2}(\mathbb{R}^2)} &:= \Big( \int_{\mathbb{R}^2} (1+|\xi_1|^2)^{s_1} (1+|\xi_2|^2)^{s_2} |\widehat{G}(\xi_1, \xi_2)|^2 \, d\xi_1 \, d\xi_2 \Big)^{1/2} < \infty .
\end{align*}

The following proposition can be seen as a bilinear analogue of Lemma 11 from \cite{Martini_Muller_Sharp_Multiplier_Grushin_2014}.
\begin{proposition}
\label{Prop: Truncated restriction with |x-y| power}
With the notation as in \eqref{Introducing cutoff in bilinear multiplier}, for all $N_1, N_2 \geq 0$ and almost all $x \in \mathbb{R}^d$ we have
\begin{align}
\label{second layer weighted Plancherel}
    & \int_{\mathbb{R}^d} \int_{\mathbb{R}^d} |x''-y''|^{2N_1} |x''-z''|^{2N_2} |\mathcal{K}_{G_{M_1, M_2}(\mathcal{L}_1, \mathcal{L}_2, T_1, T_2)}(x,y,z)|^2 \ dy \ dz \\
    &\nonumber \leq C\, 2^{M_1(2N_1- d_2)} 2^{M_2(2 N_2-d_2)} \|G\|_{L^{2}_{N_1,N_2}(\mathbb{R}^2)}^2 .
\end{align}
\end{proposition}

\begin{proof}
First note that enough to prove \eqref{second layer weighted Plancherel} for $(N_1, N_2)=(N,0)$ for any $N \geq 0$. Since for $(N_1, N_2)=(0, N)$, the proof is almost similar with obvious modification. Then combining these two estimates and an interpolation argument (see \cite[proof of Proposition 4.3]{Martini_Sharp_Multiplier_Kohn_Laplacian_2017}) gives the required estimates.

Therefore in order to prove \eqref{second layer weighted Plancherel} for $(N_1, N_2)=(N,0)$, again following the arguments of \cite[Lemma 11]{Martini_Muller_Sharp_Multiplier_Grushin_2014}, it suffices to prove
\begin{align*}
    \int_{\mathbb{R}^d} \int_{\mathbb{R}^d} | (x''-y'')^{\beta} \mathcal{K}_{G_{M_1, M_2}(\mathcal{L}_1, \mathcal{L}_2, T_1, T_2)}(x,y,z)|^2 \ dy \ dz &\leq C 2^{M_1(2|\beta|-d_2)} 2^{-M_2 d_2} \|G\|_{L^{2}_{|\beta|,0}(\mathbb{R}^2)}^2 ,
\end{align*}
for all $\beta \in \mathbb{N}^{d_2}$ with $|\beta|=N$.

For $\mu_1 \in \mathbb{N}^{d_1}$ and $\lambda_1 \in \mathbb{R}^{d_2} \setminus \{0\}$, define
\begin{align*}
    m_{M_1, M_2, [k_2],|\lambda_2|}(\mu_1, \lambda_1) &= G_{M_1, M_2}([|\mu_1|]|\lambda_1|, [k_2]|\lambda_2|, |\lambda_1|, |\lambda_2|) 
\end{align*}
and $m_{M_1, M_2, [k_2],|\lambda_2|}(\mu_1, \lambda_1)=0$ for $\mu_1 \in \mathbb{Z}^{d_1} \setminus \mathbb{N}^{d_1}$. Let $\Tilde{m}_{M_1, M_2, [k_2],|\lambda_2|}$ is a smooth extension of $m_{M_1, M_2, [k_2],|\lambda_2|}$ to $\mathbb{R}^{d_1} \times (\mathbb{R}^{d_2} \setminus \{0\})$. So that from \cite[Lemma 11]{Martini_Muller_Sharp_Multiplier_Grushin_2014} and (\ref{Explicit Kernel expression}), we get
\begin{align}
\label{Application of Lemma 11 of maritini Muller}
    & \int_{\mathbb{R}^d} \int_{\mathbb{R}^d} |(x''-y'')^{\beta} \mathcal{K}_{G_{M_1, M_2}(\mathcal{L}_1, \mathcal{L}_2, T_1, T_2)}(x,y,z)|^2 \ dz\ dy \\
    &\nonumber \leq C \int_{\mathbb{R}^{d_2}} \sum_{k_2 = 0}^{\infty} \sum_{\iota \in I_{\beta}} \int_{J_{\iota}} \int_{\mathbb{R}^{d_2}} \sum_{k_1 \geq k_{\iota}} \sum_{|\mu_1|=k_1} |\lambda_1|^{2|\beta^{\iota}| -2|\beta|} [k_1]^{2|\alpha^{\iota}|} |\partial_t^{\alpha^{\iota}} \partial_{\lambda_1}^{\beta^{\iota}} \Tilde{m}_{M_1, M_2, [k_2],|\lambda_2|}(\mu_1-s, \lambda_1)|^2 \\
    &\nonumber \hspace{7cm} \times \Phi_{\mu_1}^{\lambda_1}(x')^2 \  d\lambda_1 \ d\nu_{\iota}(s) \Big( \sum_{|\mu_2|=k_2} \Phi_{\mu_2}^{\lambda_2}(x')^2 \Big) \, d\lambda_2 .
\end{align}
where $\beta^{\iota} \in \mathbb{N}^{d_2}$, $\beta^{\iota} \leq \beta$, $\alpha^{\iota}, \Tilde{\alpha}^{\iota} \in \mathbb{N}^{d_1}$, $|\alpha^{\iota}|+|\beta^{\iota}| \leq |\beta|$, $\Tilde{\gamma}^{\iota} := (\Tilde{\gamma}^{\iota}_1, \ldots, \Tilde{\gamma}^{\iota}_{d_1}) $, $k_{\iota} := \Tilde{\gamma}^{\iota}_1 + \cdots + \Tilde{\gamma}^{\iota}_{d_1}$, $\Tilde{\gamma}^{\iota}_j:= \max\{0, \gamma^{\iota}_j + 2 r^{\iota}_j \} \geq 2(r^{\iota}_j - \Tilde{\alpha^{\iota}_j}+\alpha^{\iota}_j) $, $\gamma_j^{\iota}:= 2\max\{0,1-\rho_1^j, \ldots, 1-\rho_{u_j}^{j} \} \geq 2(\alpha_j^{\iota}-\Tilde{\alpha}_j^{\iota})$ for all $j \in \{1, \ldots, d_1\}$, $u_1+\cdots +u_{d_1} \leq |\beta|-|\beta^{\iota}|$, $J_{\iota}= \prod_{j=1}^{d_1}[2(r^{\iota}_j- \Tilde{\alpha^{\iota}_j}), 2(r^{\iota}_j - \Tilde{\alpha^{\iota}_j}+\alpha^{\iota}_j)] $, $r^{\iota} \in \mathbb{Z}^{d_1}$ with $|r^{\iota}|\leq |\beta|$, $\nu_{\iota}$ is a probability measure on $J_{\iota}$, and $I_{\beta}$ is a finite set.

Let $[t]= (2t_1+\cdots+2t_{d_1}+d_1)$ for all $t \in \mathbb{R}^{d_1}$. A smooth extension $\Tilde{m}_{M_1, M_2, [k_2],|\lambda_2|}$ of $m_{M_1, M_2, [k_2],|\lambda_2|}$ is given by
\begin{align*}
    \Tilde{m}_{M_1, M_2, [k_2],|\lambda_2|}(t, \lambda_1) = G([t]|\lambda_1|, [k_2]|\lambda_2|)\Theta(2^{M_1}|\lambda_1|)\Theta(2^{M_2}|\lambda_2|),
\end{align*}
for $\lambda \in \mathbb{R}^{d_2} \setminus \{0\}$ and $t \in (-1/2, \infty)^{d_1}$. Let us denote $G^{(v,0)}(\eta_1, \eta_2) = \partial^{v}_{\eta_1}G(\eta_1, \eta_2)$ for $v \in \mathbb{N}$. By induction, one can check that
\begin{align*}
    & \partial_t^{\alpha^{\iota}} \partial_{\lambda}^{\beta^{\iota}} \Tilde{m}_{M_1, M_2, [k_2],|\lambda_2|}(t, \lambda_1) \\
    &= \sum_{\substack{0\leq a \leq |\alpha^{\iota}| \\ 0\leq b \leq |\beta^{\iota}|}} 2^{M_1(|\beta^{\iota}|-b)} \Theta^{(|\beta^{\iota}|-b)}(2^{M_1} |\lambda_1|) \Psi_{\alpha^{\iota},a,b}([t]) |\lambda_1|^a G^{(a+b,0)}([t]|\lambda_1|, [k_2]|\lambda_2|) \Theta(2^{M_2}|\lambda_2|)
\end{align*}
for all $t \in [0, \infty)^{d_1}$, where $\Psi_{\alpha^{\iota},a,b} : \mathbb{R} \to \mathbb{C}$ are smooth functions, homogeneous of degree $a+b-|\alpha^{\iota}|$. Therefore, using the fact $|\lambda_1| [t] \leq 1$ on the support of $G$ and $|\lambda_1| \sim 2^{-M_1}$ on the support of $\theta$, we obtain
\begin{align*}
    & |\partial_{\lambda}^{\beta^{\iota}} \partial_t^{\alpha^{\iota}} \Tilde{m}_{M_1, M_2, [k_2],|\lambda_2|}(t, \lambda_1)|^2 \\
    &\hspace{2cm} \leq C \sum_{v=0}^{|\alpha^{\iota}|+|\beta^{\iota}|} 2^{2{M_1}(|\beta^{\iota}|-|\alpha^{\iota}|)} |G^{(v,0)}([t]|\lambda_1|, [k_2]|\lambda_2|)|^2 \chi(2^{M_1} |\lambda_1|) \Theta(2^{M_2}|\lambda_2|) ,
\end{align*}
for all $t \in [0, \infty)^{d_1}$, where $\chi$ is a characteristic function of $[1/2,2]$.

Note that $\sum_{|\mu|=k} \Phi_{\mu}^{\lambda}(x')^2 \leq C |\lambda|^{d_1/2} [k]^{d_1/2-1}$ for all $x \in \mathbb{R}^{d_1}$ (see \cite[Lemma 3.3]{Niedorf_Bochner_Riesz_Grushin_2022}). Let $\Tilde{J}_{\iota}$ be the image of $J_{\iota}$ under the map $(s_1, \ldots, s_{d_1}) \mapsto (s_1+\cdots+s_{d_1})$, and $\Tilde{\nu}_{\iota}$ is the corresponding push-forward of $\nu_{\iota}$ on $\Tilde{J}_{\iota}$, then $k_{\iota} \geq \max \Tilde{J}_{\iota}$ and   \eqref{Application of Lemma 11 of maritini Muller} becomes
\begin{align*}
    & C \int_{\mathbb{R}^{d_2}} \sum_{k_2 = 0}^{\infty} \sum_{\iota \in I_{\beta}} \sum_{v=0}^{|\alpha^{\iota}|+|\beta^{\iota}|} \sum_{k_1 \geq k_{\iota}} 2^{2M_1(|\beta^{\iota}|-|\alpha^{\iota}|)} \int_{\Tilde{J}_{\iota}} \int_{\mathbb{R}^{d_2}} |\lambda_1|^{2|\beta^{\iota}|-2|\alpha^{\iota}| -2|\beta|} ([k_1]|\lambda_1|)^{2|\alpha^{\iota}|} |\chi(2^{M_1} |\lambda_1|)|^2 \\
    &\times |\Theta(2^{M_2} |\lambda_2|)|^2 |G^{(v,0)}([k_1-s]|\lambda_1|, [k_2]|\lambda_2|)|^2 |\lambda_1|^{d_1/2} [k_1]^{d_1/2-1} d\lambda_1 \ d\Tilde{\nu}_{\iota}(s) |\lambda_2|^{d_1/2} [k_2]^{d_1/2-1} d\lambda_2 \\ 
    &\leq C \int_{\mathbb{R}^{d_2}} \sum_{k_2 = 0}^{\infty} \sum_{\iota \in I_{\beta}} \sum_{v=0}^{|\alpha^{\iota}|+|\beta^{\iota}|} \sum_{k_1 \geq k_{\iota}} 2^{2M_1|\beta|} \int_{\Tilde{J}_{\iota}} \int_{\mathbb{R}^{d_2}} |G^{(v,0)}([k_1-s]|\lambda_1|, [k_2]|\lambda_2|)|^2 |\chi(2^{M_1} |\lambda_1|)|^2  \\
    & \hspace{8cm} |\Theta(2^{M_2} |\lambda_2|)|^2 |\lambda_1| |\lambda_2| \  d\lambda_1 \ d\Tilde{\nu}_{\iota}(s)\, d\lambda_2 .
\end{align*}   
Now changing into polar co-ordinates and making change of variables, the above expression can be bounded by
\begin{align*}
    & C \sum_{k_2 = 0}^{\infty} \sum_{\iota \in I_{\beta}} \sum_{v=0}^{|\alpha^{\iota}|+|\beta^{\iota}|} \sum_{k_1 \geq k_{\iota}} 2^{2M_1|\beta|} 2^{-M_1(d_2+1)} 2^{-M_2(d_2+1)}  \int_{\Tilde{J}_{\iota}} \int_{0}^{\infty} \int_{0}^{\infty} |G^{(v,0)}(r_1, r_2)|^2 \\
    & \hspace{6cm} |\chi(2^{M_1} [k_1-s]^{-1} r_1)|^2 |\Theta(2^{M_2} [k_2]^{-1} r_2)|^2 \  \frac{dr_1}{r_1} \frac{dr_2}{r_2} \ d\Tilde{\nu}_{\iota}(s) \\
    &\leq C \sum_{\iota \in I_{\beta}} \sum_{v=0}^{|\alpha^{\iota}|+|\beta^{\iota}|} 2^{2M_1|\beta|} 2^{-M_1(d_2+1)} 2^{-M_2(d_2+1)}  \int_{\Tilde{J}_{\iota}} \int_{0}^{\infty} \int_{0}^{\infty} \sum_{[k_1-s] \leq 2 \cdot 2^{M_1} r_1} \sum_{[k_2] \leq 2 \cdot 2^{M_2} r_2} \\
    & \hspace{8cm} |G^{(v,0)}(r_1, r_2)|^2 \, \frac{dr_1}{r_1} \frac{dr_2}{r_2} \ d\Tilde{\nu}_{\iota}(s) \\
    &\leq C 2^{M_1(2|\beta|-d_2)} 2^{-M_2 d_2} \|G\|_{L^{2}_{|\beta|,0}(\mathbb{R}^2)}^2 .
\end{align*}
This completes the proof of the proposition.
\end{proof}

From the above proposition, we immediately obtain the following proposition.
\begin{proposition}
\label{prop: second layer weighted Plancherel}
Let $G : \mathbb{R} \times \mathbb{R} \to \mathbb{C}$ is a bounded Borel function supported in $ [0, 1] \times [0,1]$. Then for $0\leq \gamma_1, \gamma_2<d_2/2$ and almost all $x \in \mathbb{R}^d$ we have
\begin{align*}
    \left(\int_{\mathbb{R}^d} \int_{\mathbb{R}^d} |x''-y''|^{2\gamma_1} |x''-z''|^{2\gamma_2} |\mathcal{K}_{G(\mathcal{L}_1, \mathcal{L}_2)}(x,y,z)|^2 \ dy \ dz \right)^{1/2} &\leq C\, \|G\|_{L^{2}_{\gamma_1,\gamma_2}(\mathbb{R}^2)} .
\end{align*}

\end{proposition}

\begin{proof}
Similar to Proposition \ref{Bilinear weighted Plancherel}, enough to prove the estimate for $\mathcal{K}_{G_{n_1, n_2}(\mathcal{L}_1, \mathcal{L}_2, T_1, T_2)}(x, y, z)$ (see \eqref{Introducing cutoff in kernel in L2 estimate}) replacing $\mathcal{K}_{G(\mathcal{L}_1, \mathcal{L}_2)}(x,y,z)$. Using \eqref{Definition of chi} we can write
\begin{align*}
    G_{n_1, n_2}(\eta_1, \eta_2, \tau_1, \tau_2) &= \sum_{M_1=0}^{\infty} \sum_{M_2=0}^{\infty} (G_{n_1, n_2})_{M_1, M_2}(\eta_1, \eta_2, \tau_1, \tau_2)
\end{align*}
where $(G_{n_1, n_2})_{M_1, M_2}(\eta_1, \eta_2, \tau_1, \tau_2) = G_{n_1, n_2}(\eta_1, \eta_2, \tau_1, \tau_2) \Theta_{M_1}(\tau_1) \Theta_{M_2}(\tau_2)$. Therefore applying Proposition \ref{Prop: Truncated restriction with |x-y| power} for $0\leq \gamma_1, \gamma_2<d_2/2$ we obtain
\begin{align*}
    & \left(\int_{\mathbb{R}^d} \int_{\mathbb{R}^d} |x''-y''|^{2\gamma_1} |x''-z''|^{2\gamma_2} |\mathcal{K}_{G_{n_1, n_2}(\mathcal{L}_1, \mathcal{L}_2, T_1, T_2)}(x,y,z)|^2 \ dy \ dz \right)^{1/2} \\
    &\leq \sum_{M_1=0}^{\infty} \sum_{M_2=0}^{\infty} \left( \int_{\mathbb{R}^d} \int_{\mathbb{R}^d} |x''-y''|^{2\gamma_1} |x''-z''|^{2\gamma_2} |\mathcal{K}_{(G_{n_1, n_2})_{M_1, M_2}(\mathcal{L}_1, \mathcal{L}_2, T_1, T_2)}(x,y,z)|^2 \ dy \ dz \right)^{1/2} \\
    &\leq C \sum_{M_1=0}^{\infty} \sum_{M_2=0}^{\infty} 2^{M_1(\gamma_1- d_2/2)} 2^{M_2(\gamma_2-d_2/2)} \|G\|_{L^{2}_{\gamma_1,\gamma_2}(\mathbb{R}^2)} \\
    &\leq C \|G\|_{L^{2}_{\gamma_1,\gamma_2}(\mathbb{R}^2)} .
\end{align*}
    
\end{proof}

Let $F : \mathbb{R} \to \mathbb{C}$ be a bounded, Borel function supported in $ [0, 1]$. With the help of $\Theta$, as defined in \eqref{Definition of chi} and for $M \in \mathbb{N}$, let $F_M : \mathbb{R} \times \mathbb{R} \to \mathbb{C}$ be given by 
\begin{align*}
    F_M(\eta, \tau) = F(\eta) \Theta_{M}(\tau) .
\end{align*}

% It is important to note that the cutoff used in the above definition, as well as in \eqref{Introducing cutoff in bilinear multiplier},  differs from that in \cite{Martini_Muller_Sharp_Multiplier_Grushin_2014} and \cite{Niedorf_Bochner_Riesz_Grushin_2022}. In those work, the cutoff is applied along the $\eta/\tau$ variable, whereas in our case, it is applied in the $\tau$ variable. Additionally, note that in our setting, the support of $F$ is not necessarily away from the origin.

\begin{proposition}
\label{Lemma: Martini_Mullar_Weighted_Plancherel}
Let $F : \mathbb{R} \to \mathbb{C}$ be a bounded Borel function supported in $ [0, 1]$ and $F_M$ be defined as above. Then for all $ N \geq 0$ and almost all $y \in \mathbb{R}^d$,
\begin{align*}
    \left( \int_{\mathbb{R}^d}  |x''-y''|^{2N} |\mathcal{K}_{F_{M}(\mathcal{L}, T)}(x,y)|^2 \ dx \right)^{1/2} &\leq C 2^{M(N-d_2/2)} \|F\|_{L^{2}_N(\mathbb{R})} .
\end{align*}
Moreover, we also have
\begin{align*}
    \| F_M(\mathcal{L}, T) f\|_{L^2 } \leq C 2^{-M d_2/2} \|F\|_{L^2(\mathbb{R})} \, \|f\|_{L^1} .
\end{align*} 
\end{proposition}

The first part of the above result follows from a similar argument as in Proposition \ref{Prop: Truncated restriction with |x-y| power}, and is even more simpler in this case. While the proof of second estimate follows from an application of Minkowski's integral inequality (see Proposition \ref{Weighted restriction estimate}). One can also compare this results with \cite[Lemma 11]{Martini_Muller_Sharp_Multiplier_Grushin_2014}, \cite[Theorem 3.4]{Niedorf_Bochner_Riesz_Grushin_2022}.

Similar to Proposition \ref{prop: second layer weighted Plancherel}, as a consequence of the above proposition, we have the following result.
\begin{proposition}
\label{Lemma: Martini_Mullar_Weighted_Plancherel without cutoff}
Let $F : \mathbb{R} \to \mathbb{C}$ be a bounded Borel function supported in $ [0, 1]$. Then for all $0 \leq \gamma < d_2/2$ and almost all $y \in \mathbb{R}^d$,
\begin{align*}
    \left( \int_{\mathbb{R}^d}  |x''-y''|^{2\gamma} |\mathcal{K}_{F(\mathcal{L})}(x,y)|^2 \ dx \right)^{1/2} &\leq C \|F\|_{L^{2}_{\gamma}(\mathbb{R})} .
\end{align*}
\end{proposition}

\section{Proof of Theorem \ref{Bilinear Bochner-Riesz main theorem}}
\label{Section: proof of banach case}  
Here, we  present the proof of Theorem \ref{Bilinear Bochner-Riesz main theorem}, which is divided into few steps.

\subsection{Step I: Decomposition of \texorpdfstring{$\mathcal{B}^{\alpha}$.}{}}
\label{subsection: step 1 decomposition of Bj1 case}
Let $f, g \in \mathcal{E}(\mathbb{R}^d)$. We consider a dyadic decomposition of the bilinear Bochner-Riesz multiplier $(1-\eta_1-\eta_2 )_{+}^{\alpha}$:
\begin{align}
\label{Dyadic decomposition of bilinear multiplier}
    (1-\eta_1-\eta_2 )_{+}^{\alpha} &= \sum_{j\in \mathbb{Z}} (1-\eta_1-\eta_2 )_+ ^{\alpha} \varphi \big(2^j (1 - \eta_1 - \eta_2)\big),
\end{align}
where  $\varphi \in C_c  ^{\infty}(\frac{1}{2}, 2)$ is a non-negative function such that  for all $t >0$, $\sum_{j \in \mathbb{Z}} \varphi(2^j t) = 1$ holds. Then setting $\varphi_j^{\alpha} (\eta_1, \eta_2) :=  (1-\eta_1-\eta_2 )_+ ^{\alpha} \varphi \big(2^j (1 - \eta_1 - \eta_2)\big)$ and using (\ref{Dyadic decomposition of bilinear multiplier}), we express the bilinear Bochner-Riesz operator $\mathcal{B}^{\alpha}$ as follows
\begin{align*}
    \mathcal{B}^{\alpha} = \sum_{j=0}^{\infty} \mathcal{B}_j^{\alpha} ,
\end{align*}
where
\begin{align}
\label{Bilinear Bochner-Riesz operator after dyadic decomposition}
    \mathcal{B}^{\alpha}_j (f,g)(x) &= \frac{1}{(2\pi)^{2 d_2}} \int_{\mathbb{R}^{d_2}} \int_{\mathbb{R}^{d_2}} e^{i(\lambda_1 + \lambda_2)\cdot x''} \sum_{k_1, k_2 = 0}^{\infty} \varphi^{\alpha}_j ([k_1]|\lambda_1|, [k_2]|\lambda_2|) \\
    &\nonumber \hspace{6cm} P^{\lambda_1} _{k_1} f^{\lambda_1}(x') P^{\lambda_2} _{k_2} g^{\lambda_2}(x')\, d\lambda_1 d\lambda_2.
\end{align}

Let $\mathcal{K}_j^{\alpha}$ denote the kernel of the bilinear operator $\mathcal{B}_j^{\alpha}$ (see \eqref{Bilinear Bochner-Riesz operator in terms of kernel}). Fix $\varepsilon>0$ and further decompose the kernel as
\begin{align}
\label{Expression: Kernel expression for Kj1}
   & \mathcal{K}^{\alpha}_{j}(x, y, z) \\
   & = \mathcal{K}^{\alpha}_j (x, y, z) \chi_{B (x, 2^{j(1+\varepsilon)})} (y) \chi_{B (x, 2^{j(1+\varepsilon)})} (z) + \mathcal{K}^{\alpha}_j (x, y, z) \chi_{B (x, 2^{j(1+\varepsilon)})} (y) \chi_{B (x, 2^{j(1+\varepsilon)})^c } (z)\nonumber\\
   & + \mathcal{K}^{\alpha}_j (x, y, z) \chi_{B(x, 2^{j(1+\varepsilon)})^c}(y)  \chi_{B(x, 2^{j(1+\varepsilon)})} (z) +\mathcal{K}^{\alpha}_j (x, y, z) \chi_{B(x, 2^{j(1+\varepsilon)})^c}(y) \chi_{B(x, 2^{j(1+\varepsilon)})^c}(z)\nonumber\\
   & =: \mathcal{K}^{\alpha}_{j,1} + \mathcal{K}^{\alpha}_{j,2} + \mathcal{K}^{\alpha}_{j,3} + \mathcal{K}^{\alpha}_{j,4}.\nonumber
\end{align}

For $l=1, 2, 3, 4$, let $\mathcal{B}^{\alpha}_{j,l}$ denote the bilinear operator whose kernel is $\mathcal{K}^{\alpha}_{j,l}$. Consequently, to prove Theorem \ref{Bilinear Bochner-Riesz main theorem}, it is enough to show that whenever $\alpha>\alpha(p_1,p_2)$, there exists a $\delta>0$ such that for each $l=1, 2, 3, 4$, the following estimates holds:
\begin{align}\label{equation of estimate of Balpha j,l}
    \|\mathcal{B}^{\alpha}_{j,l}(f,g)\|_{L^p} &\leq C 2^{-j \delta} \|f\|_{L^{p_1}} \|g\|_{L^{p_2}} ,
\end{align}
with $1/p=1/p_1+1/p_2$ and $1\leq p_1, p_2 \leq \infty$.

\subsection{Step II: Pointwise estimate of \texorpdfstring{$\mathcal{K}^{\alpha}_{j}$}{}}
\label{Step 2, pointwise kernel estimate}
Let us denote the kernel corresponding to the operator $\mathcal{B}^{\alpha}_{j}$ by $\mathcal{ K}^{\alpha}_{j}$. Then we have the following estimate.
\begin{lemma}
\label{Lemma: Pointwise kernel estimate for Bj}
For any $\beta_1, \beta_2 \geq 0$ and $\epsilon>0$, we have
\begin{enumerate}
    \item $\left|\mathcal{K}_j^{\alpha}(x,y,z) (1+\varrho(x,y))^{\beta_1} (1+\varrho(x,z))^{\beta_2} \right| \leq C |B(x, 1)|^{-2} 2^{j(\beta_1 +\beta_2 +1/2+\epsilon)} ;$
    \item $\left|\mathcal{K}_j^{\alpha}(x,y,z) (1+\varrho(x,y))^{\beta_1} (1+\varrho(x,z))^{\beta_2} \right| \leq C |B(x, 1)|^{-1} |B(z, ,1)|^{-1} 2^{j(\beta_1 +\beta_2 +1/2+\epsilon)} ;$
    \item $\left|\mathcal{K}_j^{\alpha}(x,y,z) (1+\varrho(x,y))^{\beta_1} (1+\varrho(x,z))^{\beta_2} \right| \leq C |B(x, 1)|^{-1} |B(y, 1)|^{-1} 2^{j(\beta_1 +\beta_2 +1/2+\epsilon)} ;$
    \item $\left|\mathcal{K}_j^{\alpha}(x,y,z) (1+\varrho(x,y))^{\beta_1} (1+\varrho(x,z))^{\beta_2} \right| \leq C |B(y, 1)|^{-1} |B(z, 1)|^{-1} 2^{j(\beta_1 +\beta_2 +1/2+\epsilon)} .$
\end{enumerate}
\end{lemma}

\begin{proof}
In the following, we only prove the estimate $\emph{(1)}$, as the remaining estimates follow from the similar arguments with obvious modification.

Set $F(\eta_{1},\eta_{2}) = \exp(\eta_{1}+\eta_{2}) \varphi_j^{\alpha}(\eta_{1},\eta_{2})$. Then from \eqref{Kernel expression in terms of heat kernel} we can write
\begin{align*}
    \mathcal{K}_{j}^{\alpha}(x,y,z) &= \frac{1}{4 \pi^2}\int_{\mathbb{R}^2} \widehat{F}(\tau_1, \tau_2) \mathcal{K}_{\exp((i \tau_1-1) \mathcal{L})}(x,y) \mathcal{K}_{\exp((i \tau_2-1) \mathcal{L})}(x,z) \,d\tau_1 \,d\tau_2 .
\end{align*}
Consequently, we obtain the estimate
\begin{align}
\label{Introducing weights in pointwise kernel estimate}
    & \left| \mathcal{K}_j^{\alpha}(x, y, z) (1+ \varrho(x,y))^{\beta_1} (1+ \varrho(x,z))^{\beta_2} \right| \leq C \int_{\mathbb{R}^2} |\widehat{F}_1(\tau_1, \tau_2)| |\mathcal{K}_{\exp((i \tau_1-1) \mathcal{L}_1)}(x,y)| \\
    &\nonumber \hspace{4cm} \times |\mathcal{K}_{\exp((i \tau_2-1) \mathcal{L}_2)}(x,z)| (1+ \varrho(x,y))^{\beta_1} (1+ \varrho(x,z))^{\beta_2} \ d\tau_1 \ d\tau_2 .
\end{align}
Using the fact \eqref{Heat kernel bound for Grushin} and \eqref{Doubling property of balls}, from \cite[Theorem 7.3]{Ouhabaz_Analysis_heat_equation_domain_2005}, we can conclude that
\begin{align*}
    |\mathcal{K}_{\exp((i \tau_1-1) \mathcal{L})}(x,y)| &\leq C |B(x, 1)|^{-1} \exp{\big(-b \tfrac{\varrho(x,y)^2}{(1+\tau_1^2)} \big)} .
\end{align*}
Similarly for $|\mathcal{K}_{\exp((i \tau_2-1) \mathcal{L})}(x,z)|$. Therefore, with the help of the above estimate, we obtain
\begin{align}
\label{Pointwise heat distance estimate}
    &|\mathcal{K}_{\exp((i \tau_1-1) \mathcal{L}_1)}(x,y) \mathcal{K}_{\exp((i \tau_2-1) \mathcal{L}_2)}(x,z)| (1+ \varrho(x,y))^{\beta_1} (1+ \varrho(x,z))^{\beta_2} \\
    &\nonumber \leq C |B(x, 1)|^{-2} (1+|\tau_1|)^{\beta_1} (1+|\tau_2|)^{\beta_2} .
\end{align}
Finally, plugging the estimate \eqref{Pointwise heat distance estimate} into \eqref{Introducing weights in pointwise kernel estimate}, and an application of H\"older's inequality yields
\begin{align*}
    &\left| \mathcal{K}_j^{\alpha}(x, y, z) (1+ \varrho(x,y))^{\beta_1} (1+ \varrho(x,z))^{\beta_2} \right| \\
    &\leq C |B(x, 1)|^{-2} \int_{\mathbb{R}^2} |\widehat{F_1}(\tau_1, \tau_2)| (1+|\tau_1|)^{\beta_1} (1+|\tau_2|)^{\beta_2} \ d\tau_1 \ d\tau_2 \\
    &\leq C |B(x, 1)|^{-2} \Big(\int_{\mathbb{R}^2} |\widehat{F}(\tau_1, \tau_2)|^2 (1+|\tau_1|^2+|\tau_2|^2)^{\beta_1+\beta_2+\frac{2+2\epsilon}{2}} d\tau_1 d\tau_2 \Big)^{\frac{1}{2}} \\
    & \hspace{8cm} \Big(\int_{\mathbb{R}^2} \frac{\ d\tau_1 d\tau_2}{(1+|\tau_1|^2+|\tau_2|^2)^{\frac{2+2\epsilon}{2}}} \Big)^{\frac{1}{2}} \\
    &\leq C |B(x, 1)|^{-2} \|F\|_{L^2_{\beta_1+\beta_2+1+\epsilon}(\mathbb{R}^2)} \\
    % &\leq C \|\Psi_j^{\alpha}\|_{L^2_{\beta_1+\beta_2+1+\epsilon}(\mathbb{R}^2)} \\
    &\leq C |B(x, 1)|^{-2} 2^{j(\beta_1 +\beta_2 +1/2+\epsilon)} ,
\end{align*}
where the last line follows from the fact $\|F\|_{L^2_{s}(\mathbb{R}^2)}= \|\varphi_j^{\alpha}\|_{L^2_{s}(\mathbb{R}^2)} \leq C 2^{j(s-1/2)}$ for $s>0$.
    
\end{proof}

\subsection{Step III: Estimates of \texorpdfstring{$\mathcal{B}^{\alpha}_{j,2}$, $\mathcal{B}^{\alpha}_{j,3}$}{} and \texorpdfstring{$\mathcal{B}^{\alpha}_{j,4}$}{} } 
\label{Step 3, estimate of Bj234 case}
In this step, we aim to prove that for any $\alpha>0$,  $(\ref{equation of estimate of Balpha j,l})$  holds for $l=2,3,4$. We remark that the estimate of  $\mathcal{B}^{\alpha}_{j,1}$ requires a more delicate analysis, which we will develop in the next step. We also remark that the estimates of $\mathcal{B}^{\alpha}_{j,2}$, $\mathcal{B}^{\alpha}_{j,3}$ and $\mathcal{B}^{\alpha}_{j,4}$ are significantly more complicated than that of in the Euclidean case. Among the other things, we lack the group structure and the ball volume also depends on the center in Grushin setup. Our approach relies on a combination of $L^p$-boundedness of maximal operator $\mathcal{M}$ and the crucial pointwise kernel estimates given in Lemma \ref{Lemma: Pointwise kernel estimate for Bj}.

We are now ready to estimate $\mathcal{B}^{\alpha}_{j,l}$ for $l=2, 3, 4$. Let us begin with the estimate for $\mathcal{B}^{\alpha}_{j,3}$.

\subsubsection{Case I: \texorpdfstring{$p_1=p_2=1$}{}}
Recall that $$\mathcal{K}^{\alpha}_{j,3} (x, y, z) = \mathcal{K}^{\alpha}_j (x, y, z) \chi_{B(x, 2^{j(1+\varepsilon)})^c}(y)  \chi_{B(x, 2^{j(1+\varepsilon)})} (z) .$$
Applying Lemma \ref{Lemma: Pointwise kernel estimate for Bj}, there exists $\epsilon_1>0$ such that
\begin{align*}
    & |\mathcal{B}^{\alpha}_{j,3}(f,g)(x)| \\
    &\leq \int_{\mathbb{R}^{d}} \int_{\mathbb{R}^{d}} \frac{ |\mathcal{K}^{\alpha}_{j}(x, y, z)|{\big(1 +  \varrho(x, y) \big)^N} }{ {\big(1 +  \varrho(x, y) \big)^N} } |f(y)| \chi_{B(x, 2^{j(1+\varepsilon)})^c}(y) |g(z)| \,\chi_{B(x, 2^{j(1+\varepsilon)})}(z)    \, dy\, dz \\
    & \leq C 2^{j(N+1/2+\epsilon_1)} \Biggl\{ \int_{\mathbb{R}^{d}}  \frac{|f(y)| \chi_{B (x, 2^{j(1+\varepsilon)})^c}(y)} {| B ( y, 1)| \big(1 +  \varrho(x, y) \big)^N } dy \Biggr\} \Biggl\{ \int_{\mathbb{R}^{d}}  \frac{|g(z)| \chi_{B (x, 2^{j(1+\varepsilon)})}(z)} {| B ( z, 1)|} dz \Biggr\} .
\end{align*}
Now, applying H\"older's inequality together with Lemma \ref{lemma: outside distance} and the fact \eqref{Doubling property of balls}, we  obtain 
\begin{align*}
    &\|\mathcal{B}^{\alpha}_{j,3}(f,g)\|_{L^{1/2}} \leq C 2^{j(N+1+\varepsilon)} \left\{ \int_{\mathbb{R}^{d} } \frac{|f(y)|}{| B (y, 1)|} \left\{ \int_{ \varrho(x, y) \geq 2^{j(1 +\varepsilon)}} \frac{dx}{{\big(1 +\varrho(x, y) \big)^N} } \right\} dy \right\} \\
    & \hspace{6cm} \left\{\int_{\mathbb{R}^{d} } \frac{|g(z)|}{| B (z, 1)|} \left\{ \int_{\varrho(x, z) \leq 2^{j(1 +\varepsilon)}} dx \right\} dz \right\} \\
    & \leq C 2^{j(N+1/2+\epsilon_1)} \Big\{ \int_{\mathbb{R}^{d}} \frac{|f(y)|}{| B ( y, 1)|} \{ 2^{-j(1 + \varepsilon)N} |B(y, 2^{j(1+\varepsilon)})| \} dy\Big\} \Big\{ \int_{\mathbb{R}^{d} } \frac{|g(z)|}{| B (z, 1)|} |B(z, 2^{j(1+\varepsilon)})| dz \Big\} \\
    &\leq C 2^{j(N+1/2+\epsilon_1)}\, 2^{-j(1 + \varepsilon)N} \,2^{2j(1+ \varepsilon)Q}\,\|f\|_{L^1} \, \|g\|_{L^1} \\
    &\leq C 2^{-j \delta} \|f\|_{L^1} \, \|g\|_{L^1} ,
\end{align*}
by choosing $N> 0$ large and $\varepsilon, \epsilon_1>0$ small, we get $\delta > 0$ so that $N \varepsilon> 1/2 + \epsilon+1 +2 Q(1+\varepsilon)$.

\subsubsection{Case II: \texorpdfstring{$p_1=1$}{} and \texorpdfstring{$p_2 > 1$}{}}
\label{subsection: case II P1 1 and P2 bigger than 1 case}
Similarly as in Case I, using Lemma \ref{Lemma: Pointwise kernel estimate for Bj}
\begin{align*}
    & |\mathcal{B}^{\alpha}_{j,3}(f,g)(x)| \\
    & \leq C 2^{j(N+1/2+\epsilon_1)} \Biggl\{ \int_{\mathbb{R}^{d}}  \frac{|f(y)| \chi_{B (x, 2^{j(1+\varepsilon)})^c}(y)} {| B ( y, 1)| \big(1 +  \varrho(x, y) \big)^N } dy \Biggr\} \Bigg\{\frac{1}{| B ( x, 1)|} \int_{\varrho(x,z) \leq 2^{j(1+\varepsilon)}} |g(z)| dz \Biggr\} .
\end{align*}
Then, use of H\"older's inequality along with Lemma \ref{lemma: outside distance}, imply
\begin{align*}
   \|\mathcal{B}^{\alpha}_{j,3}(f,g)\|_{L^{p}} & \leq C 2^{j(N+1/2+\epsilon_1)} \left\{ \int_{\mathbb{R}^{d} } \frac{|f(y)|}{| B (y, 1)|} \left\{ 2^{-j(1 + \varepsilon)N} |B(y, 2^{j(1+\varepsilon)})| \right\} dy \right\} \\
   & \hspace{5cm} \Biggl\{ \int_{\mathbb{R}^{d}} \frac{|B(x, 2^{j(1+\varepsilon)})|^{p_2}}{| B (x, 1)|^{p_2}} |\mathcal{M}g(x)|^{p_2} dx \Biggr\}^{1/{p_2}} \\
   &\leq C 2^{j(N+1/2+\epsilon_1)} 2^{-j(1 + \varepsilon)N} \,2^{2j(1+ \varepsilon)Q}\,\|f\|_{L^{1}} \|\mathcal{M}g\|_{L^{p_2}}  \\
    &\leq C 2^{-j \delta} \|f\|_{L^{1}}  \|g\|_{L^{p_2}} ,
\end{align*}
by choosing $N$ sufficiently large.

\subsubsection{Case III: \texorpdfstring{$p_1 > 1$}{} and \texorpdfstring{$p_2 = 1$}{}}
Again using Lemma \ref{Lemma: Pointwise kernel estimate for Bj}, we can see
\begin{align*}
    & |\mathcal{B}^{\alpha}_{j,3}(f,g)(x)| \leq C 2^{j(N+1/2+\epsilon_1)} \\
    & \hspace{2cm} \Bigg\{\frac{1}{| B ( x, 1)|} \int_{\varrho(x,z) \geq 2^{j(1+\varepsilon)}}  \frac{ |f(y)|} {\big(1 + \varrho(x, y) \big)^N } dy \Biggr\} \Biggl\{ \int_{\mathbb{R}^{d}}  \frac{|g(z)| \chi_{B (x, 2^{j(1+\varepsilon)})}(z)} {| B ( z, 1)|} dz \Biggr\} .
\end{align*}
Consequently, H\"older's inequality and Lemma \ref{lemma: outside distance} gives
\begin{align*}
   \|\mathcal{B}^{\alpha}_{j,3}(f,g)\|_{L^{p}} & \leq C 2^{j(N+1/2+\epsilon_1)}\, \Biggl\{ \int_{\mathbb{R}^{d}} \frac{|\mathcal{M}f(x)|^{p_1}}{| B (x, 1)|^{p_1}} \left\{ 2^{-j(1 + \varepsilon)N p_1} |B(x, 2^{j(1+\varepsilon)})|^{p_1} \right\} dx \Biggr\}^{1/{p_1}} \\
   & \hspace{6cm} \left\{ \int_{\mathbb{R}^{d} } \frac{|g(z)|}{| B (z, 1)|} \left\{ |B(z, 2^{j(1+\varepsilon)})| \right\} dz \right\}  \\
   &\leq C 2^{j(N+1/2+\epsilon_1)}\, 2^{-j(1 + \varepsilon)N} \,2^{2j(1+ \varepsilon)Q}\,\|\mathcal{M}f\|_{L^{p_1} } \, \|g\|_{L^{1}} \\
   &\leq C 2^{-j \delta} \|f\|_{L^{p_1} } \, \|g\|_{L^{1}} ,
\end{align*}
provided we choose $N$ sufficiently large.

\subsubsection{Case IV: \texorpdfstring{$p_1 > 1$}{} and \texorpdfstring{$p_2 > 1$}{}} This can be proved by combining the above arguments from Case III and Case II for $f$ and $g$ respectively.

The estimation of $\mathcal{B}^{\alpha}_{j,2}$ is identical to that of $\mathcal{B}^{\alpha}_{j,3}$, if we interchange the role of $f$ and $g$. Thus, the details are omitted. Furthermore, we can also estimate $\mathcal{B}^{\alpha}_{j,4}$ by following the above line arguments with obvious modification. In this case, we again omit the details.

\subsection{Step IV: Estimate of \texorpdfstring{$\mathcal{B}^{\alpha}_{j,1}$}{}}
\label{subsection: Estimate of Bj1 case}
In this final step, we establish (\ref{equation of estimate of Balpha j,l}) for $l=1$. First, we   decompose the operator $\mathcal{B}^{\alpha}_{j,1}$ further. Let $\varepsilon>0$. We choose a sequence $\{a_n\}_{n \in \mathbb{N}}$ in $\mathbb{R}^d$  with the property that whenever $n_1 \neq n_2$, we have $\varrho(a_{n_1}, a_{n_2}) > \frac{2^{j(1+\varepsilon)}}{10}$ and that $\sup_{a \in \mathbb{R}^d} \inf_{n} \varrho(a, a_{n}) \leq \frac{2^{j(1+\varepsilon)}}{10}$. We then define  $ S_{n}^j $ corresponding to this sequence  by 
\begin{align*}
    S_{n}^j = \Bar{B}\left(a_n, \tfrac{2^{j(1+\varepsilon)}}{10}\right) \setminus \bigcup_{m < n} \Bar{B}\left(a_m, \tfrac{2^{j(1+\varepsilon)}}{10} \right).
\end{align*}
Note that for $n_1 \neq n_2$, the balls $B(a_{n_1}, \frac{2^{j(1+\varepsilon)}}{20})$ and $ B(a_{n_2}, \frac{2^{j(1+\varepsilon)}}{20})$ are disjoint. Since these balls are disjoint, we have the following bounded overlapping property, 
\begin{align}
\label{Bounded overlapping property in j balls}
   \sup_{n} \#\{n_1 : \varrho(a_{n}, a_{n_1}) \leq  2 \cdot 2^{j(1+\varepsilon)}\} \leq C . 
\end{align} 
From (\ref{Expression: Kernel expression for Kj1}), we have
\begin{align*}
    \supp{\mathcal{K}^{\alpha}_{j,1}} \subseteq \mathcal{D}_{j} := \{(x,y,z) : \varrho(x,y) \leq 2^{j(1+\varepsilon)},\  \varrho(x,z) \leq 2^{j(1+\varepsilon)} \},
\end{align*}
which yields
\begin{align*}
    \mathcal{D}_{j} \subseteq \bigcup_{\substack{n,n_1,n_2: \varrho(a_n, a_{n_1})\leq 2\cdot 2^{j(1+\varepsilon)} \\ \varrho(a_n, a_{n_2})\leq 2\cdot 2^{j(1+\varepsilon)}}} S_{n}^j \times (S_{n_1}^j \times S_{n_2}^j) .
\end{align*}
Therefore, we can write
\begin{align}\label{decomposition of B alpha j 1}
    \mathcal{B}^{\alpha}_{j,1} (f,g)(x) &= \sum_{n=0}^{\infty} \sum_{\substack{n_1: \varrho(a_n , a_{n_1})\leq 2 \cdot 2^{j(1+\varepsilon)} \\ n_2: \varrho(a_n , a_{n_2})\leq 2 \cdot 2^{j(1+\varepsilon)}}} \chi_{S_{n}^j}(x) \mathcal{B}^{\alpha}_{j,1} (f_{n_1},g_{n_2})(x) ,
\end{align}
where $f_{n_1}^j = f \chi_{S_{n_1}^j}$ and $g_{n_2}^j = g \chi_{S_{n_2}^j}$. 

We now claim that, for
\begin{align}
\label{Assumption of the main claim}
    \varrho(a_n , a_{n_1})\leq 2 \cdot 2^{j(1+\varepsilon)} \quad \text{and} \quad \varrho(a_n , a_{n_2})\leq 2 \cdot 2^{j(1+\varepsilon)} ,
\end{align}
whenever $\alpha>\alpha(p_1, p_2)$, there exists a $\delta>0$ such that
\begin{align}
\label{claim in the estimate of Balpha j,1}
    \| \chi_{S_{n}^j} \mathcal{B}^{\alpha}_{j,1} (f_{n_1}^j,g_{n_2}^j) \|_{L^{p}} \leq C 2^{-j \delta} \|f_{n_1}^j\|_{L^{p_1}} \|g_{n_2}^j\|_{L^{p_2}} ,
\end{align}
with $1/p=1/p_1+1/p_2$ and $1\leq p_1, p_2 \leq \infty$. 

At this point, we postpone the proof of this claim and continue with the estimate of $\mathcal{B}^{\alpha}_{j,1}$, assuming the claim is true. Using the fact that the sets $S_n^j$ are disjoint, together with the bounded overlapping property \eqref{Bounded overlapping property in j balls} discussed above, it follows from (\ref{decomposition of B alpha j 1}) that
\begin{align}\label{break1}
    &\|\mathcal{B}^{\alpha}_{j,1} (f,g)\|_{L^{p}} = \Big\|\sum_{n=0}^{\infty} \sum_{\substack{n_1: \varrho(a_n , a_{n_1})\leq 2 \cdot 2^{j(1+\varepsilon)} \\ n_2: \varrho(a_n ,a_{n_2})\leq 2 \cdot 2^{j(1+\varepsilon)}}} \chi_{S_{n}^j} \mathcal{B}^{\alpha}_j (f_{n_1}^j,g_{n_2}^j) \Big\|_{L^{p}} \\
    &\nonumber \leq C \Big(\sum_{n=0}^{\infty} \sum_{\substack{n_1: \varrho(a_n , a_{n_1})\leq 2 \cdot 2^{j(1+\varepsilon)} \\ n_2: \varrho(a_n , a_{n_2})\leq 2 \cdot 2^{j(1+\varepsilon)}}} \| \chi_{S_{n}^j} \mathcal{B}^{\alpha}_j (f_{n_1}^j,g_{n_2}^j) \|_{L^{p}}^{p} \Big)^{1/p}.
\end{align}
With the help of the above claim \eqref{claim in the estimate of Balpha j,1}, applying H\"older's inequality and again bounded overlapping property \eqref{Bounded overlapping property in j balls}, we further estimate right hand side of (\ref{break1})  by 
\begin{align*}
    &\nonumber C 2^{-j \delta } \Bigg(\sum_{n=0}^{\infty} \Big\{ \sum_{n_1: \varrho(a_n , a_{n_1})\leq 2 \cdot 2^{j(1+\varepsilon)}} \|f_{n_1}^j\|_{L^{p_1}}^{p} \Big\} \Big\{ \sum_{n_2: \varrho(a_n , a_{n_2})\leq 2 \cdot 2^{j(1+\varepsilon)}} \|g_{n_2}^j\|_{L^{p_2}}^{p} \Big\} \Bigg)^{\frac{1}{p}} \\
    &\nonumber \leq C 2^{-j \delta } \Big\{\sum_{n=0}^{\infty} \sum_{n_1: \varrho(a_n , a_{n_1})\leq 2 \cdot 2^{j(1+\varepsilon)}} \|f_{n_1}^j\|_{L^{p_1}}^{p_1} \Big\}^{\frac{1}{p_1}} \Big\{\sum_{n=0}^{\infty} \sum_{n_2: \varrho(a_n , a_{n_2})\leq 2 \cdot 2^{j(1+\varepsilon)}} \|g_{n_2}^j\|_{L^{p_2}}^{p_2} \Big\}^{\frac{1}{p_2}} \\
    &\nonumber \leq C 2^{-j \delta} \|f\|_{L^{p_1}} \|g\|_{L^{p_2}} .
\end{align*}
Thus, we have completed the estimate of $\mathcal{B}^{\alpha}_{j,1}$.

Therefore, it only remains to verify the claim \eqref{claim in the estimate of Balpha j,1}. First, observe that an application of bilinear interpolation, as discussed in \cite[Section 4.3]{Bernicot_Grafakos_Song_Yan_Bilinear_Bochner_Riesz_2015}, it is enough to verify the claim \eqref{claim in the estimate of Balpha j,1} for all points $(p_1,p_2,p) \in \{(\infty, \infty, \infty)$, $(2, \infty, 2)$, $(1,1,1/2)$, $(1, 2, 2/3)$, $(1, \infty, 1)$, $(2, 2, 1)$, $(\infty,2,2)$, $(\infty, 1, 1)$, $(2,1,2/3)\}$. Again, interchanging the role of $f$ and $g$, we may restrict ourself to the cases when $(p_1,p_2,p) \in \{(\infty, \infty, \infty)$, $(2, \infty, 2)$, $(1,1,1/2)$, $(1, 2, 2/3)$, $(1, \infty, 1)$, $(2, 2, 1)\}$. Thus, in the next few sections, our goal is to establish the claim \eqref{claim in the estimate of Balpha j,1} for the points $(\infty, \infty, \infty)$, $(2, \infty, 2)$, $(1,1,1/2)$, $(1, 2, 2/3)$, $(1, \infty, 1)$, $(2, 2, 1)$.

\section{Proof of the claim (\ref{claim in the estimate of Balpha j,1}) for \texorpdfstring{$(p_1,p_2,p)=(\infty,\infty,\infty)$}{}} \label{section 5}
In this section, we establish the validity of the claim \eqref{claim in the estimate of Balpha j,1} for the point $(p_1, p_2,p)=(\infty, \infty, \infty)$. In the Euclidean settings, the result at this point follows straightforwardly from the explicit kernel expression of the bilinear Bochner-Riesz operator. However, since such a kernel representation is not known in our framework, we adopt a different approach. Our proof is based on a bilinear analogue of the weighted Plancherel estimate with respect to the second-layer of the bilinear Bochner-Riesz operator (see Proposition \ref{prop: second layer weighted Plancherel}). In this case, the relevant exponent is $\alpha(\infty, \infty)=d-1/2$.

Recall that $\mathcal{K}^{\alpha}_{j,1}$ denotes the kernel of $\mathcal{B}^{\alpha}_{j,1}$ (see \eqref{Bilinear Bochner-Riesz operator in terms of kernel}). Let $\gamma_1, \gamma_2\geq 0$. An application of  H\"older's inequality leads to
\begin{align}
\label{Application of Holder in infinity}
    |\chi_{S_{n}^j} \mathcal{B}^{\alpha}_{j,1} (f_{n_1}^j,g_{n_2}^j)(x)| &\leq \left(\int_{\mathbb{R}^{d}} \int_{\mathbb{R}^{d}} |x''-y''|^{2\gamma_1} |x''-z''|^{2\gamma_2} |\mathcal{K}^{\alpha} _{j,1} (x, y, z)|^2 \ dy \ dz \right)^{1/2} \\
    &\nonumber \hspace{2cm} \left(\int_{\mathbb{R}^{d}} \frac{|f_{n_1}^j(y)|^2}{|x''-y''|^{2\gamma_1}} \ dy \right)^{1/2} \left(\int_{\mathbb{R}^{d}} \frac{|g_{n_2}^j(z)|^2}{|x''-z''|^{2\gamma_2}} \ dz \right)^{1/2} .
\end{align}

Since, we are proving claim \eqref{claim in the estimate of Balpha j,1}, from \eqref{Assumption of the main claim} we always have $ \varrho(a_n , a_{n_1})\leq 2 \cdot 2^{j(1+\varepsilon)}$ and $\varrho(a_n , a_{n_2})\leq 2 \cdot 2^{j(1+\varepsilon)}$. In the Grushin setup, the volume of a ball depends both on its radius as well as on its center, we divide the proof into the following two cases.

\subsection{Case-I : \texorpdfstring{$|a'_n| > \frac{11}{5} 2^{j(1+\varepsilon)}$}{}}
\label{subsection: case I of infinity infinity}
In this case, we take $\gamma_1=\gamma_2=0$. Note that $\supp{\varphi_j^{\alpha}} \subseteq [0,1]^2$. Therefore, by applying Proposition \ref{Bilinear weighted Plancherel}, we get
\begin{align*}
    \left(\int_{\mathbb{R}^{d}} \int_{\mathbb{R}^{d}} |\mathcal{K}^{\alpha}_{j,1} (x, y, z)|^2 \ dy \ dz \right)^{1/2} &\leq C \|\varphi_j^{\alpha}\|_{L^2(\mathbb{R}^2)} |x'|^{-d_2/2} |x'|^{-d_2/2} .
\end{align*}

Note that $ \varrho(a_n , a_{n_1})\leq 2 \cdot 2^{j(1+\varepsilon)}$ and $\varrho(a_n , a_{n_2})\leq 2 \cdot 2^{j(1+\varepsilon)}$ and together with the assumption $|a_{n}'|> 11 \cdot 2^{j(1+\varepsilon)}/5$, implies $|a_n'| \sim |a_{n_1}'| \sim |a_{n_2}'|$.

Also, since $x\in B(a_n, 2^{j(1+\varepsilon)}/5)$, we have $|x'-a_n'| \leq 2^{j(1+\varepsilon)}/5$. This again implies that $|x'| \geq 2|a'_n|$.

Hence, using the above observations along with fact \eqref{Estimate : Ball volume}, it follows from \eqref{Application of Holder in infinity} that
\begin{align*}
    & |\chi_{S_{n}^j} \mathcal{B}^{\alpha}_{j,1} (f_{n_1}^j,g_{n_2}^j)(x)| \\
    %&\leq \left(\int_{\mathbb{R}^{d}} \int_{\mathbb{R}^{d}} |\mathcal{K}^{\alpha}_{j,1} (x, y, z)|^2 \ dy \ dz \right)^{1/2} \left(\int_{\mathbb{R}^{d}} |f_{n_1}^j(y)|^2 \ dy \right)^{1/2} \left(\int_{\mathbb{R}^{d}} |g_{n_2}^j(z)|^2 \ dz \right)^{1/2} \\
    &\leq C \|\varphi_j^{\alpha}\|_{L^2(\mathbb{R}^2)} |x'|^{-d_2} \|f_{n_1}^j\|_{L^{\infty}}\, 2^{j(1+\varepsilon)d/2}\, |a_{n_1}'|^{d_2/2} \|g_{n_2}^j\|_{L^{\infty}}\, 2^{j(1+\varepsilon)d/2}\, |a_{n_2}'|^{d_2/2} \\
    &\leq C 2^{-j \alpha} 2^{-j/2} 2^{j d (1+\varepsilon)} \|f_{n_1}^j\|_{L^{\infty}} \|g_{n_2}^j\|_{L^{\infty}} ,
\end{align*}
where we have used the fact $\|\varphi_j^{\alpha}\|_{L^2(\mathbb{R}^2)} \leq C 2^{-j \alpha} 2^{-j/2}$.

Since $\alpha>d-1/2$, we can choose $\varepsilon>0$ very small such that $\delta=\alpha-d-1/2-\varepsilon d>0$ and
\begin{align*}
    \|\chi_{S_{n}^j} \mathcal{B}^{\alpha}_{j,1} (f_{n_1}^j,g_{n_2}^j)\|_{L^{\infty}} &\leq C 2^{-j\delta} \|f_{n_1}^j\|_{L^{\infty}} \|g_{n_2}^j\|_{L^{\infty}} .
\end{align*}

\subsection{Case-II : \texorpdfstring{$|a'_n| \leq \frac{11}{5} 2^{j(1+\varepsilon)}$}{}}
\label{subsection: case II for infinity infinity}
For this scenario, by applying Proposition \ref{prop: second layer weighted Plancherel} for $0\leq \gamma_1, \gamma_2 < d_2/2$, we obtain
\begin{align}
\label{kernel estimate in infinity case 2}
    \left(\int_{\mathbb{R}^{d}} \int_{\mathbb{R}^{d}} |x''-y''|^{2\gamma_1} |x''-z''|^{2\gamma_2} |\mathcal{K}^{\alpha} _{j,1} (x, y, z)|^2 \ dy \ dz \right)^{1/2} &\leq C \|\varphi_j^{\alpha}\|_{L_{\gamma_1, \gamma_2}^2(\mathbb{R}^2)} \\
    &\nonumber \leq C 2^{-j \alpha} 2^{\gamma_1+\gamma_2-j/2} .
\end{align}

Also note that, for $i=1,2$, the conditions $|a_n' - a_{n_i}'|\leq 2 \cdot 2^{j(1+\varepsilon)}$ implies $|a_{n_i}'| \leq \frac{21}{5} 2^{j(1+\varepsilon)}$. Therefore, using \eqref{kernel estimate in infinity case 2} and applying Lemma \ref{Lemma: Integral of weight over ball in second layer} for $0\leq \gamma_1, \gamma_2 < d_2/2$, the estimate \eqref{Application of Holder in infinity} gives
\begin{align*}
    & |\chi_{S_{n}^j} \mathcal{B}^{\alpha}_{j,1} (f_{n_1}^j,g_{n_2}^j)(x)| \\
    &\leq C 2^{-j \alpha} 2^{\gamma_1+\gamma_2-j/2} \|f_{n_1}^j\|_{L^{\infty}} 2^{j(1+\varepsilon)(d_1/2+(d_2-2\gamma_1))} \|g_{n_2}^j\|_{L^{\infty}} 2^{j(1+\varepsilon)(d_1/2+(d_2-2\gamma_2))} .
\end{align*}
As $\alpha > d-1/2$, we can choose $\varepsilon>0$ sufficiently small and $0 \leq \gamma_1 + \gamma_2 < d_2$ such that $\delta=\alpha-(d_1+\gamma_1+ \gamma_2-1/2) -(2d_2-2\gamma_1-2\gamma_2)-\varepsilon(d_1+(2d_2-2\gamma_1-2\gamma_2))>0$ and
\begin{align*}
    \|\chi_{S_{n}^j} \mathcal{B}^{\alpha}_{j,1} (f_{n_1}^j,g_{n_2}^j)\|_{L^{\infty}} &\leq C 2^{-j\delta} \|f_{n_1}^j\|_{L^{\infty}} \|g_{n_2}^j\|_{L^{\infty}} .
\end{align*}

\section{Proof of the claim (\ref{claim in the estimate of Balpha j,1}) for \texorpdfstring{$(p_1,p_2,p)=(2,\infty,2)$}{}}

In this context, we have $\alpha(2, \infty) =(d-1)/2$. Let $\gamma \geq 0$. Applying Cauchy-Schwartz inequality, we obtain 
\begin{align}
\label{Application of Holder in 2, infty}
    &\|\chi_{S_{n}^j} \mathcal{B}^{\alpha}_{j,1} (f_{n_1}^j, g_{n_2}^j)\|_{L^2} \leq \Big[\sup_x \Big(\int_{\mathbb{R}^d} \frac{|g_{n_2}^j(z)|^2}{|x''-z''|^{2\gamma}} \, dz \Big)^{\frac{1}{2}} \Big] \\
    &\nonumber \hspace{3cm} \times \Big(\int_{\mathbb{R}^d } \int_{\mathbb{R}^d} |x''-z''|^{2\gamma}  \Big|\int_{\mathbb{R}^d} \mathcal{K}^{\alpha} _{j,1} (x, y,z)\, f_{n_1}^j(y)\, dy \Big|^2 dz \, dx \Big)^{\frac{1}{2}} .
\end{align}
Similarly as in \eqref{Introducing cutoff in the kernel expression}, take two non-negative, increasing sequence of Borel functions $\{\zeta_{n_3}\}_{n_3 \in \mathbb{N}}$ and $\{\zeta_{n_4}\}_{n_4 \in \mathbb{N}}$ on $\mathbb{R}$ such that they are compactly supported in $\mathbb{R} \setminus \{0\}$ and converging pointwise on $\mathbb{R} \setminus \{0\}$ to the constant function $1$. Then define
\begin{align*}
    \mathcal{K}^{\alpha} _{j,1, n_3, n_4}(x, \cdot, \cdot) &:= \zeta_{n_3}(T_1) \zeta_{n_4}(T_2)(\mathcal{K}^{\alpha} _{j,1} (x, \cdot, \cdot)) ,
\end{align*}
Then $\mathcal{K}^{\alpha} _{j,1, n_3, n_4}(x, \cdot, \cdot) \to \mathcal{K}^{\alpha} _{j,1} (x, \cdot, \cdot)$ in $L^2(\mathbb{R}^d \times \mathbb{R}^d)$ for almost all $x \in \mathbb{R}^d$. Therefore to handle the second factor in the right hand side of \eqref{Application of Holder in 2, infty}, an application of Fatou's lemma implies that enough to estimate the second term with $\mathcal{K}^{\alpha} _{j,1, n_3, n_4}(x, \cdot, \cdot)$ replacing $\mathcal{K}^{\alpha} _{j,1} (x, \cdot, \cdot)$. From \eqref{Explicit Kernel expression} we get
\begin{align}
\label{writting kernel in different way in 2, infinity}
&\int_{\mathbb{R}^d} \mathcal{K}^{\alpha} _{j,1,n_3,n_4} (x, y,z)\, f_{n_1}^j(y)\, dy \\
&\nonumber = \frac{1}{(2\pi)^{d_2}}\int_{\mathbb{R}^{d_2}} e^{i\lambda_2\cdot(x'' - z'')}\, \sum_{k_2= 0} ^{\infty} \mathcal{G}_{x,n_3}^j ([k_2]|\lambda_2|) \zeta_{n_4}(|\lambda_2|) \sum_{|\mu_2|=k_2} \Phi_{\mu_2}^{\lambda_2}(x') \Phi_{\mu_2}^{\lambda_2}(z') \,  d\lambda_2\\
&\nonumber =: \mathcal{K}_{\mathcal{G}_{x, n_3, n_4}^j(\mathcal{L}, T)} (z, x),
\end{align}
where $\mathcal{G}_{x, n_3, n_4}^j(\eta_2, \tau_2)= \mathcal{G}_{x,n_3}^j(\eta_2) \zeta_{n_4}(\tau_2)$ with
\begin{align}
\label{Expression for mathcal g}
  \mathcal{G}_{x,n_3}^j (\eta_2)= \frac{1}{(2\pi)^{d_2}}\int_{\mathbb{R}^{d_2}} e^{i\lambda_1\cdot(x'' - y'')}\, \sum_{k_1= 0} ^{\infty}  \varphi_j^{\alpha} ([k_1]|\lambda_1|, \eta_2) \zeta_{n_3}(|\lambda_1|) P_{k_1}^{\lambda_1}(f_{n_1}^j)^{\lambda_1}(x') \, d\lambda_1 .
\end{align}
This leads us to write the square of the second factor of \eqref{Application of Holder in 2, infty} as follows.
\begin{align}
\label{eqn. L^2-norm of |z'| substitute}
& \int_{\mathbb{R}^d } \int_{\mathbb{R}^d} |x''-z''|^{2\gamma}  
\Big|\int_{\mathbb{R}^d} \mathcal{K}^{\alpha} _{j,1,n_3,n_4} (x, y,z)\, f_{n_1}^j(y)\, dy \Big|^2 dz dx \\
&\nonumber= \int_{\mathbb{R}^d } \int_{\mathbb{R}^d} |x''-z''|^{2\gamma} | \mathcal{K}_{\mathcal{G}_{x,n_3,n_4}^j(\mathcal{L})} (z, x) |^2 dz dx .
\end{align}

Now we consider the following two cases.

\subsection{Case-I : \texorpdfstring{$|a'_n| > \frac{11}{5} 2^{j(1+\varepsilon)}$}{}}
In this case we take $\gamma=0$. First, recall that, from subsection \ref{subsection: case I of infinity infinity} whenever $x \in S_{n}^j$, we have $|x'| \geq 2|a'_n|$ and $|a_n'| \sim |a_{n_1}'| \sim |a_{n_2}'|$. Now Proposition \ref{Weighted Plancherel Martini Sikora} with $\gamma=0$ provides
\begin{align}
\label{eqn. L^2-norm of |z'| for gamma 0 case}
    \int_{\mathbb{R}^d} | \mathcal{K}_{\mathcal{G}_{x,n_3,n_4}^j(\mathcal{L})} (z, x) |^2 dz & \leq C \int_{0}^{1} |\mathcal{G}_{x,n_3,n_4}^j (\eta_2)|^2 |x'|^{- d_2}\,d\eta_2 .
\end{align}
Combining the estimate \eqref{eqn. L^2-norm of |z'| for gamma 0 case}, \eqref{eqn. L^2-norm of |z'| substitute} and plugging them in \eqref{Application of Holder in 2, infty} along with \eqref{Estimate : Ball volume} yields
\begin{align}
\label{Estmate of L2 norm in case I for 2 infinity}
  \|\chi_{S_{n}^j} \mathcal{B}^{\alpha}_{j,1} (f_{n_1}^j, g_{n_2}^j)\|_{L^2} &\leq C \Big[\int_{S_{n}^j} \int_{0}^{1} |\mathcal{G}_{x,n_3,n_4}^j(\eta_2)|^2 \,  |x'|^{- d_2}\,d\eta_2 \, dx \Big]^{\frac{1}{2}} \|g_{n_2}^j\|_{L^{\infty}}\, 2^{j(1+\varepsilon)\frac{d}{2}}\, |a_{n_2}'|^{\frac{d_2}{2}}  \\
  &\nonumber\leq C \Big( \int_{\mathbb{R}^d } \int_{0}^{1} |\mathcal{G}_{x,n_3,n_4}^j (\eta_2)|^2 \,  \,d\eta_2 \,  dx   \Big)^{\frac{1}{2}} \|g_{n_2}^j\|_{L^{\infty}}\, 2^{j(1+\varepsilon)d/2} .
\end{align}
Furthermore, from \eqref{Expression for mathcal g}, we have the following estimate.
\begin{align}
\label{L2 estimate of mathcal g}
    & \int_{0}^{1} \int_{\mathbb{R}^d } |\mathcal{G}_{x,n_3,n_4}^j (\eta_2)|^2 \, dx \,d\eta_2 \\
    &\nonumber =C \int_{0}^{1} \int_{\mathbb{R}^{d_2} } \sum_{k_1= 0}^{\infty}  |\varphi_j^{\alpha} ([k_1]|\lambda_1|, \eta_2)|^2 |\zeta_{n_3}([k_1]|\lambda_1|)|^2 |\zeta_{n_4}(\eta_2)|^2 \|P_{k_1}^{\lambda_1}(f_{n_1}^j)^{\lambda_1}\|_{L^2}^2 \, d\lambda_1 \,d\eta_2 \\
    &\nonumber \leq C \int_{\mathbb{R}^{d_2} } \sum_{k_1= 0}^{\infty}  \left( \int_{0}^{1} |\varphi_j^{\alpha} ([k_1]|\lambda_1|, \eta_2)|^2 \,d\eta_2 \right) \|P_{k_1}^{\lambda_1}(f_{n_1}^j)^{\lambda_1}\|_{L^2}^2 \, d\lambda_1 \\
    &\nonumber \leq C 2^{-2j \alpha} 2^{-j} \|f_{n_1}^j\|_{L^2}^2 ,
\end{align}
where we have used the fact: $\sup_{k_1}\int_{0}^{1} |\varphi_j^{\alpha} ([k_1]|\lambda_1|, \eta_2)|^2 \,d\eta_2 \leq C 2^{-2j \alpha} 2^{-j} $.

Finally, substituting \eqref{L2 estimate of mathcal g} in \eqref{Estmate of L2 norm in case I for 2 infinity}, we obtain
\begin{align*}
    \|\chi_{S_{n}^j} \mathcal{B}^{\alpha}_{j,1} (f_{n_1}^j, g_{n_2}^j)\|_{L^2} &\leq C 2^{-j \alpha} 2^{-j/2} 2^{j(1+\varepsilon)d/2} \|f_{n_1}^j\|_{L^2} \|g_{n_2}^j\|_{L^{\infty}} \\
    &\leq C 2^{-j \delta} \|f_{n_1}^j\|_{L^2} \|g_{n_2}^j\|_{L^{\infty}} ,
\end{align*}
where $\delta=\alpha-(d-1)/2-\varepsilon d/2$, since $\alpha>(d-1)/2$, we can choose $\varepsilon>0$ sufficiently small such that $\alpha-(d-1)/2-\varepsilon d/2>0$.

\subsection{Case-II : \texorpdfstring{$|a'_n| \leq \frac{11}{5} 2^{j(1+\varepsilon)}$}{}}

First note that, applying Proposition \ref{Lemma: Martini_Mullar_Weighted_Plancherel without cutoff} for $0\leq \gamma< d_2/2$, we have
\begin{align}
\label{eqn. L^2-norm of |z'|}
   \int_{\mathbb{R}^d} |x''-z''|^{2\gamma} | \mathcal{K}_{\mathcal{G}_{x,n_3,n_4}^j(\mathcal{L})} (z, x) |^2 dz &\leq C \int_{\mathbb{R}} (1+\xi_2^2)^{\gamma} |\widehat{\mathcal{G}_{x,n_3,n_4}^j} (\xi_2)|^2 \, d\xi_2 ,
\end{align}
where for fixed $x\in \mathbb{R}^d$, $\widehat{\mathcal{G}_{x,n_3,n_4}^j}$ denotes the Fourier transform of $\mathcal{G}_{x,n_3,n_4}^j$.

From \eqref{Application of Holder in 2, infty} along with the estimate \eqref{eqn. L^2-norm of |z'|}, \eqref{eqn. L^2-norm of |z'| substitute} and Lemma \ref{Lemma: Integral of weight over ball in second layer} for $0\leq \gamma <d_2/2$ yields
\begin{align*}
   & \|\chi_{S_{n}^j} \mathcal{B}^{\alpha}_{j,1} (f_{n_1}^j, g_{n_2}^j)\|_{L^2} \\
   &\leq C \Bigl(\int_{\mathbb{R}^d } \int_{\mathbb{R}} (1+\xi_2^2)^{\gamma} 
   |\widehat{\mathcal{G}_{x,n_3,n_4}^j} (\xi_2)|^2 \, d\xi_2 \,  dx   \Bigr)^{\frac{1}{2}} \|g_{n_2}^j\|_{L^{\infty}}\, 2^{j(1+\varepsilon) (d_1+2(d_2-2\gamma))/2}  .
\end{align*}
Similar to \eqref{L2 estimate of mathcal g}, here also we have the following estimate.
\begin{align*}
    & \int_{\mathbb{R}} \int_{\mathbb{R}^d } (1+\xi_2^2)^{\gamma} 
    |\widehat{\mathcal{G}_{x,n_3,n_4}^j} (\xi_2)|^2 \, dx \, d\xi_2 \\
    &=C \int_{\mathbb{R}} \int_{\mathbb{R}^d } \sum_{k_1= 0}^{\infty} (1+\xi_2^2)^{\gamma} |\mathcal{F}_2 \varphi_j^{\alpha} ([k_1]|\lambda_1|, \xi_2)|^2 |\zeta_{n_3}([k_1]|\lambda_1|)|^2 |\zeta_{n_4}(\xi_2)|^2 \|P_{k_1}^{\lambda_1}(f_{n_1}^j)^{\lambda_1}\|_{L^2}^2 \, d\lambda_1 \,d\xi_2 \\
    &\leq C \int_{\mathbb{R}^d } \sum_{k_1= 0}^{\infty} \Big(\int_{\mathbb{R}} (1+\xi_2^2)^{\gamma} |\mathcal{F}_2 \varphi_j^{\alpha} ([k_1]|\lambda_1|, \xi_2)|^2 \,d\xi_2 \Big) \|P_{k_1}^{\lambda_1}(f_{n_1}^j)^{\lambda_1}\|_{L^2}^2 \, d\lambda_1 \\
    &\leq C 2^{-2 j \alpha} 2^{2j\gamma-j} \|f_{n_1}^j\|_{L^2}^2 ,
\end{align*}
where we have used $\sup_{k_1} \|\varphi_j^{\alpha}([k_1]|\lambda_1|, \cdot)\|_{L^2_{\gamma}} \leq C 2^{j\gamma-j/2}$.

With the help of the above estimate we obtain,
\begin{align*}
    \|\chi_{S_{n}^j} \mathcal{B}^{\alpha}_{j,1} (f_{n_1}^j, g_{n_2}^j)\|_{L^2} &\leq C 2^{-j \alpha} 2^{j\gamma-j/2} 2^{j(1+\varepsilon) (d_1+2(d_2-2\gamma))/2} \|f_{n_1}^j\|_{L^2} \|g_{n_2}^j\|_{L^{\infty}} \\
    &\leq C 2^{-j \delta} \|f_{n_1}^j\|_{L^2} \|g_{n_2}^j\|_{L^{\infty}} ,
\end{align*}
where, since $\alpha > (d-1)/2$, we can choose $\varepsilon>0$ very small and $0\leq \gamma <d_2/2$ such that  $\delta=\alpha-((d_1/2+\gamma-1/2)-(d_2-2\gamma)-\varepsilon((d_1+2(d_2-2\gamma))/2)>0$.

\section{Proof of the claim (\ref{claim in the estimate of Balpha j,1}) for \texorpdfstring{$(p_1,p_2,p)=(1,1,1/2)$}{}}
\label{Section: Proof of claim for 1,1 case}

Before discussing the proof, we would like to point out that in the Euclidean context, the boundedness of bilinear Bochner-Riesz means at $(1,1, 1/2)$ follows easily from the H\"older's inequality and the explicit kernel expression of the bilinear Bochner-Riesz means, as shown in (see \cite[Proposition 4.2]{Bernicot_Grafakos_Song_Yan_Bilinear_Bochner_Riesz_2015}). However, in our case, due to the unavailability of explicit kernel formula, the estimate of $\mathcal{B}^{\alpha}_{j,1}$ at $(1,1,1/2)$ is complicated and more involved. To address this, we use some ideas from \cite{Niedorf_Bochner_Riesz_Grushin_2022}.

Recall  from \eqref{Bilinear Bochner-Riesz operator after dyadic decomposition} that we have
\begin{align*}
    \mathcal{B}^{\alpha}_{j,1} (f_{n_1}^j,g_{n_2}^j)(x) &= \frac{1}{(2\pi)^{2 d_2}} \int_{\mathbb{R}^{d_2}} \int_{\mathbb{R}^{d_2}} e^{i(\lambda_1 + \lambda_2)\cdot x''} \sum_{k_1, k_2 = 0}^{\infty} \varphi^{\alpha}_j ([k_1]|\lambda_1|, [k_2]|\lambda_2|) \\
    & \hspace{6cm} P^{\lambda_1} _{k_1} (f_{n_1}^j)^{\lambda_1}(x') P^{\lambda_2} _{k_2} (g_{n_2}^j)^{\lambda_2}(x')\, d\lambda_1 d\lambda_2.
\end{align*}

We proceed by expressing $\varphi_j^{\alpha}$ as a Fourier series on a second variable. To do so, let us fix $\eta_1=[k_1]|\lambda_1|$ and observe that $\supp \varphi_j^{\alpha} (\eta_1, \cdot) \subseteq [0,1]$. By setting $\varphi_j^{\alpha} (\eta_1, \cdot)=0$ on $[-1,0]$, we extend $\varphi_j^{\alpha} (\eta_1, \cdot)$ to a $2$-periodic function on the whole of $\mathbb{R}$. This allows us to expand it in a Fourier series as
\begin{align}
\label{Fourier series decomposition}
    \varphi_j^{\alpha} (\eta_1, \eta_2) = \sum_{l \in \mathbb{Z}} \phi_{j,l}^{\alpha}(\eta_1) e^{i \pi l \eta_2} ,
\end{align}
where $\phi_{j,l}^{\alpha}(\eta_1)= \frac{1}{2} \int_{-1}^1 \varphi_j^{\alpha} (\eta_1, \eta_2) e^{-i \pi l \eta_2}\, d\eta_2$.

Let $\Tilde{\chi} \in C_c^{\infty}(\mathbb{R})$ such that it is $1$ in $[-1,1]$ and $0$ outside $[-2,2]$. Then the above expansion of $\varphi_j^{\alpha}$ yields the following decomposition of $\mathcal{B}^{\alpha}_{j,1}$.
\begin{align}
\label{Equation: Decomposition in terms of Fourier series}
    \mathcal{B}^{\alpha}_{j,1}(f_{n_1}^j, g_{n_2}^j)(x) &= C \sum_{l \in \mathbb{Z}} \Big\{\int_{\mathbb{R}^{d_2}} e^{i\lambda_1 \cdot x''} \sum_{k_1 = 0} ^{\infty} \phi_{j, l}^{\alpha} ([k_1] |\lambda_1|) P^{\lambda_1}_{k_1} (f_{n_1}^j)^{\lambda_1}(x') \, d\lambda_1 \Big\} \\
    &\hspace{2cm} \Big\{ \int_{\mathbb{R}^{d_2}} e^{i\lambda_2 \cdot x''} \sum_{k_2 = 0} ^{\infty} e^{i \pi l [k_2] |\lambda_2|} \Tilde{\chi}([k_2] |\lambda_2|) P^{\lambda_2}_{k_2} (g_{n_2}^j)^{\lambda_2}(x') \, d\lambda_2 \Big\}\nonumber \\
    &=: C \sum_{l \in \mathbb{Z}} \left\{ \phi_{j,l}^{\alpha}(\mathcal{L})f_{n_1}^j(x) \right\} \left\{ \psi_l(\mathcal{L})g_{n_2}^j(x) \right\} \nonumber,
\end{align}
where we have used the fact that $0\leq [k_2]|\lambda_2|\leq 1$ and $\psi_l$ is defined by $\psi_l(\eta_2) := e^{i \pi l \eta_2} \Tilde{\chi}(\eta_2)$. 

Since $d_1, d_2 \geq 1$ and the volume of any ball in Grushin setting depends on both of its radius and its center, we consider the following cases.

\subsection{Case-I : \texorpdfstring{$|a'_n| > \frac{11}{5} 2^{j(1+\varepsilon)}$}{}}
\label{subsection: case I of 1,1 case}
Note that $\alpha(1,1)= \min\{D, d_1+d_2+1\}>d$. Application of H\"older's inequality, together with the fact \eqref{Estimate : Ball volume} and Proposition \ref{Weighted restriction estimate} with $\gamma=0$, we obtain
\begin{align}
\label{Application of Holder in case I for 1,1 case}
    \|\chi_{S_{n}^j} \mathcal{B}^{\alpha}_{j,1} &(f_{n_1}^j,g_{n_2}^j)\|_{L^{1/2}} \leq C |S_{n}^j| \sum_{l \in \mathbb{Z}} \| \phi_{j,l}^{\alpha}(\mathcal{L})f_{n_1}^j\|_{L^2} \| \psi_l(\mathcal{L})g_{n_2}^j \|_{L^2} \\
    &\nonumber\leq C 2^{j d (1 + \varepsilon)} |a_n'|^{d_2} \sum_{l \in \mathbb{Z}} |a_{n_1}'|^{- d_2/2} \|\phi_{j,l} ^{\alpha}\|_{L^{\infty} } \|f_{n_1}^j\|_{L^{1}} |a_{n_2}'|^{- d_2/2}  \|g_{n_2}^j\|_{L^{1}} .
\end{align}
From \eqref{Fourier series decomposition}, one can easily see that for any $\beta>0$,
\begin{align}
\label{Convergence of sum of l using phi}
\sup_{\eta_1 \in [0,1]} |\phi_{j,l} ^{\alpha}(\eta_1)| (1 + |l|)^{1 + \beta} \leq C 2^{-j \alpha} 2^{j \beta} .
\end{align}

As noted in subsection \ref{subsection: case I of infinity infinity}, in this case we have $|a_n'| \sim |a_{n_1}'| \sim |a_{n_2}'|$. Hence, using this fact and \eqref{Convergence of sum of l using phi} from \eqref{Application of Holder in case I for 1,1 case}, we arrive at
\begin{align}
\label{Final estimate of case I for 1,1 case}
    \|\chi_{S_{n}^j} \mathcal{B}^{\alpha}_{j,1} (f_{n_1}^j,g_{n_2}^j)\|_{L^{1/2}}  &\leq C 2^{j d (1 + \varepsilon)} \sum_{l \in \mathbb{Z}} \frac{(1 + |l|)^{1 + \varepsilon} \|\phi_{j,l}^{\alpha}\|_{L^{\infty}}}{(1+|l|)^{(1+\varepsilon)}} \|f_{n_1}^j\|_{L^{1}} \|g_{n_2}^j\|_{L^{1}} \\
    &\nonumber\leq C 2^{j d (1 + \varepsilon)} 2^{-j \alpha} 2^{j \varepsilon} \|f_{n_1}^j\|_{L^{1}} \|g_{n_2}^j\|_{L^{1}} \\
    &\nonumber \leq C 2^{-j \delta} \left\| f_{n_1}^j \right\|_{L^1} \left\| g_{n_2}^j \right\|_{L^1} ,
\end{align}
where $\delta= \alpha-d-\varepsilon(1+d)>0$, since $\alpha> d$, we can choose $\varepsilon>0$ very small such that $\alpha-d-\varepsilon(1+d)>0$.

\subsection{Case-II : \texorpdfstring{$|a'_n| \leq \frac{11}{5} 2^{j(1+\varepsilon)}$}{} and \texorpdfstring{$d_1 \geq d_2$}{}}
Note that $\alpha(1,1)= d$. Using H\"older's inequality we have
\begin{align}
\label{application of Holder in case II in 1,1 case}
    \|\chi_{S_{n}^j} \mathcal{B}^{\alpha}_{j,1} (f_{n_1}^j,g_{n_2}^j)\|_{L^{1/2}} &\leq C \Big\{\int_{B (a_n, \frac{1}{5} 2^{j(1 + \varepsilon)})}  \frac{dx}{|x'|^{(\gamma_1 + \gamma_2)}}  \Big\} \| |\mathbf{P}|^{\gamma_1 + \gamma_2} \mathcal{B}^{\alpha}_{j,1} (f_{n_1}^j,g_{n_2}^j) \|_{L^1} .
\end{align}

Applying H\"older's inequality again, along with Proposition \ref{Weighted restriction estimate} for $0\leq \gamma_1, \gamma_2 <d_2/2$ and the fact \eqref{Convergence of sum of l using phi}, the equation \eqref{Equation: Decomposition in terms of Fourier series} yields
\begin{align*}
 \left\| |\mathbf{P}|^{\gamma_1 + \gamma_2} \mathcal{B}^{\alpha}_{j,1}(f_{n_1}^j, g_{n_2}^j) \right\|_{L^1} &\leq C \sum_{l \in \mathbb{Z}} \left\| |\mathbf{P}|^{\gamma_1} \phi_{j,l}^{\alpha}(\mathcal{L})f_{n_1}^j \right\|_{L^2} \left\| |\mathbf{P}|^{\gamma_2} \psi_l(\mathcal{L})g_{n_2}^j \right\|_{L^2} \\
 &\leq C \sum_{l \in \mathbb{Z}} \|\phi_{j,l}^{\alpha}\|_{L^{\infty}} \|f_{n_1}^j\|_{L^{1}} \|g_{n_2}^j\|_{L^{1}} \leq C 2^{-j\alpha} 2^{j \varepsilon} \|f_{n_1}^j\|_{L^{1}} \|g_{n_2}^j\|_{L^{1}} .
\end{align*}

Plugging the above estimate in \eqref{application of Holder in case II in 1,1 case} and using Lemma \ref{Lemma: Integral of weight over ball} for $0\leq \gamma_1+\gamma_2<d_1$, it follows that
\begin{align*}
    \|\chi_{S_{n}^j} \mathcal{B}^{\alpha}_{j,1} (f_{n_1}^j,g_{n_2}^j)\|_{L^{1/2}} &\leq C 2^{j(1 + \varepsilon)(Q-\gamma_1 - \gamma_2)} 2^{-j\alpha} 2^{j \varepsilon} \|f_{n_1}^j\|_{L^{1}} \|g_{n_2}^j\|_{L^{1}} .
\end{align*}
Since $\alpha>d$ in this case, choosing $\varepsilon>0$ very small for $0 \leq \gamma_1 + \gamma_2 < d_2$, there exists $\delta=\alpha-(Q-\gamma_1-\gamma_2)-\varepsilon(1+Q-\gamma_1-\gamma_2)>0$ such that
\begin{align*}
    \|\chi_{S_{n}^j} \mathcal{B}^{\alpha}_{j,1} (f_{n_1}^j,g_{n_2}^j)\|_{L^{1/2}} &\leq C 2^{-j\delta} \|f_{n_1}^j\|_{L^{1}} \|g_{n_2}^j\|_{L^{1}} .
\end{align*}

\subsection{Case-III : \texorpdfstring{$|a'_n| \leq \frac{11}{5} 2^{j(1+\varepsilon)}$}{} and \texorpdfstring{$d_1, d_2 \geq 1$}{}}
\label{subsection: Case III for 1,1 case}
Note that $\alpha(1,1)= d+1$. Let $\Theta$ be defined as in (\ref{Definition of chi}). Define the functions $\phi_{j,l,M_1} ^{\alpha} : \mathbb{R} \times \mathbb{R} \to \mathbb{C}$ and $\psi_{l,M_2} : \mathbb{R} \times \mathbb{R} \to \mathbb{C}$ by setting
\begin{align*}
    \phi_{j,l,M_1}^{\alpha} (\eta, \tau) = \phi_{j,l}^{\alpha} (\eta) \,\Theta_{M_1}(\tau) \quad \text{and} \quad \psi_{l,M_2}(\eta, \tau) = \psi_{l}(\eta) \,\Theta_{M_2}(\tau) .
\end{align*}
Then we can then write
\begin{align}
\label{Cutoff of M1 for the first linear multiplier}
    \phi_{j,l} ^{\alpha}(\mathcal{L}) f_{n_1}^j &= \sum_{M_1=0 } ^{\infty} \phi_{j,l,M_1}^{\alpha}(\mathcal{L}, T) f_{n_1}^j , \quad \text{and} \quad \psi_{l}(\mathcal{L}) g_{n_2}^j = \sum_{M_2=0 }^{\infty} \psi_{l,M_2}(\mathcal{L}, T) g_{n_2}^j ,
\end{align}
where we have used the fact that, because of $0\leq [k_1]|\lambda_1|, [k_2]|\lambda_2| \leq 1$, the terms in the sum will be non-zero only if $M_1, M_2 \geq 0$. As a result, this leads to the following decomposition.
\begin{align}
\label{Expression: Decomposition of Bj1 into linear}
    & \chi_{S_{n}^j}(x) \mathcal{B}^{\alpha}_{j,1}(f_{n_1}^j, g_{n_2}^j)(x) \\
    &\nonumber= \chi_{S_{n}^j}(x) \left( \sum_{M_1=0 }^{j} \sum_{M_2=0 }^{j} + \sum_{M_1=0 }^{j} \sum_{M_2=j+1 }^{\infty} + \sum_{M_1=j+1 }^{\infty} \sum_{M_2=0 }^{j} + \sum_{M_1=j+ 1 }^{\infty} \sum_{M_2=j+ 1 }^{\infty} \right) \mathcal{B}^{\alpha}_{j,1,M_1, M_2}(f_{n_1}^j, g_{n_2}^j)(x) \\
    &\nonumber:= I_1 + I_2 + I_3 + I_4,
\end{align}
where
\begin{align}
\label{Expression for bilinear Bochner-Riesz j1}
    \mathcal{B}^{\alpha}_{j,1,M_1, M_2}(f_{n_1}^j, g_{n_2}^j)(x) &= \sum_{l\in \mathbb{Z}} \phi_{j,l,M_1}^{\alpha}(\mathcal{L}, T) f_{n_1}^j(x)\, \psi_{l,M_2}(\mathcal{L}, T) g_{n_2}^j(x) .
\end{align}
In the following, we estimate each $I_1$, $I_2$, $I_3$ and $I_4$ individually.
\subsubsection{\textbf{Estimate of} \texorpdfstring{$I_{1}$}{}{}}
\label{Estimate of I1 for 1,1 case}
From the claim \eqref{claim in the estimate of Balpha j,1},  the conditions $|a_n' - a_{n_i}'|\leq 2 \cdot 2^{j(1+\varepsilon)}$ imply  $|a_{n_i}'| \leq \frac{21}{5} 2^{j(1+\varepsilon)}$ for $i=1,2$. Since   
$\supp f_{n_1}^j \subseteq B\left((a'_{n_1}, a''_{n_1}), 2^{j(1+\varepsilon)}/5\right) $ and $\supp g_{n_2}^j \subseteq B\left((a'_{n_2}, a''_{n_2}), 2^{j(1+\varepsilon)}/5\right) $, applying Lemma \ref{Lemma: Decomposition of balls for small radius} gives us the inclusion
\begin{align*}
    B\left((a'_{n_i}, a''_{n_i}), 2^{j(1+\varepsilon)}/5\right) \subseteq B^{|\cdot|}\left(a'_{n_i}, 2^{j(1+\varepsilon)}/5\right) \times B^{|\cdot|}\left(a''_{n_i}, C 2^{2j(1+\varepsilon)}/25\right) .
\end{align*}
Note that in this case $0\leq M_1, M_2 \leq j$, therefore we can decompose $\supp{f_{n_1}^j} \subseteq S_{n_1}^j$ and $\supp{g_{n_2}^j} \subseteq S_{n_2}^j$ as follows,
\begin{align}
\label{Expression: Decomposition of ball into smaller balls}
    S_{n_i}^j &= \bigcup_{m_i=1}^{N_{M_i}} S_{n_i,m_i}^{M_i,j} ,
\end{align}
where $S_{n_i,m_i}^{M_i,j} \subseteq B^{|\cdot|}\left(a_{n_i}', 2^{j(1+\varepsilon)}/5\right) \times B^{|\cdot|}\left(b_{n_i,m_i}^{M_i}, C 2^{(j+M_1)(1+\varepsilon)}/25\right)$ are disjoint subsets and for $m_i \neq m_i'$, we have $|b_{n_i,m_i}^{M_i}-b_{n_i,m_i'}^{M_i}|> C 2^{(j+M_1)(1+\varepsilon)}/50$ for $i=1,2$. Then the number of subsets $N_{M_i}$ in this decomposition is at most $C 2^{(j-M_i)(1+\varepsilon)d_2}$. Moreover, for $\gamma>0$, we also define the set
\begin{align}
\label{Defination of dilated balls}
    \widetilde{S}_{n, n_i,m_i}^{M_i,j} &:= B^{|\cdot|}\left(a_{n}', 2^{j(1+\varepsilon)}/5 \right) \times B^{|\cdot|}\left(b_{n_i,m_i}^{M_i}, C 2^{(j+M_1)(1+\varepsilon)} 2^{\gamma j+1}/25\right) .
\end{align}
With the aid of the above decomposition, correspondingly we also decompose $f_{n_1}^j$ and $g_{n_2}^j$ as follows 
\begin{align}
\label{Decomposition of f and g}
    f_{n_1}^j = \sum_{m_1=1}^{N_{M_1}} f_{n_1,m_1}^{M_1,j} \quad \text{and} \quad g_{n_2}^j = \sum_{m_2=1}^{N_{M_2}} g_{n_2,m_2}^{M_2,j} .
\end{align}
where $f_{n_1,m_1}^{M_1,j} = f_{n_1}^j \chi_{B_{n_1,m_1}^{M_1,j}} $ and $g_{n_2,m_2}^{M_2,j} = g_{n_2}^j \chi_{B_{n_2,m_2}^{M_2,j}}$.

Given the decomposition outlined above, we can now break $I_1$ as
\begin{align}
\label{Decomposition of S1 into S11 and S12 in (1,1,1/2)}
    & I_{1} = \sum_{M_1=0 }^{j} \sum_{M_2=0 }^{j} \sum_{m_1=1}^{N_{M_1}} \sum_{m_2=1}^{N_{M_2}} \chi_{S_{n}^j}(x) \chi_{\widetilde{S}_{n, n_1,m_1}^{M_1,j}}(x) \chi_{\widetilde{S}_{n, n_2,m_2}^{M_2,j}}(x) \mathcal{B}^{\alpha}_{j,1,M_1, M_2}(f_{n_1,m_1}^{M_1,j}, g_{n_2, m_2}^{M_2,j})(x) \\
    &\nonumber+ \sum_{M_1=0 }^{j} \sum_{M_2=0 }^{j} \sum_{m_1=1}^{N_{M_1}} \sum_{m_2=1}^{N_{M_2}} \chi_{S_{n}^j}(x) \chi_{\widetilde{S}_{n, n_1,m_1}^{M_1,j}}(x) (1-\chi_{\widetilde{S}_{n, n_2,m_2}^{M_2,j}})(x) \mathcal{B}^{\alpha}_{j,1,M_1, M_2}(f_{n_1,m_1}^{M_1,j}, g_{n_2,m_2}^{M_2,j})(x) \\
    &\nonumber+ \sum_{M_1=0 }^{j} \sum_{M_2=0 }^{j} \sum_{m_1=1}^{N_{M_1}} \chi_{S_{n}^j}(x) (1-\chi_{\widetilde{S}_{n, n_1,m_1}^{M_1,j}})(x) \mathcal{B}^{\alpha}_{j,1,M_1, M_2}(f_{n_1,m_1}^{M_1,j}, g_{n_2}^{j})(x) \\
    &\nonumber=: I_{11} + I_{12} + I_{13} .
\end{align}

\subsubsection{\textbf{Estimate of} \texorpdfstring{$I_{11}$}{}{}}
Because of \eqref{Expression for bilinear Bochner-Riesz j1}, we can express
\begin{align*}
    I_{11} &=  \sum_{M_1=0 }^{j} \sum_{M_2=0 }^{j} \sum_{l \in \mathbb{Z}} \Big\{ \sum_{m_1=1}^{N_{M_1}}  \chi_{\widetilde{S}_{n,n_1,m_1}^{M_1,j}}(x)  \phi_{j,l,M_1}^{\alpha}(\mathcal{L}, T) f_{n_1,m_1}^{M_1,j}(x) \Big\} \\
    & \hspace{6cm} \Big\{ \sum_{m_2=1}^{N_{M_2}} \chi_{\widetilde{S}_{n,n_2,m_2}^{M_2,j}}(x) \psi_{l,M_2}(\mathcal{L}, T) g_{n_2,m_2}^{M_2,j}(x) \Big\} .
    %&=: \sum_{M_1=0 }^{j} \sum_{M_2=0 }^{j} \sum_{l \in \mathbb{Z}} F_{l, M_1}(x) G_{l,M_2}(x)
\end{align*}
Firs note that for $0<p<1$, $\|\cdot\|_{L^p}^p$ satisfies
\begin{align}
\label{Triangle inequality for p less than 1 case}
    \|f+g\|_{L^p}^p &\leq \|f\|_{L^p}^p + \|g\|_{L^p}^p .
\end{align}
With the help of the above observation and H\"older's inequality, we obtain 
\begin{align}
\label{Estimate: Estimate of S_{111}}
     \|I_{11}\|_{L^{1/2}}^{1/2} &\leq \sum_{M_1=0 }^{j} \sum_{M_2=0 }^{j} \sum_{l \in \mathbb{Z} } \Big\| \sum_{m_1=1}^{N_{M_1}}  \chi_{\widetilde{S}_{n,n_1,m_1}^{M_1,j}}(x)  \phi_{j,l,M_1}^{\alpha}(\mathcal{L}, T) f_{n_1,m_1}^{M_1,j}\Big\|_{L^{1}}^{1/2} \\
     &\nonumber \hspace{4cm} \Big\| \sum_{m_2=1}^{N_{M_2}} \chi_{\widetilde{S}_{n,n_2,m_2}^{M_2,j}}(x) \psi_{l,M_2}(\mathcal{L}, T) g_{n_2,m_2}^{M_2,j}\Big\|_{L^{1}}^{1/2} .
\end{align}
Now, invoking H\"older's inequality, \eqref{Defination of dilated balls} and Proposition \ref{Lemma: Martini_Mullar_Weighted_Plancherel}, it immediately follows that 
\begin{align}
\label{Estimate of FlM1 for (1,2,2/3)}
    & \Big\| \sum_{m_1=1}^{N_{M_1}}  \chi_{\widetilde{S}_{n,n_1,m_1}^{M_1,j}}(x)  \phi_{j,l,M_1}^{\alpha}(\mathcal{L}, T) f_{n_1,m_1}^{M_1,j}\Big\|_{L^{1}} \leq C \sum_{m_1=1}^{N_{M_1}} |\widetilde{S}_{n,n_1,m_1}^{M_1,j}|^{1/2} \|\phi_{j,l,M_1}^{\alpha}(\mathcal{L}, T) f_{n_1,m_1}^{M_1,j} \|_{L^2} \\
    &\nonumber \leq C \sum_{m_1=1}^{N_{M_1}} 2^{j(1+ \varepsilon)d_1/2}   2^{(j+M_1)(1+\varepsilon)d_2/2} 2^{\gamma j d_2/2} 2^{-M_1 d_2/2} \|\phi_{j,l}^{\alpha}\|_{L^{\infty}} \| f_{n_1,m_1}^{M_1,j} \|_{L^1} \\
    &\nonumber \leq C 2^{\varepsilon M_1 d_2/2} 2^{\gamma j d_2/2} 2^{j(1+ \varepsilon)d/2} \|\phi_{j,l}^{\alpha}\|_{L^{\infty}} \| f_{n_1}^j\|_{L^1} .
\end{align}
By a similar reasoning, we also get
\begin{align}
\label{Estimate of GlM1 for (1,2,2/3)}
    \Big\| \sum_{m_2=1}^{N_{M_2}} \chi_{\widetilde{S}_{n,n_2,m_2}^{M_2,j}}(x) \psi_{l,M_2}(\mathcal{L}, T) g_{n_2,m_2}^{M_2,j}\Big\|_{L^{1}} &\leq C 2^{\varepsilon M_2 d_2/2} 2^{\gamma j d_2/2} 2^{j(1+ \varepsilon)d/2} \| g_{n_2}^j\|_{L^1} .
\end{align}
Plugging the estimates \eqref{Estimate of FlM1 for (1,2,2/3)} and \eqref{Estimate of GlM1 for (1,2,2/3)} into  (\ref{Estimate: Estimate of S_{111}}) yields
\begin{align}
\label{After plugging the estimate Fl and Gl}
    \|I_{11}\|_{L^{1/2}} ^{1/2} &\leq C \sum_{M_1=0 }^{j} \sum_{M_2=0 }^{j} \sum_{l \in \mathbb{Z}} 2^{\varepsilon (M_1 +M_2) d_2/4} 2^{\gamma j d_2/2} 2^{j(1+ \varepsilon)d/2} \|\phi_{j,l}^{\alpha}\|_{L^{\infty}}^{1/2} \| f_{n_1}^j\|_{L^1}^{1/2} \| g_{n_2}^j \|_{L^1}^{1/2} \\
    &\nonumber \leq C \sum_{l \in \mathbb{Z}} 2^{C \epsilon_1 j} 2^{\gamma j d_2/2} 2^{j d/2} \|\phi_{j,l}^{\alpha}\|_{L^{\infty}}^{1/2} \| f_{n_1}^j\|_{L^1}^{1/2} \| g_{n_2}^j \|_{L^1}^{1/2},
\end{align}
for some $\epsilon_1>0$ depending on $\varepsilon>0$. Also notice that using (\ref{Convergence of sum of l using phi}) and summing over $l \in \mathbb{Z}$, gives us
\begin{align*}
    \sum_{l \in \mathbb{Z}} \|\phi_{j,l}^{\alpha}\|_{L^{\infty}}^{1/2} &\leq \sum_{l \in \mathbb{Z}} \frac{1}{(1+|l|)^{1+\varepsilon}} \{(1+|l|)^{2+2\varepsilon} \|\phi_{j,l}^{\alpha}\|_{L^{\infty}} \}^{1/2} \leq C 2^{-j \alpha/2} 2^{j(1/2+\varepsilon)} .
\end{align*}
Finally, putting the above estimate in \eqref{After plugging the estimate Fl and Gl} and since $\alpha>d+1$, we can choose $\varepsilon, \epsilon_1, \gamma>0$ small enough such that
\begin{align*}
    \|I_{11}\|_{L^{1/2}} &\leq C 2^{-j \delta} \| f_{n_1}^j\|_{L^1} \| g_{n_2}^j \|_{L^1} ,
\end{align*}
where $\delta=\alpha-(d+1)-2C\epsilon_1-2\varepsilon-\gamma d_2>0$.

\subsubsection{\textbf{Estimate of} \texorpdfstring{$I_{12}$}{}{}}
\label{subsection: estimate of I12 for 1,1 case}
In this case, with the help of Proposition \ref{Lemma: Martini_Mullar_Weighted_Plancherel}, we show that $I_{12}$ has an arbitrary large decay. Note that since $S_{n}^j \subseteq B(a_n, 2^{j(1+\varepsilon)}/5)$ and $|a'_n| \leq \frac{11}{5} 2^{j(1+\varepsilon)}$, aplying \eqref{Estimate : Ball volume} yields $|S_{n}^j| \lesssim 2^{j(1+\varepsilon)Q}$. Hence, from \eqref{Decomposition of S1 into S11 and S12 in (1,1,1/2)}, and applying H\"older's inequality we obtain 
\begin{align}
\label{Estimate: Estimate of S_{12} in 1,1}
    \|I_{12}\|_{L^{1/2}} & \leq C 2^{j(1+\varepsilon)Q} \sum_{M_1=0 }^{j} \sum_{M_2=0 }^{j} \sum_{m_1=1}^{N_{M_1}} \sum_{m_2=1}^{N_{M_2}}  \big\|\chi_{S_{n}^j} (1-\chi_{\widetilde{S}_{n, n_2,m_2}^{M_2,j}}) \mathcal{B}^{\alpha}_{j,1,M_1, M_2}(f_{n_1,m_1}^{M_1,j}, g_{n_2,m_2}^{M_2,j}) \big\|_{L^1} .
\end{align}
Recalling the expression \eqref{Expression for bilinear Bochner-Riesz j1} and again using H\"older's inequality, we see that
\begin{align}
\label{Estimate: Application of Holder inequality for error part in 1,1}
    & \|\chi_{S_{n}^j} (1-\chi_{\widetilde{S}_{n, n_2,m_2}^{M_2,j}}) \mathcal{B}^{\alpha}_{j,1,M_1, M_2}(f_{n_1,m_1}^{M_1,j}, g_{n_2,m_2}^{M_2,j})\|_{L^{1}} \\
    &\nonumber \leq C \sum_{l \in \mathbb{Z}} \|\phi_{j,l,M_1} ^{\alpha}(\mathcal{L}, T) f_{n_1,m_1}^{M_1,j}\|_{L^2} \|\chi_{S_{n}^j} (1-\chi_{\widetilde{S}_{n, n_2,m_2}^{M_2,j}})\psi_{l,M_2}(\mathcal{L}, T) g_{n_2,m_2}^{M_2,j}\|_{L^2} .
\end{align}
In the following, we show that the term $\|\chi_{S_{n}^j} (1-\chi_{\widetilde{S}_{n, n_2,m_2}^{M_2,j}})\psi_{l,M_2}(\mathcal{L}, T) g_{n_2,m_2}^{M_2,j}\|_{L^2}$ gives arbitrary large decay in this scenario. An application of Minkowski's integral inequality yields
\begin{align}
\label{Estimate: Application of Minkowski inequality for error part in 1,1}
    & \|\chi_{S_{n}^j} (1-\chi_{\widetilde{S}_{n, n_2,m_2}^{M_2,j}})\psi_{l,M_2}(\mathcal{L}, T) g_{n_2,m_2}^{M_2,j}\|_{L^2} \\
    &\nonumber \leq \int_{\mathbb{R}^d} |g_{n_2,m_2}^{M_2,j}(z)| \left( \int_{\mathbb{R}^d} |\chi_{S_{n}^j}(x) (1-\chi_{\widetilde{S}_{n, n_2,m_2}^{M_2,j}})(x) \mathcal{K}_{\psi_{l,M_2}(\mathcal{L}, T)}(x,z)|^2 \ dx \right)^{1/2} dz .
\end{align}
From \eqref{Defination of dilated balls}, we see that if $x \in \supp{\chi_{S_{n}^j} (1-\chi_{\widetilde{S}_{n, n_2,m_2}^{M_2,j}})}$, then $|x''-b_{n_2,m_2}^{M_2}| \geq C 2^{\gamma j+1} 2^{(j+M_2)(1+\varepsilon)}$, and if $z \in \supp{g_{n_2,m_2}^{M_2,j}}$, then $|z''-b_{n_2,m_2}^{M_2}| \leq C 2^{(j+M_2)(1+\varepsilon)}$. These in particular imply $|x''-z''| \geq C 2^{\gamma j} 2^{(j+M_2)(1+\varepsilon)}$. Consequently, using Proposition \ref{Lemma: Martini_Mullar_Weighted_Plancherel}, we derive
\begin{align}
\label{Estimate: Weighted plancherel on outside ball in (1,1,1/2)}
    & \left( \int_{\mathbb{R}^d} |\chi_{S_{n}^j}(x) (1-\chi_{\widetilde{S}_{n, n_2,m_2}^{M_2,j}})(x) \mathcal{K}_{\psi_{l,M_2}(\mathcal{L}, T)}(x,z)|^2 \ dx \right)^{1/2} \\
    % &= \left( \int_{B^{y''}} | K_{\phi_{j,l,M_1} ^{\alpha}(\mathcal{L}, T)}(x,y)|^2 \ dx \right)^{1/2} \\
    &\leq C (2^{\gamma j} 2^{(j+M_2)(1+\varepsilon)} )^{-N} \left( \int_{\mathbb{R}^d} | |x''-z''|^N \mathcal{K}_{\psi_{l,M_2}(\mathcal{L}, T)}(x,z)|^2 \ dx \right)^{1/2}\nonumber \\
    &\leq C (2^{\gamma j} 2^{(j+M_2)(1+\varepsilon)} )^{-N} 2^{M_2(N-d_2/2)} \|\psi_{l}\|_{L^2_N}\nonumber \\
    &\leq C 2^{-\gamma j N} 2^{-j N} |l|^N \nonumber.
\end{align}
Therefore, plugging the estimate \eqref{Estimate: Weighted plancherel on outside ball in (1,1,1/2)} into (\ref{Estimate: Application of Minkowski inequality for error part in 1,1}) gives us
\begin{align}
\label{Estimate: Outside of M1 ball for f in (1,1,1/2)}
    \|\chi_{S_{n}^j} (1-\chi_{\widetilde{S}_{n, n_2,m_2}^{M_2,j}})\psi_{l,M_2}(\mathcal{L}, T) g_{n_2,m_2}^{M_2,j}\|_{L^2} &\leq C 2^{-\gamma j N} 2^{-j N} |l|^N \|g_{n_2,m_2}^{M_2,j}\|_{L^1} .
\end{align}
Taking the above estimate into account and Proposition \ref{Lemma: Martini_Mullar_Weighted_Plancherel}, we see from (\ref{Estimate: Application of Holder inequality for error part in 1,1}) that
\begin{align}
\label{Estimate: Estimate of error part wrt f for (1,1,1/2)}
    & \|\chi_{S_{n}^j} (1-\chi_{\widetilde{S}_{n, n_2,m_2}^{M_2,j}}) \mathcal{B}^{\alpha}_{j,1,M_1, M_2}(f_{n_1,m_1}^{M_1,j}, g_{n_2,m_2}^{M_2,j})\|_{L^{1}} \\
    &\nonumber \leq C \sum_{l \in \mathbb{Z}} 2^{-M_1 d_2/2} \|\phi_{j,l} ^{\alpha}\|_{L^{\infty}} \|f_{n_1,m_1}^{M_1,j}\|_{L^1} 2^{-\gamma j N} 2^{-j N} |l|^N \|g_{n_2,m_2}^{M_2,j}\|_{L^1} \\
    &\nonumber \leq C 2^{j \varepsilon} 2^{-j \alpha} 2^{-\gamma j N} \|f_{n_1,m_1}^{M_1,j}\|_{L^1} \|g_{n_2,m_2}^{M_2,j}\|_{L^1} ,
\end{align}
where we have used $2^{-j N} \sum_{l \in \mathbb{Z}} (1+|l|)^{-(1+\varepsilon)} \{(1+|l|)^{1+\varepsilon} |l|^N \|\phi_{j,l} ^{\alpha}\|_{L^{\infty}}\} \leq C 2^{-j \alpha} 2^{j \varepsilon}$ (see (\ref{Convergence of sum of l using phi})).

Finally, choosing $N$ large enough and combining the estimates (\ref{Estimate: Estimate of S_{12} in 1,1}) and \eqref{Estimate: Estimate of error part wrt f for (1,1,1/2)}, we can conclude that
\begin{align}
\label{Estimate of S_12 for (1,1,1/2)}
    \|I_{12}\|_{L^{1/2}} & \leq C  2^{j(1+\varepsilon)Q} 2^{j \varepsilon} 2^{-j \alpha} 2^{-\gamma j N} \sum_{M_1=0 }^{j} \sum_{M_2=0 }^{j} \sum_{m_1=1}^{N_{M_1}} \|f_{n_1,m_1}^{M_1,j}\|_{L^1}  \sum_{m_2=1}^{N_{M_2}} \|g_{n_2,m_2}^{M_2,j}\|_{L^1}  \\
    & \nonumber \leq C 2^{-j \delta} \|f_{n_1}^j\|_{L^1} \|g_{n_2}^j\|_{L^1} ,
    %&\nonumber \leq C 2^{-j \delta} \|f_{n_1}\|_{L^1}^{1/2} \|g_{n_2}\|_{L^1}^{1/2} ,
\end{align}
for some $\delta>0$.

This completes the estimate of $I_{12}$. The estimation of $I_{13}$ follows similarly as of $I_{12}$ (see subsection \ref{subsection: estimate of I12 for 1,1 case}) with obvious modification, therefore we omit the details and with this estimate of $I_1$ is completed. Since $I_2$ and $I_3$ are symmetric with respect to $M_1$ and $M_2$ (see \eqref{Expression: Decomposition of Bj1 into linear}), it is enough to estimate $I_2$. Therefore, it remains to estimate $I_2$ and $I_4$. We begin with the estimation of $I_4$.

\subsubsection{\textbf{Estimate of} \texorpdfstring{$I_{4}$}{}{}}
\label{Estimate of I4 for 1,1 case}
To estimate $I_4$ (see \eqref{Expression: Decomposition of Bj1 into linear}), as in \eqref{Estimate: Estimate of S_{12} in 1,1} and \eqref{Estimate: Application of Holder inequality for error part in 1,1}, application of H\"older's inequality yields the following
\begin{align*}
    \|I_{4}\|_{L^{1/2}} & \leq C 2^{jQ(1+\varepsilon)} \sum_{l \in \mathbb{Z}} \sum_{M_1=j+ 1 }^{\infty} \sum_{M_2=j+ 1 }^{\infty} \| \phi_{j,l, M_1}^{\alpha}(\mathcal{L}, T)f_{n_1}^j \|_{L^2} \| \psi_{l,M_2}(\mathcal{L}, T) g_{n_2}^j \|_{L^2} .
\end{align*}
Hence, using Proposition \ref{Lemma: Martini_Mullar_Weighted_Plancherel} and summing over $l \in \mathbb{Z}$ with the estimate from \eqref{Convergence of sum of l using phi}, we get
\begin{align}
\label{Estimate of S4 in (1,1,1/2)}
    \|I_{4}\|_{L^{1/2}} & \leq C \sum_{l \in \mathbb{Z}} \sum_{M_1=j+ 1 }^{\infty} \sum_{M_2=j+ 1 }^{\infty} 2^{jQ(1+\varepsilon)} 2^{-M_1 d_2/2} \|\phi_{j,l}^{\alpha}\|_{L^{\infty}} \| f_{n_1}^j\|_{L^1} 2^{-M_2 d_2/2} \| g_{n_2}^j \|_{L^1} \\
    &\nonumber \leq C 2^{j\varepsilon Q} 2^{-j \alpha} 2^{jd} \| f_{n_1}^j\|_{L^1} \| g_{n_2}^j\|_{L^1} \\
    &\nonumber \leq C 2^{-j \delta} \| f_{n_1}^j\|_{L^1} \| g_{n_2}^j\|_{L^1} ,
\end{align}
where $\delta=\alpha-d-\varepsilon Q>0$, since $\alpha>d$.

\subsubsection{\textbf{Estimate of} \texorpdfstring{$I_{2}$}{}{}}
To estimate $I_2$, we split it into two parts as follows:
\begin{align*}
    I_2 &= \chi_{S_{n}^j}(x) \left( \sum_{M_1=0 }^{j} \sum_{M_2=j+1 }^{2j} + \sum_{M_1=0 }^{j} \sum_{M_2=2j+ 1 } ^{\infty} \right) \mathcal{B}^{\alpha}_{j,1,M_1, M_2}(f_{n_1}^j, g_{n_2}^j)(x) \\
    &= I_{21} + I_{22} .
\end{align*}

The estimation of $I_{22}$ follows similarly to the  estimation of $I_4$ (see \eqref{Estimate of I4 for 1,1 case}), utilizing the fact $M_2>2j$. While the estimation of $I_{21}$ can be handled similar to the estimate of $I_1$ (see \eqref{Estimate of I1 for 1,1 case}), with the key difference being that in this case,  we do not decompose $\supp g_{n_2}^j \subseteq S_{n_2}^j$ (see \eqref{Expression: Decomposition of ball into smaller balls}). This is because in $I_{21}$, $ M_2$ is bigger than $j$, and the support of $g_{n_2}^j$ is already contained in $B\left(a_{n_2}, 2^{j(1+\varepsilon)}/5\right)$. With this modification, the proof of $I_{21}$ can be completed analogously, as the estimate of $I_1$ (see \eqref{Estimate of I1 for 1,1 case}).

This completes the proof of claim \eqref{claim in the estimate of Balpha j,1} for $(p_1, p_2,p)=(1,1,1/2)$.

\section{Proof of the claim (\ref{claim in the estimate of Balpha j,1}) for \texorpdfstring{$(p_1,p_2,p)=(1,2,2/3)$}{}}
The structure of the proof follows a similar approach to the one already done for the point $(p_1,p_2,p)=(1,1,1/2)$ (see Section \ref{Section: Proof of claim for 1,1 case}). Accordingly,  we divide the proof into the following three cases.

\subsection{Case-I : \texorpdfstring{$|a'_n| > \frac{11}{5} 2^{j(1+\varepsilon)}$}{}}
Note that $\alpha(1,2)= \min\{D, d_1+d_2+1\}/2>d/2$. Applying H\"older's inequality twice, along with the fact \eqref{Estimate : Ball volume}, we deduce 
from \eqref{Equation: Decomposition in terms of Fourier series} that
\begin{align*}
    \|\chi_{S_{n}^j} \mathcal{B}^{\alpha}_{j,1} (f_{n_1}^j,g_{n_2}^j)\|_{L^{2/3}} &\leq C 2^{j d (1 + \varepsilon)/2} |a_n'|^{d_2/2} \sum_{l \in \mathbb{Z}} \| \phi_{j,l}^{\alpha}(\mathcal{L})f_{n_1}^j \|_{L^2} \| \psi_l(\mathcal{L})g_{n_2}^j \|_{L^2} .
\end{align*}
Note that, as observed in subsection \ref{subsection: case I of 1,1 case}, we have $|a_n'| \sim |a_{n_1}'|$. Therefore, applying Proposition \ref{Weighted restriction estimate} with $\gamma=0$ and estimate \eqref{Convergence of sum of l using phi}, we obtain 
\begin{align*}
    \|\chi_{S_{n}^j} \mathcal{B}^{\alpha}_{j,1} (f_{n_1}^j,g_{n_2}^j)\|_{L^{2/3}} &\leq C 2^{j d (1 + \varepsilon)/2} |a_n'|^{d_2/2} \sum_{l \in \mathbb{Z}} |a_{n_1}'|^{- d_2/2} \|\phi_{j,l} ^{\alpha}\|_{L^{\infty} } \|f_{n_1}^j\|_{L^{1}} \|g_{n_2}^j\|_{L^{2}} \\
    &\leq C 2^{-j \alpha} 2^{j \varepsilon} 2^{j d (1 + \varepsilon)/2} \|f_{n_1}^j\|_{L^{1}} \|g_{n_2}^j\|_{L^{2}} .
\end{align*}
The condition $\alpha > d/2$ allows us to choose $\varepsilon>0$ small enough to ensure that $\delta=\alpha-d/2-\varepsilon(1+d/2)>0$. As a result, we obtain the bound  
\begin{align*}
    \|\chi_{S_{n}^j} \mathcal{B}^{\alpha}_{j,1} (f_{n_1}^j,g_{n_2}^j)\|_{L^{2/3}} \leq C 2^{-j \delta} \|f_{n_1}^j\|_{L^{1}} \|g_{n_2}^j\|_{L^{2}} .
\end{align*}

\subsection{Case-II : \texorpdfstring{$|a'_n| \leq \frac{11}{5} 2^{j(1+\varepsilon)}$}{} and \texorpdfstring{$d_1 \geq d_2$}{}}
Note that $\alpha(1,2)= d/2$. In this case, using first-layer weight and applying H\"older's inequality leads to
\begin{align*}
    \|\chi_{S_{n}^j} \mathcal{B}^{\alpha}_{j,1} (f_{n_1}^j,g_{n_2}^j)\|_{L^{2/3}} &\leq C \Big(\int_{B (a_n, \frac{1}{5} 2^{j(1 + \varepsilon)})} \frac{dx}{|x'|^{2\gamma}}  \Big)^{1/2} \| |\mathbf{P}|^{\gamma} \mathcal{B}^{\alpha}_{j,1} (f_{n_1}^j,g_{n_2}^j)\|_{L^1} .
\end{align*}
Now, an use of H\"older's inequality, Proposition \ref{Weighted restriction estimate} for $0\leq \gamma <d_2/2$ and \eqref{Convergence of sum of l using phi} yields 
\begin{align*}
 \| |\mathbf{P}|^{\gamma} \mathcal{B}^{\alpha}_{j,1}(f_{n_1}^j, g_{n_2}^j) \|_{L^1} &\leq C \sum_{l \in \mathbb{Z}} \left\||\mathbf{P}|^{\gamma} \phi_{j,l}^{\alpha}(\mathcal{L})f_{n_1}^j \right\|_{L^2} \left\| \psi_l(\mathcal{L})g_{n_2}^j \right\|_{L^2} \\
 %&\leq C \sum_{l \in \mathbb{Z}} \|\phi_{j,l}^{\alpha}\| \|f_{n_1}^j\|_{L^{1}} \|g_{n_2}^j\|_{L^{2}} \\
 &\leq C 2^{-j \alpha} 2^{j \varepsilon} \|f_{n_1}^j\|_{L^{1}} \|g_{n_2}^j\|_{L^{2}} .
\end{align*}
Taking into account the above two estimates and invoking Lemma \ref{Lemma: Integral of weight over ball} for $0\leq \gamma<d_2/2$, we are led to conclude that
\begin{align*}
    \|\chi_{S_{n}^j} \mathcal{B}^{\alpha}_{j,1} (f_{n_1}^j,g_{n_2}^j)\|_{L^{2/3}}   &\leq C 2^{j(1 + \varepsilon)(Q/2-\gamma)} 2^{-j \alpha} 2^{j \varepsilon} \|f_{n_1}^j\|_{L^{1}} \|g_{n_2}^j\|_{L^{2}} \\
    &\leq C 2^{-j \delta} \|f_{n_1}^j\|_{L^{1}} \|g_{n_2}^j\|_{L^{2}} ,
\end{align*}
where $\delta=\alpha-(Q/2-\gamma)-\varepsilon(1+Q/2-\gamma)$, as $\alpha > d/2$, we can choose $\varepsilon>0$ so small and $0 \leq \gamma < d_2/2$ such that $\alpha-(Q/2-\gamma)-\varepsilon(1+Q/2-\gamma)>0$.

\subsection{Case-III : \texorpdfstring{$|a'_n| \leq \frac{11}{5} 2^{j(1+\varepsilon)}$}{} and \texorpdfstring{$d_1, d_2 \geq 1$}{}}

Note that $\alpha(1,2)= (d+1)/2$. This proof closely resembles to Case-III of $(1,1,1/2)$ (see subsection \ref{subsection: Case III for 1,1 case}). Using the same decomposition as in (\ref{Cutoff of M1 for the first linear multiplier}), but applied only to $\phi_{j,l}^{\alpha}$, we arrive at
\begin{align}
\label{Decomposition of phi into M1 cutoff}
    & \chi_{S_{n}^j}(x) \mathcal{B}^{\alpha}_{j,1}(f_{n_1}^j, g_{n_2}^j)(x) \\
    &\nonumber = C \chi_{S_{n}^j}(x) \Big(\sum_{M_1=0 }^{j} + \sum_{M_1=j+1 }^{\infty} \Big) \sum_{l \in \mathbb{Z}} \phi_{j,l,M_1}^{\alpha}(\mathcal{L}, T) f_{n_1}(x) \psi_{l}(\mathcal{L}) g_{n_2}(x) =: J_1 + J_2 .
\end{align}

\subsubsection{\textbf{Estimate of} \texorpdfstring{$J_{1}$}{}{}}
\label{subsection: Estimate of J1 case}
To estimate $J_1$, similar to \eqref{Expression: Decomposition of ball into smaller balls}- \eqref{Decomposition of f and g}, but here we only decompose the support of $f_{n_1}^j$ and the corresponding function itself. Consequently, analogous to \eqref{Decomposition of S1 into S11 and S12 in (1,1,1/2)}, we break $J_1$ into two parts as
\begin{align*}
    J_{1} &= \sum_{M_1=0 }^{j} \sum_{m_1=1}^{N_{M_1}} \sum_{l\in \mathbb{Z}} \chi_{S_{n}^j}(x) \chi_{\widetilde{S}_{n,n_1,m_1}^{M_1,j}}(x) \phi_{j,l,M_1}^{\alpha}(\mathcal{L}, T) f_{n_1, m_1}^{M_1,j}(x)\, \psi_{l}(\mathcal{L}) g_{n_2}^j(x) \\
    &+ \sum_{M_1=0 }^{j}  \sum_{m_1=1}^{N_{M_1}} \sum_{l\in \mathbb{Z}} \chi_{S_{n}^j}(x) (1-\chi_{\widetilde{S}_{n,n_1,m_1}^{M_1,j}})(x) \phi_{j,l,M_1}^{\alpha}(\mathcal{L}, T) f_{n_1, m_1}^{M_1,j}(x)\, \psi_{l}(\mathcal{L}) g_{n_2}^j(x) =: J_{11} + J_{12} .
\end{align*}

\subsubsection{\textbf{Estimate of} \texorpdfstring{$J_{11}$}{}{}}
An application of the fact \eqref{Triangle inequality for p less than 1 case} along with the H\"older's inequality gives
\begin{align*}
    \|J_{11}\|_{L^{2/3}}^{2/3} &\leq \sum_{M_1=0 }^{j} \sum_{l \in \mathbb{Z}} \Big\| \sum_{m_1=1}^{N_{M_1}} \chi_{\widetilde{S}_{n,n_1,m_1}^{M_1,j}}  \phi_{j,l,M_1}^{\alpha}(\mathcal{L}, T) f_{n_1,m_1}^{M_1,j} \Big\|_{L^1}^{2/3} \|\psi_{l}(\mathcal{L}) g_{n_2,j} \|_{L^2}^{2/3} .
\end{align*}
In turn, from the estimate (\ref{Estimate of FlM1 for (1,2,2/3)}) we get
\begin{align}
\label{Estimate of S111 for (1,2,2/3) for triangle inequality}
    \|J_{11}\|_{L^{2/3}} & \leq C \left(\sum_{M_1=0 }^{j} \sum_{l \in \mathbb{Z}} 2^{\varepsilon M_1 d_2/3} 2^{\gamma j d_2/3} 2^{j(1+ \varepsilon)d/3} \|\phi_{j,l}^{\alpha}\|_{L^{\infty} }^{2/3} \| f_{n_1}^j\|_{L^1}^{2/3} \| g_{n_2}^j\|_{L^{2}}^{2/3} \right)^{3/2}    \\
    &\nonumber \leq C 2^{j(1+ \varepsilon)d/2} 2^{\gamma j d_2/2} \| f_{n_1}^j\|_{L^1}  \| g_{n_2}^j\|_{L^{2}} \left(\sum_{M_1=0 }^{j} 2^{\varepsilon M_1 d_2/3} \sum_{l \in \mathbb{Z}} \|\phi_{j,l}^{\alpha}\|_{L^{\infty} } ^{2/3}  \right)^{3/2} .
\end{align}
 In the light of (\ref{Convergence of sum of l using phi}), we can observe that
\begin{align}
\label{sum of l for 1,2 case}
    \sum_{l \in \mathbb{Z}} \|\phi_{j,l}^{\alpha}\|_{L^{\infty}}^{2/3} &\leq \sum_{l \in \mathbb{Z}} \frac{1}{(1+|l|)^{1+\varepsilon}} \{(1+|l|)^{3/2+3\varepsilon/2} \|\phi_{j,l}^{\alpha}\|_{L^{\infty}} \}^{2/3} \leq C 2^{-j 2\alpha/3} 2^{j(1/3+\varepsilon)} .
\end{align}
Thus, plugging the above estimate into \eqref{Estimate of S111 for (1,2,2/3) for triangle inequality} yields
\begin{align*}
    \|J_{11}\|_{L^{2/3}} &\leq C 2^{C \epsilon_1 j} 2^{\gamma j d_2/2} 2^{j d/2} 2^{-j \alpha} 2^{j(1/2+3\varepsilon/2)} \|f_{n_1}^j\|_{L^1} \|g_{n_2}^j\|_{L^2} \\
    &\leq C 2^{-j \delta} \|f_{n_1}^j\|_{L^1} \|g_{n_2}^j\|_{L^2} ,
\end{align*}
where $\epsilon_1$ depends on $\varepsilon$ and $\delta=\alpha-(d+1)/2-\gamma d_2/2-C \epsilon_1-3\varepsilon/2>0$. Since $\alpha>(d+1)/2$, we can choose $\gamma$, $\varepsilon$, $\epsilon_1>0$ sufficiently small such that $\alpha-(d+1)/2-\gamma d_2/2-C \epsilon_1-3\varepsilon/2>0$.

\subsubsection{\textbf{Estimate of} \texorpdfstring{$J_{12}$}{}{}}
\label{Estimate of J12 for 1,2 case} The estimate for $J_{12}$ is similar to that of $I_{12}$ (see \eqref{subsection: estimate of I12 for 1,1 case}) or $I_{13}$, so we will be very brief. Applying H\"older's inequality twice yields
\begin{align*}
    \|J_{12}\|_{L^{2/3}} & \leq C 2^{j(1+\varepsilon)Q/2} \sum_{M_1=0 }^{j}  \sum_{m_1=1}^{N_{M_1}} \sum_{l \in \mathbb{Z}} \|\chi_{S_{n}^j} (1-\chi_{\widetilde{S}_{n,n_1,m_1}^{M_1,j}}) \phi_{j,l,M_1}^{\alpha}(\mathcal{L}, T) f_{n_1,m_1}^{M_1,j} \|_{L^2} \| \psi_{l}(\mathcal{L}) g_{n_2}^j \|_{L^2} .
\end{align*}
The $L^2$-norm on the right hand side corresponding to $f_{n_1,m_1}^{M_1,j}$ can be calculated in a similar manner to \eqref{Estimate: Outside of M1 ball for f in (1,1,1/2)}. Indeed, here for any $N>0$ we obtain
\begin{align}
\label{L1 to L2 norm for phijl on outside}
    \|\chi_{S_{n}^j} (1-\chi_{\widetilde{S}_{n,n_1,m_1}^{M_1,j}}) \phi_{j,l,M_1}^{\alpha}(\mathcal{L}, T) f_{n_1,m_1}^{M_1,j} \|_{L^2} &\leq C 2^{-\gamma j N} \|\phi_{j,l}^{\alpha}\|_{L^{\infty}} \|f_{n_1,m_1}^{M_1,j}\|_{L^1} .
\end{align}
Therefore, summing over $l \in \mathbb{Z}$ via \eqref{Convergence of sum of l using phi}, we can choose $N>0$ sufficiently large such that there exists a $\delta>0$ and
\begin{align*}
    \|J_{12}\|_{L^{2/3}} &\leq C 2^{-j \delta} \|f_{n_1}^{j}\|_{L^1} \|g_{n_2}^j\|_{L^2} .
\end{align*}

\subsubsection{\textbf{Estimate of} \texorpdfstring{$J_{2}$}{}{}}
\label{Estimate of J2 for 1,2 case}
This estimate is analogous to that of $I_4$ (see \eqref{Estimate of I4 for 1,1 case}) and can be easily derived from Proposition \ref{Lemma: Martini_Mullar_Weighted_Plancherel}. In particular, applying H\"older's inequality twice and Proposition \ref{Lemma: Martini_Mullar_Weighted_Plancherel} leads to
\begin{align*}
    \|J_2\|_{L^{2/3}} &\leq C \sum_{l \in \mathbb{Z}} \sum_{M_1=j+ 1 }^{\infty} 2^{jQ(1+\varepsilon)/2} \| \phi_{j,l, M_1}^{\alpha}(\mathcal{L}, T)f_{n_1}^j\|_{L^2} \| \psi_l(\mathcal{L})g_{n_2}^j\|_{L^2} \\
    &\leq C \sum_{l \in \mathbb{Z}} \sum_{M_1=j+ 1 }^{\infty} 2^{jQ(1+\varepsilon)/2} 2^{-M_1 d_2/2} \|\phi_{j,l}^{\alpha}\|_{L^{\infty}} \| f_{n_1}^j\|_{L^1} \| g_{n_2}^j\|_{L^2} .
\end{align*}
Next, employing the fact from \eqref{Convergence of sum of l using phi} and summing for $M_1\geq j+1$, we can conclude that
\begin{align}
\label{Summing of l and M1 for 1,2 case}
    \|J_2\|_{L^{2/3}} &\leq C 2^{j\varepsilon(1+d_2/2)} 2^{-j \alpha} 2^{j d(1+\varepsilon)/2} \| f_{n_1}^j\|_{L^1} \| g_{n_2}^j\|_{L^2} \\ 
    &\nonumber\leq C 2^{-j \delta} \| f_{n_1}^j\|_{L^1} \| g_{n_2}^j\|_{L^2} ,
\end{align}
where $\delta=\alpha-d/2-\varepsilon(1+d_2/2)>0$, since $\alpha>d/2$ and we can choose $\varepsilon>0$ very small.

\section{Proof of the claim (\ref{claim in the estimate of Balpha j,1}) for \texorpdfstring{$(p_1,p_2,p)=(1,\infty,1)$}{} and \texorpdfstring{$(2,2,1)$}{}}\label{section 9}
Let us first proceed with the proof of claim (\ref{claim in the estimate of Balpha j,1}) corresponding to the point $(p_1,p_2,p)=(1,\infty,1)$.
\subsection{For \texorpdfstring{$(p_1,p_2,p)=(1,\infty,1)$}{}}
In the present case, $\alpha(1, \infty)= Q/2$. From \eqref{Equation: Decomposition in terms of Fourier series} and applying H\"older's inequality yields
\begin{align}
\label{Application of Holder in 1, infinity case}
    \|\chi_{S_{n}^j} \mathcal{B}^{\alpha}_{j,1} (f_{n_1}^j,g_{n_2}^j)\|_{L^{1}} & \leq C \sum_{l \in \mathbb{Z}} \|\phi_{j,l}^{\alpha}(\mathcal{L})f_{n_1}^j\|_{L^2} \|\psi_l(\mathcal{L})g_{n_2}^j\|_{L^2} .
\end{align}
Following the earlier approaches, we split the proof into the following two cases.

\subsubsection{Case-I : \texorpdfstring{$|a'_n| > \frac{11}{5} 2^{j(1+\varepsilon)}$}{}}
\label{subsection: Case I for 1, infinity case}
First, note that as shown in subsection \ref{subsection: case I of 1,1 case}, we have $|a_n'| \sim |a_{n_1}'| \sim |a_{n_2}'|$. Therefore, using Proposition \ref{Weighted restriction estimate}, summing over $l \in \mathbb{Z}$ (see \eqref{Final estimate of case I for 1,1 case}) and applying H\"older's inequality, from \eqref{Application of Holder in 1, infinity case} we obtain
\begin{align*}
    \|\chi_{S_{n}^j} \mathcal{B}^{\alpha}_{j,1} (f_{n_1}^j,g_{n_2}^j)\|_{L^1} & \leq C \sum_{l \in \mathbb{Z}} |a'_{n_1}|^{-d_2/2} \|\phi_{j,l}^{\alpha}\|_{L^{\infty}} \|f_{n_1}^j\|_{L^1} \|\psi_l\|_{L^{\infty}} \|g_{n_2}^j\|_{L^2} \\
    &\leq C 2^{j \varepsilon} 2^{-j\alpha} |a'_{n_1}|^{-d_2/2} \|f_{n_1}^j\|_{L^1} \|g_{n_2}^j\|_{L^{\infty}} 2^{j d (1+\varepsilon)/2} |a'_{n_2}|^{d_2/2}  \\
    %&\leq C 2^{j \varepsilon} 2^{-j\alpha} \|f_{n_1}\|_{L^1} \|g_{n_2}\|_{L^{q_0}} 2^{j \varepsilon} 2^{j d/2} |a_{n_2}'|^{-d_2/q_0}  \\
    &\leq C 2^{-j \delta} \|f_{n_1}^j\|_{L^1} \|g_{n_2}^j\|_{L^{\infty}},
\end{align*}
where, since $\alpha> d/2$, we can choose $\varepsilon>0$ so small so that $\delta=\alpha-d/2-\varepsilon(1+d/2) >0$.

\subsubsection{Case-II : \texorpdfstring{$|a'_n| \leq \frac{11}{5} 2^{j(1+\varepsilon)}$}{}} Following the same reasoning as in the previous case, the estimate \eqref{Application of Holder in 1, infinity case}, Proposition \ref{Weighted restriction estimate} for $\gamma=0$, and H\"older's inequality yields
\begin{align*}
   \|\chi_{S_{n}^j} \mathcal{B}^{\alpha}_{j,1} (f_{n_1}^j,g_{n_2}^j)\|_{L^1} &\leq C \sum_{l \in \mathbb{Z}} \|\phi_{j,l}^{\alpha}\|_{L^{\infty}} \|f_{n_1}^j\|_{L^{1}} \|g_{n_2}^j\|_{L^{2}} \\
   &\leq C 2^{-j\alpha} 2^{j \varepsilon} \|f_{n_1}^j\|_{L^{2}} 2^{j (1+\varepsilon)Q/2} \|g_{n_2}\|_{L^{\infty}} \\
   &\leq C 2^{-j \delta} \|f_{n_1}\|_{L^{2}} \|g_{n_2}\|_{L^{\infty}},
\end{align*}
where $\delta=\alpha-Q/2-\varepsilon(1+Q/2)>0$, as $\alpha>Q/2$ and we can choose $\varepsilon>0$ very small.

This completes the proof the claim \ref{claim in the estimate of Balpha j,1} for $(1, \infty,1)$. Next, we move to the proof of the claim for the point $(2,2,1)$.

\subsection{For \texorpdfstring{$(p_1,p_2,p)=(2,2,1)$}{}}
\label{subsection: Proof at 2,2,1 case}
Here, $\alpha(2,2)=0$. Following the same steps as  in the previous estimate together with H\"older's inequality immediately leads to
\begin{align*}
   \|\chi_{S_{n}^j} \mathcal{B}^{\alpha}_{j,1} (f_{n_1}^j,g_{n_2}^j)\|_{L^1} &\leq C \sum_{l \in \mathbb{Z}} \| \phi_{j,l}^{\alpha}(\mathcal{L})f_{n_1}^j\|_{L^2} \| \psi_l(\mathcal{L})g_{n_2}^j\|_{L^2} \\
   &\leq C \sum_{l \in \mathbb{Z}}  \|\phi_{j,l}^{\alpha}\|_{L^{\infty}} \|f_{n_1}^j\|_{L^{2}} \|g_{n_2}^j\|_{L^{2}} \leq C 2^{- j \delta} \|f_{n_1}^j\|_{L^{2}} \|g_{n_2}^j\|_{L^{2}} ,
   %&\leq C 2^{j \varepsilon} 2^{-j \alpha} \|f_{n_1}\|_{L^{2}} \|g_{n_2}\|_{L^{2}} \\
   %&\leq C 2^{- j \delta} \|f_{n_1}^j\|_{L^{2}} \|g_{n_2}^j\|_{L^{2}} ,
\end{align*}
where $\delta=\alpha-\varepsilon>0$ for $\varepsilon>0$ very small, since $\alpha>0$.

\section{Proof of Theorem \ref{Theorem: Bilinear Bochner-Riesz with Fourier transform away from origin}}
\label{Section: Proof of theorem for restricted f and g}
This section is devoted to the proof of Theorem \ref{Theorem: Bilinear Bochner-Riesz with Fourier transform away from origin}. The argument closely follows that of Theorem \ref{Bilinear Bochner-Riesz main theorem} (see Section \ref{Section: proof of banach case}) In fact, the first three steps remain identical, and only the final step requires a different approach. As in that case, the key estimate is $\mathcal{B}^{\alpha}_{j,1}$.

Assume that $\supp \mathcal{F}_2 f(x', \cdot) \subseteq \{\lambda_1 : |\lambda_1| \geq \kappa_1 \}$ and $\supp \mathcal{F}_2 g(x', \cdot) \subseteq \{\lambda_2 : |\lambda_2| \geq \kappa_2 \}$ for some $\kappa_1, \kappa_2>0$. We define two smooth functions $\Omega_1$ and $\Omega_2$ on $\mathbb{R}$ such that, $\Omega_1(\tau_1) = 0$, if $\tau_1 \in (-\kappa_1/2, \kappa_1/2)$, and $\Omega_1 \equiv 1$ on the support of $\mathcal{F}_2 f(x', \cdot)$, while $\Omega_2(\tau_2) = 0$, if $\tau_2 \in (-\kappa_2/2, \kappa_2/2)$ and $\Omega_2 \equiv 1$ on the support of $\mathcal{F}_2 g(x', \cdot)$. Consequently, from \eqref{Bilinear Bochner-Riesz operator after dyadic decomposition} and the support conditions, we can express $\mathcal{B}^{\alpha}_{j,1}$ as follows
\begin{align*}
    \mathcal{B}^{\alpha}_{j,1}(f,g)(x) &= \frac{1}{(2\pi)^{2 d_2}} \int_{\mathbb{R}^{d_2}} \int_{\mathbb{R}^{d_2}} e^{i(\lambda_1 + \lambda_2)\cdot x''} \sum_{k_1, k_2 = 0}^{\infty} \varphi^{\alpha}_j ([k_1]|\lambda_1|, [k_2]|\lambda_2|) \,\Omega_1(|\lambda_1|) \Omega_2(|\lambda_2|) \\
    & \hspace{7cm} P^{\lambda_1} _{k_1} f^{\lambda_1}(x') \,P^{\lambda_2} _{k_2} g^{\lambda_2}(x'')\, d\lambda_1 d\lambda_2 \\
    &=: \mathcal{B}^{\alpha, \kappa_1, \kappa_2}_{j,1}(f,g)(x) .
\end{align*}

Therefore, as in subsection \ref{subsection: Estimate of Bj1 case}, we are once again reduced to proving the following: for $ \varrho(a_n , a_{n_1})\leq 2 \cdot 2^{j(1+\varepsilon)}$ and $\varrho(a_n , a_{n_2})\leq 2 \cdot 2^{j(1+\varepsilon)}$, whenever $\alpha>\alpha(p_1, p_2)$, there exists a $\delta>0$ such that
\begin{align}
\label{claim in the estimate of Balpha j,1 for restricted case}
    \| \chi_{S_{n}^j} \mathcal{B}^{\alpha, \kappa_1, \kappa_2}_{j,1} (f_{n_1}^j,g_{n_2}^j) \|_{L^{p}} \leq C 2^{-j \delta} \|f_{n_1}^j\|_{L^{p_1}} \|g_{n_2}^j\|_{L^{p_2}} ,
\end{align}
where $(p_1,p_2,p)=$ $(1,1,1/2)$, $(1,2,2/3)$, $(2,2,1)$ and $(1, \infty,1)$.

When we need only one function to be supported away from origin, we introduce $\Omega_1$ or $\Omega_2$ in the above expression of $\mathcal{B}^{\alpha}_{j,1}$, depending on which functions $\mathcal{F}_2 f(x', \cdot)$ and $\mathcal{F}_2 g(x', \cdot)$ has support away from the origin. Accordingly, in such cases, we also denote the corresponding operators by either $\mathcal{B}^{\alpha, \kappa_1}_{j,1}$ or $\mathcal{B}^{\alpha, \kappa_2}_{j,1}$. In this section we only show how the prove \ref{claim in the estimate of Balpha j,1 for restricted case} for the points $(1, \infty,1)$. Since the proof for the points $(1,1,1/2)$ and $(1,2,2/3)$ are similar to that of $(1, \infty, 1)$, we omit it. And proof at $(2,2,1)$ follows from subsection \ref{subsection: Proof at 2,2,1 case}.

\subsection{Proof of (\ref{claim in the estimate of Balpha j,1 for restricted case}) at \texorpdfstring{$(p_1, p_2,p)=(1, \infty,1)$}{}}
In this case, we only need the support of $\mathcal{F}_2 g(x', \cdot)$ to lie outside the origin. Also, recall $\alpha(1, \infty)=d/2$ in the present case. First, observe that we can assume $|a'_n| \leq \frac{11}{5} 2^{j(1+\varepsilon)}$, since the other case has  already been addressed in subsection \eqref{subsection: Case I for 1, infinity case} for $\alpha>d/2$.

Similar to (\ref{Equation: Decomposition in terms of Fourier series}), we decompose the multiplier into Fourier series, and consequently write
\begin{align*}
    \mathcal{B}_{j,1}^{\alpha, \kappa_2}(f_{n_1}^j,g_{n_2}^j)(x) = C \sum_{l \in \mathbb{Z}} \left\{ \phi_{j,l}^{\alpha}(\mathcal{L})f_{n_1}^j(x) \right\} \left\{ \psi_l^{\kappa_2}(\mathcal{L}, T)g_{n_2}^j(x) \right\} ,
\end{align*}
where $\psi_l^{\kappa_2} : \mathbb{R} \times \mathbb{R} \to \mathbb{C}$ is defined by $\psi_l^{\kappa_2}(\eta_2, \tau_2)= \psi_l(\eta_2) \Omega_2(\tau_2)$.

Analogous to \eqref{Cutoff of M1 for the first linear multiplier}, we introduce cut-off in $|\lambda_1|$-variable associated to the function $\phi_{j,l}^{\alpha}$ and similar to \eqref{Decomposition of phi into M1 cutoff}, we decompose as follows
\begin{align*}
    \chi_{S_{n}^j}(x) \mathcal{B}_{j,1}^{\alpha, \kappa_2}(f_{n_1}^j,g_{n_2}^j)(x) &= C \chi_{S_{n}^j}(x) \sum_{l \in \mathbb{Z}} \left(\sum_{M_1=0 }^{j} + \sum_{M_1=j+1 }^{\infty} \right) \phi_{j,l,M_1}^{\alpha}(\mathcal{L}, T) f_{n_1}^j(x) \psi_{l}^{\kappa_2}(\mathcal{L}, T) g_{n_2}^j(x) \\
    &:= S_1 + S_2 .
\end{align*}

\subsubsection{\textbf{Estimate of} \texorpdfstring{$S_{2}$}{}{}}
This estimate can be obtained in a similar fashion to that of $J_2$ (see subsection \ref{Estimate of J2 for 1,2 case}). Applying H\"older's inequality and Proposition \ref{Lemma: Martini_Mullar_Weighted_Plancherel} yields
\begin{align*}
    \|S_2\|_{L^{1}} &\leq C \sum_{l \in \mathbb{Z}} \sum_{M_1=j+ 1 }^{\infty} \| \phi_{j,l, M_1}^{\alpha}(\mathcal{L}, T)f_{n_1}^j\|_{L^2} \| \psi_l^{\kappa_2}(\mathcal{L}, T)g_{n_2}^j\|_{L^2} \\
    &\leq C \sum_{l \in \mathbb{Z}} \sum_{M_1=j+ 1 }^{\infty} 2^{-M_1 d_2/2} \|\phi_{j,l}^{\alpha}\|_{L^{\infty}} \| f_{n_1}^j\|_{L^1} \| g_{n_2}^j\|_{L^2} .
\end{align*}
Now, following the same argument as in \eqref{Summing of l and M1 for 1,2 case} and again using H\"older's inequality, we get
\begin{align*}
    \|S_2\|_{L^{1}} &\leq C 2^{-j \alpha} 2^{j \varepsilon} 2^{-j d_2/2} \| f_{n_1}^j\|_{L^1} \|g_{n_2}^j\|_{L^{\infty}} 2^{j Q(1+\varepsilon)/2} \\ 
    &\leq C 2^{-j \delta} \| f_{n_1}^j \|_{L^1} \|g_{n_2}^j\|_{L^{\infty}} ,
\end{align*}
where $\delta=\alpha-d/2-\varepsilon(1+Q/2)>0$, by choosing $\varepsilon>0$ very small, since $\alpha>d/2$.

\subsubsection{\textbf{Estimate of} \texorpdfstring{$S_{1}$}{}{}}
To proceed further, we again introduce a cut-off in $|\lambda_2|$-variable associated to the function $\psi_{l}^{\kappa_2}$, in a manner similar to the decomposition carried out in  \eqref{Cutoff of M1 for the first linear multiplier}. Specifically, we decompose it as follows
\begin{align*}
    S_1 &= C \chi_{S_{n}^j}(x) \sum_{l \in \mathbb{Z}} \sum_{M_1=0 }^{j} \sum_{M_2=0 }^{\infty} \phi_{j,l,M_1}^{\alpha}(\mathcal{L}, T) f_{n_1}^j(x) \psi_{l, M_2}^{\kappa_2}(\mathcal{L}, T) g_{n_2}^j(x) ,
    % &= \left( \sum_{M_1=-\ell_0 }^{j} \sum_{M_2=-\ell_0 }^{3j} + \sum_{M_1=-\ell_0 }^{j} \sum_{M_2=3j+1 }^{\infty} \right) \chi_{B_{0}}(x,u) \mathcal{B}^{\alpha}_{j,1, M_1, M_2}(f_{n_1}, g_{n_2})(x) \\
    % &= S_{11} + S_{12} .
\end{align*}
where $\psi_{l, M_2}^{\kappa_2}(\eta_2, \tau_2)= \psi_{l}(\eta_2) \Omega_2(\tau_2) \Theta(2^{M_2}\tau_2)$.

First, note that due to the support considerations, $\Omega_2(\tau_2) \Theta(2^{M_2}\tau_2)$ is non vanishing only if there exists some $L_0>0$, depending on $\kappa_2$ such that $M_2 \leq L_0$. Now, similar to the estimate of $I_1$ (see subsection \ref{Estimate of I1 for 1,1 case}), for $0\leq M_1 \leq j$, we decompose both the support of $f_{n_1}^j$ and $g_{n_2}^j$ as well as the functions themselves, with respect to the second layer, see \eqref{Expression: Decomposition of ball into smaller balls}-\eqref{Decomposition of f and g}. With this decomposition in hand, we can split $S_1$ as in \eqref{Decomposition of S1 into S11 and S12 in (1,1,1/2)} as follows.
\begin{align}
\label{Decomposition of S1 in 1, infinity case}
    S_{1} &= \sum_{M_1=0 }^{j} \sum_{M_2=0 }^{L_0} \sum_{m_1=1}^{N_{M_1}} \sum_{m_2=1}^{N_{M_1}} \chi_{S_{n}^j}(x) \chi_{\widetilde{S}_{n,n_1,m_1}^{M_1,j}}(x) \chi_{\widetilde{S}_{n_2,m_2}^{M_1,j}}(x) \mathcal{B}^{\alpha, \kappa_2}_{j,1,M_1, M_2}(f_{n_1,m_1}^{M_1,j}, g_{n_2,m_2}^{M_1,j})(x) \\
    &\nonumber + \sum_{M_1=0 }^{j} \sum_{M_2=0 }^{L_0} \sum_{m_1=1}^{N_{M_1}} \chi_{S_{n}^j}(x) (1-\chi_{\widetilde{S}_{n,n_1,m_1}^{M_1,j}})(x) \mathcal{B}^{\alpha, \kappa_2}_{j,1,M_1, M_2}(f_{n_1,m_1}^{M_1,j}, g_{n_2}^{j})(x) \\
    &\nonumber + \sum_{M_1=0 }^{j} \sum_{M_2=0 }^{L_0} \sum_{m_1=1}^{N_{M_1}} \sum_{m_2=1}^{N_{M_1}} \chi_{S_{n}^j}(x) \chi_{\widetilde{S}_{n_1,m_1}^{M_1,j}}(x) (1-\chi_{\widetilde{S}_{n,n_2,m_2}^{M_1,j}})(x) \mathcal{B}^{\alpha, \kappa_2}_{j,1,M_1, M_2}(f_{n_1,m_1}^{M_1,j}, g_{n_2,m_2}^{M_1,j})(x) \\
    &\nonumber =: S_{11} + S_{12} + S_{13} ,
\end{align}
where
\begin{align*}
    \mathcal{B}_{j,1, M_1, M_2}^{\alpha, \kappa_2}(f,g)(x) = C \sum_{l \in \mathbb{Z}} \left\{ \phi_{j,l, M_1}^{\alpha}(\mathcal{L}, T)f(x) \right\} \left\{ \psi_{l,M_2}^{\kappa_2}(\mathcal{L}, T)g(x) \right\} .
\end{align*}

\subsubsection{\textbf{Estimate of} \texorpdfstring{$S_{11}$}{}{}}
\label{Estimate of S11 for 1, infinity restrictive case}
Recall that $S_n^j \subseteq B(a_n, 2^{j(1+\varepsilon)}/5)$ and $|a_n'| \leq 2^{j(1+\varepsilon)} 11/5$. Therefore, we proceed as in (\ref{Expression: Decomposition of ball into smaller balls}) by decomposing $S_n^j$ into disjoint sets, and accordingly, we write
\begin{align}
\label{Use of bounded overlapping in estimate of S11 case}
    S_{11} &= \sum_{M_1=0 }^{j} \sum_{M_2=0 }^{L_0} \sum_{m=1}^{N_{M_1}} \sum_{m_1=1}^{N_{M_1}} \sum_{m_2=1}^{N_{M_1}} \chi_{S_{n,m}^{M_1,j}}(x) \chi_{\widetilde{S}_{n,n_1,m_1}^{M_1,j}}(x) \chi_{\widetilde{S}_{n,n_2,m_2}^{M_1,j}}(x) \mathcal{B}^{\alpha, \kappa_2}_{j,1,M_1, M_2}(f_{n_1,m_1}^{M_1,j}, g_{n_2,m_2}^{M_1,j})(x) \\
    &\nonumber = \sum_{M_1=0 }^{j} \sum_{M_2=0 }^{L_0} \sum_{m=1}^{N_{M_1}} \sum_{m_1: S_{n,m}^{M_1,j} \cap \widetilde{S}_{n,n_1,m_1}^{M_1,j} \neq \emptyset} \,\,\sum_{m_2: S_{n,m}^{M_1,j} \cap \widetilde{S}_{n,n_2,m_2}^{M_1,j} \neq \emptyset} \\
    &\nonumber \hspace{4cm} \chi_{S_{n,m}^{M_1,j}}(x) \chi_{\widetilde{S}_{n_1,m_1}^{M_1,j}}(x) \chi_{\widetilde{S}_{n_2,m_2}^{M_1,j}}(x) \mathcal{B}^{\alpha, \kappa_2}_{j,1,M_1, M_2}(f_{n_1,m_1}^{M_1,j}, g_{n_2,m_2}^{M_1,j})(x) .
\end{align}
Now in order to estimate the $L^1$-norm of $S_{11}$, we need to first evaluate the following quantity. Using H\"older's inequality and Proposition \ref{Lemma: Martini_Mullar_Weighted_Plancherel}, we obtain
\begin{align*}
    \|\mathcal{B}^{\alpha, \kappa_2}_{j,1,M_1, M_2}(f_{n_1,m_1}^{M_1,j}, g_{n_2,m_2}^{M_1,j})\|_{L^1} &\leq C \sum_{l \in \mathbb{Z}} \|\phi_{j,l,M_1}^{\alpha}(\mathcal{L}, T) f_{n_1,m_1}^{M_1,j} \|_{L^2} \|\psi_{l,M_2}^{\kappa_2}(\mathcal{L}, T) g_{n_2,m_2}^{M_1,j} \|_{L^2} \\
    &\leq C \sum_{l \in \mathbb{Z}} 2^{-M_1 d_2/2} \|\phi_{j,l}^{\alpha}\|_{L^{\infty}} \|f_{n_1,m_1}^{M_1,j} \|_{L^1} \|g_{n_2,m_2}^{M_1,j}\|_{L^2} .
\end{align*}
Summing over $l\in \mathbb{Z}$ using \eqref{Convergence of sum of l using phi}, and once again applying H\"older's inequality to the previous estimate, we derive
\begin{align*}
    \|\mathcal{B}^{\alpha, \kappa_2}_{j,1,M_1, M_2}(f_{n_1,m_1}^{M_1,j}, g_{n_2,m_2}^{M_1,j})\|_{L^1} & \leq C 2^{j \varepsilon} 2^{-j \alpha} 2^{- M_1 \frac{d_2}{2}}  2^{j(1+\epsilon) \frac{d_1}{2}} 2^{(j+M_1)(1+\varepsilon) \frac{d_2}{2}} \|f_{n_1,m_1}^{M_1,j} \|_{L^1} \|g_{n_2,m_2}^{M_1,j}\|_{L^{\infty}} .
    % &\leq C 2^{j \varepsilon} 2^{-j \alpha} 2^{-M_1 d_2/q_0} 2^{j d (1/2-1/q_0)} 2^{j (M_1-j)(d_1-d_2) (1/2-1/q_0)} \|f_{n_1,m_1}^{M_1} \|_{L^1} \|g_{n_2,m_2}^{M_1}\|_{L^{q_0}} .
\end{align*}
With the above estimate in hand, we arrive at
\begin{align}
\label{Putting the l1 estimate in S11}
    & \|S_{11}\|_{L^1} \leq C 2^{j \varepsilon(1+d_2/2)} 2^{-j \alpha} 2^{j(1+\varepsilon)d/2} \sum_{M_1=0 }^{j} \sum_{M_2=0 }^{L_0} \\
    &\nonumber \hspace{3cm} \sum_{m=1}^{N_{M_1}} \Big\{\sum_{m_1: S_{n,m}^{M_1,j} \cap \widetilde{S}_{n_1,m_1}^{M_1,j} \neq \emptyset} \|f_{n_1,m_1}^{M_1,j} \|_{L^1} \Big\} \Big\{ \sum_{m_2: S_{n,m}^{M_1,j} \cap \widetilde{S}_{n_2,m_2}^{M_1,j} \neq \emptyset} \|g_{n_2,m_2}^{M_1,j}\|_{L^{\infty}} \Big\} .
\end{align}

Before proceed further, recall that for each $i=1,2$, the sets $S_{n_i,m_i}^{M_1,j} \subseteq B^{|\cdot|}\left(a_{n_i}', 2^{j(1+\varepsilon)}/5\right) \times B^{|\cdot|}\left(b_{n_i,m_i}^{M_i}, C 2^{(j+M_1)(1+\varepsilon)}/25\right)$ were constructed to be disjoint, with the $|b_{n_i,m_i}^{M_i}-b_{n_i,m_i'}^{M_i}|> C 2^{(j+M_1)(1+\varepsilon)}/50$ for $m_i \neq m_i'$. Therefore the sets $\widetilde{S}_{n,n_i,m_i}^{M_1,j}$ (see \eqref{Defination of dilated balls}) satisfy the following bounded overlapping property
\begin{align}
\label{Bounded overlapping property for M1 balls}
    \sup_{m} \# \left\{ m_i : S_{n,m}^{M_1,j} \cap \widetilde{S}_{n,n_i,m_i}^{M_1,j} \neq \emptyset \right\} &\leq \sup_{m} \# \left\{ m_i : |b_{n,m}^{M_1} - b_{n_i,m_i}^{M_1}| \leq C 2^{(j+M_1)(1+\varepsilon)} 2^{\gamma j+1} \right\} \\
    &\nonumber \leq C 2^{C \gamma j} ,
\end{align}
and similarly,
\begin{align}\label{Bounded overlapping property for M1 balls 2}
    \sup_{m_i} \# \left\{ m : S_{n,m}^{M_1,j} \cap \widetilde{S}_{n,n_i,m_i}^{M_1,j} \neq \emptyset \right\} \leq C 2^{C \gamma j}.
\end{align}

Finally, using the above bounded overlap estimates \eqref{Bounded overlapping property for M1 balls} and \ref{Bounded overlapping property for M1 balls 2}, from the estimate obtained in \eqref{Putting the l1 estimate in S11}, we can conclude that
\begin{align}
\label{Final estimate of S11 in 1,infinity restrictive case}
    \|S_{11}\|_{L^1} & \leq C 2^{j \varepsilon(1+d_2/2)} 2^{-j \alpha} 2^{j(1+\varepsilon)d/2} \sum_{M_1=0 }^{j} \sum_{M_2=0 }^{L_0} \Big\{ \sum_{m=1}^{N_{M_1}} \sum_{m_1: S_{n,m}^{M_1,j} \cap \widetilde{S}_{n,n_1,m_1}^{M_1,j} \neq \emptyset} \|f_{n_1,m_1}^{M_1,j} \|_{L^1} \Big\} \\
    &\nonumber \hspace{5cm} \times \Big\{ \sup_m \sum_{m_2: S_{n,m}^{M_1,j} \cap \widetilde{S}_{n,n_2,m_2}^{M_1,j} \neq \emptyset} \|g_{n_2,m_2}^{M_1,j}\|_{L^{\infty}} \Big\} \\
    &\nonumber \leq C 2^{C \epsilon_1} 2^{C \gamma j} 2^{-j \alpha} 2^{j d/2} \|f_{n_1}^j\|_{L^1} \|g_{n_2}^j\|_{L^{\infty}} \\
    &\nonumber \leq C 2^{-j \delta} \|f_{n_1}^j\|_{L^1} \|g_{n_2}^j\|_{L^{\infty}} ,
\end{align}
where $\epsilon_1>0$ which depends on $\varepsilon>0$ and $\delta=\alpha-d/2-C \epsilon_1-C \gamma>0$, since $\alpha>d/2$, we can choose $\gamma, \epsilon_1>0$ very small such that $\alpha-d/2-C \epsilon_1-C \gamma>0$.

\subsubsection{\textbf{Estimate of} \texorpdfstring{$S_{12}$}{}{}}
The process of estimating  $S_{12}$ is analogous to that of $J_{12}$ (see subsection \ref{Estimate of J12 for 1,2 case}). Applying H\"older's inequality and utilizing the 
bound from \eqref{L1 to L2 norm for phijl on outside} yields
\begin{align*}
    \|S_{12}\|_{L^{1}} & \leq C \sum_{M_1=0 }^{j} \sum_{M_2=0 }^{L_0} \sum_{m_1=1}^{N_{M_1}} \sum_{l \in \mathbb{Z}} \|\chi_{S_{n}^j} (1-\chi_{\widetilde{S}_{n,n_1,m_1}^{M_1,j}}) \phi_{j,l,M_1}^{\alpha}(\mathcal{L}, T) f_{n_1,m_1}^{M_1,j}\|_{L^2} \|\psi_{l,M_2}^{\kappa_2}(\mathcal{L}, T) g_{n_2}^{j}\|_{L^2} \\
    &\leq C 2^{-\gamma j N} \sum_{M_1=0 }^{j} \sum_{M_2=0 }^{L_0} \Big\{\sum_{l \in \mathbb{Z}} \|\phi_{j,l}^{\alpha}\|_{L^{\infty}} \Big\} \Big\{\sum_{m_1=1}^{N_{M_1}} \|f_{n_1,m_1}^{M_1,j}\|_{L^1} \Big\} \|g_{n_2}^{j}\|_{L^2} .
\end{align*}
Now, again applying H\"older's inequality, summing over $l\in \mathbb{Z}$ using \eqref{Convergence of sum of l using phi}, and choosing $N>0$ large enough, we get the desire bound, that is
\begin{align*}
    \|S_{12}\|_{L^{1}} & \leq C 2^{-\gamma j N} 2^{2 j \varepsilon} 2^{-j \alpha} \|f_{n_1}^j\|_{L^1} 2^{j(1+\varepsilon)Q/2} \|g_{n_2}^j\|_{L^{\infty}} \\
    &\leq C 2^{-j \delta} \|f_{n_1}^j\|_{L^1} \|g_{n_2}^j\|_{L^{\infty}} ,
\end{align*}
for some $\delta>0$.

\subsubsection{\textbf{Estimate of} \texorpdfstring{$S_{13}$}{}{}}
This estimate is nearly similar to the estimate for $I_{12}$ (see \eqref{subsection: estimate of I12 for 1,1 case}), but there is little difference. Instead of the cut-off $(1-\chi_{\widetilde{S}_{n, n_2,m_2}^{M_2,j}})$, here we have $(1-\chi_{\widetilde{S}_{n, n_2,m_2}^{M_1,j}})$. This is the reason why we need to assume the support of the Fourier transform of $g$ with respect to the second variable, to lie outside the origin.

Applying H\"older's inequality  shown in \eqref{Estimate: Application of Holder inequality for error part in 1,1}, we get
\begin{align}
\label{Estimate of L1 norm for 1, infinity case restricted}
    \|S_{13}\|_{L^1} &\leq C \sum_{M_1=0 }^{j} \sum_{M_2=0 }^{L_0} \sum_{m_1=1}^{N_{M_1}} \sum_{m_2=1}^{N_{M_1}} \sum_{l \in \mathbb{Z}} \|\phi_{j,l,M_1}^{\alpha}(\mathcal{L}, T) f_{n_1,m_1}^{M_1,j} \|_{L^2} \\
    &\nonumber \hspace{5cm} \|\chi_{S_{n}^j} (1-\chi_{\widetilde{S}_{n,n_2,m_2}^{M_1,j}}) \psi_{l,M_2}^{\kappa_2}(\mathcal{L}, T) g_{n_2,m_2}^{M_1,j} \|_{L^2} .
\end{align}

First note that, from \eqref{Defination of dilated balls}, if $x \in \supp{\chi_{S_{n}^j} (1-\chi_{\widetilde{S}_{n, n_2,m_2}^{M_1,j}})}$ and $z \in \supp{g_{n_2,m_2}^{M_1,j}}$, then $|x''-b_{n_2,m_2}^{M_1}| \geq C 2^{\gamma j+1} 2^{(j+M_1)(1+\varepsilon)}$ and $|z''-b_{n_2,m_2}^{M_1}| \leq C 2^{(j+M_1)(1+\varepsilon)}$. So that these implies $|x''-z''| \geq C 2^{\gamma j} 2^{(j+M_1)(1+\varepsilon)}$.

Similar to the estimate \eqref{Estimate: Application of Minkowski inequality for error part in 1,1} and (\ref{Estimate: Weighted plancherel on outside ball in (1,1,1/2)}), the estimation of the $L^2$-norm for $\chi_{S_{n}^j} (1-\chi_{\widetilde{S}_{n,n_2,m_2}^{M_1,j}}) \psi_{l,M_2}^{\kappa_2}(\mathcal{L}, T) g_{n_2,m_2}^{M_1,j}$ leads to the following calculation 
\begin{align*}
    & \left( \int_{\mathbb{R}^d} |\chi_{S_{n}^j}(x) (1-\chi_{\widetilde{S}_{n, n_2,m_2}^{M_1,j}})(x) \mathcal{K}_{\psi_{l,M_2}^{\kappa_2}(\mathcal{L}, T)}(x,z)|^2 \ dx \right)^{1/2} \\
    % &= \left( \int_{B^{y''}} | K_{\phi_{j,l,M_1} ^{\alpha}(\mathcal{L}, T)}(x,y)|^2 \ dx \right)^{1/2} \\
    &\leq C (2^{\gamma j} 2^{(j+M_1)(1+\varepsilon)} )^{-N} \left( \int_{\mathbb{R}^d} | |x''-z''|^N \mathcal{K}_{\psi_{l,M_2}(\mathcal{L}, T)}(x,z)|^2 \ dx \right)^{1/2} \\
    &\leq C (2^{\gamma j} 2^{(j+M_1)(1+\varepsilon)} )^{-N} 2^{M_2(N-d_2/2)} \|\psi_{l}^{\kappa_2}\|_{L^2_N} \\
    &\leq C 2^{-\gamma j N} 2^{-j N} |l|^N ,
\end{align*}
where we have used the fact $M_2 \leq L_0$.

Thus, as estimated in \eqref{Estimate: Outside of M1 ball for f in (1,1,1/2)} and \eqref{Estimate: Estimate of error part wrt f for (1,1,1/2)}, we deduce from \eqref{Estimate of L1 norm for 1, infinity case restricted} that
\begin{align*}
    &\|S_{13}\|_{L^1} \\
    &\leq C 2^{-\gamma j N} \sum_{M_1=0 }^{j} \sum_{M_2=0 }^{L_0} 2^{-M_1 d_2/2} \Big\{\sum_{l \in \mathbb{Z}} 2^{-j N} |l|^N \|\phi_{j,l} ^{\alpha}\|_{L^{\infty}} \Big\} \Big\{ \sum_{m_1=1}^{N_{M_1}}\|f_{n_1,m_1}^{M_1,j}\|_{L^1} \Big\} \Big\{ \sum_{m_2=1}^{N_{M_1}} \|g_{n_2,m_2}^{M_1,j}\|_{L^1} \Big\} \\
    &\leq C 2^{-\gamma j N} 2^{j \varepsilon} 2^{-j \alpha} \|f_{n_1}^j\|_{L^1} 2^{j(1+\varepsilon)Q} \|g_{n_2}^j\|_{L^{\infty}} \\
    &\leq C 2^{-j \delta} \|f_{n_1}^j\|_{L^1} \|g_{n_2}^j\|_{L^{\infty}} ,
\end{align*}
provided we choose $N>0$ sufficiently large and $\varepsilon>0$ very small.

This completes the proof of \eqref{claim in the estimate of Balpha j,1 for restricted case} for the point $(p_1, p_2,p)=(1, \infty,1)$.

\section{Mixed norm estimates}\label{section 11}
The proof of Theorem \ref{Theorem: Mixed norm estimate for first layer} is presented here. Just like in subsection \ref{subsection: step 1 decomposition of Bj1 case}, our goal is to prove, there exists a $\delta>0$ such that
\begin{align*}
    \|\mathcal{B}^{\alpha}_{j}(f,g)\|_{L_{x''}^{2/3} L_{x'}^{1}} &\leq C 2^{-j \delta} \|f\|_{L_{x''}^1 L_{x'}^1} \|g\|_{L_{x''}^2 L_{x'}^{\infty}} .
\end{align*}
We decompose the kernel $\mathcal{K}_j^{\alpha}$, which corresponds to the operator $\mathcal{B}^{\alpha}_{j}$ as
\begin{align*}
     \mathcal{K}_j^{\alpha} &= \sum_{\theta_1, \theta_2 \in \{1,2,3,4\}} \mathcal{K}_{j,A_{\theta_1}, A_{\theta_2}}^{\alpha} , 
\end{align*}
where $\mathcal{K}_{j,A_{\theta_1}(x), A_{\theta_2}(x)}^{\alpha}$ is given by
\begin{align}
\label{decomposition of kernel in mixed case}
      \mathcal{K}_{j,A_{\theta_1}, A_{\theta_2}}^{\alpha}(x,y,z) &= \mathcal{K}_{j}^{\alpha}(x,y,z) \,  \chi_{A_{\theta_1}} (y) \,  \chi_{A_{\theta_2}} (z) ,
\end{align}
and for $\varepsilon>0$, and $\theta_1, \theta_2 \in \{1,2,3,4\}$, we define
\begin{align*}
    A_{1} &:= B^{|\cdot|} (x', 2^{j(1+\varepsilon)}) \times B^{|\cdot|} (x'', 2^{j(1+\varepsilon)} \max\{4 \cdot 2^{j(1+\varepsilon)}, |x'|\}) ;\\
    A_{2} &:= B^{|\cdot|} (x', 2^{j(1+\varepsilon)}) \times B^{|\cdot|} (x'', 2^{j(1+\varepsilon)} \max\{4 \cdot 2^{j(1+\varepsilon)}, |x'|\})^c ; \\
    A_{3} &:= B^{|\cdot|} (x', 2^{j(1+\varepsilon)})^{c} \times B^{|\cdot|} (x'', 2^{j(1+\varepsilon)} \max\{4 \cdot 2^{j(1+\varepsilon)}, |x'|\}) ; \\
    A_{4} &:= B^{|\cdot|} (x', 2^{j(1+\varepsilon)})^{c} \times B^{|\cdot|} (x'', 2^{j(1+\varepsilon)} \max\{4 \cdot 2^{j(1+\varepsilon)}, |x'|\})^{c} .
\end{align*} 
Associated to the kernel $\mathcal{K}_{j,A_{\theta_1}, A_{\theta_2}}^{\alpha}$, denote the corresponding operator by $\mathcal{B}^{\alpha}_{j,A_{\theta_1}, A_{\theta_2}}$.

Estimate of $\mathcal{K}_{j,A_{\theta_1}, A_{\theta_2}}^{\alpha}$, except for $\mathcal{K}_{j,A_{1}, A_{1}}^{\alpha}$, can be treated similarly as in subsection \ref{Step 3, estimate of Bj234 case}, with the help of Lemma \ref{Lemma: Pointwise kernel estimate for Bj}. We illustrate how to estimate $\mathcal{K}_{j,A_{2}, A_{4}}^{\alpha}$, as the remaining cases can be handled analogously, with the exception of $\mathcal{K}_{j,A_{1}, A_{1}}^{\alpha}$.

Let us adopt the following short hand notation $[2^{j(1+\varepsilon)},|x'|]:= 2^{j(1+\varepsilon)} \max\{4 \cdot 2^{j(1+\varepsilon)}, |x'|\}$. Proceeding as in subsection \ref{subsection: case II P1 1 and P2 bigger than 1 case} and an application of Lemma \ref{Lemma: Pointwise kernel estimate for Bj} yields
\begin{align}
\label{Application of pointwise kernel lemma in mixed norm}
    & |\mathcal{B}^{\alpha}_{j,A_2,A_4}(f,g)(x)| \\
    &\nonumber \leq C 2^{j(2N+1/2+\epsilon_1)} \Big\{ \int_{\mathbb{R}^{d}}  \frac{|f(y)| \chi_{A_2}(y)\, dy} {| B ( y, 1)| \big(1 +  \varrho(x, y) \big)^N } \Big\} \Big\{ \frac{1}{|B(x,1)|}\int_{\mathbb{R}^d} \frac{|g(z)| \chi_{A_4}(z)\, dz}{(1+\varrho(x,z))^N} \Big\} ,
\end{align}
for any $\epsilon_1>0$.

Thus, using Lemma \ref{Lemma: Integral of distance with second Hardy-Littlewood} and applying the fact in \eqref{Estimate : Ball volume}, from \eqref{Application of pointwise kernel lemma in mixed norm} we obtain
\begin{align*}
    & |\mathcal{B}^{\alpha}_{j,A_2,A_4}(f,g)(x)| \\
    & \leq C 2^{j(2N+1/2+\epsilon_1)} \Big\{ \int_{\mathbb{R}^{d}}  \frac{|f(y)| \chi_{A_2}(y) \, dy} {| B ( y, 1)| \big(1 + \varrho(x, y) \big)^N } \Big\} \Big\{ \int_{|x'-z'|>2^{j(1+\varepsilon)}} \frac{\mathcal{M}^{|\cdot|_2}(g(z', \cdot))(x'') \, dz'}{(1 +  |x'-z'| )^{N-2d_2-\epsilon_1}} \Big\} .
\end{align*}
Consequently, applying H\"older's inequality, we arrive at
\begin{align}
\label{Application of Holder in mixed norm for outside}
    & \|\mathcal{B}^{\alpha}_{j,A_2,A_4}(f,g)\|_{L_{x''}^{2/3} L_{x'}^{1}} \leq C 2^{j(2N+1/2+\epsilon_1)} \\
    &\nonumber \hspace{0.5cm} \times \Big\{ \int_{\mathbb{R}^{d} } \frac{|f(y)|}{| B (y, 1)|} \Big[\int_{|x'-y'|\leq 2^{j(1+\varepsilon)}} \Big( \int_{|x''-y''|>[2^{j(1+\varepsilon)},|x'|]} \frac{ dx''}{(1+\varrho(x,y))^N} \Big)\, dx' \Big] \, dy \Big\} \\
    &\nonumber \hspace{3cm} \times \Big\{ \sup_{x'} \Big(\int_{\mathbb{R}^{d_2}} \Big|\int_{|x'-z'|>2^{j(1+\varepsilon)}} \frac{\mathcal{M}^{|\cdot|_2}(g(z', \cdot))(x'')}{(1 +  |x'-z'| )^{N-2d_2-\epsilon_1}} \, dz' \Big|^2 dx'' \Big)^{1/2} \Big\} .
\end{align}

Let us calculate the second factor in the right hand side of the above expression. Recall that $\mathcal{M}^{|\cdot|_1}$ to be the Hardy-Littlewood maximal function on $\mathbb{R}^{d_1}$. Application of Minkowski's integral inequality and boundedness of maximal function $\mathcal{M}^{|\cdot|_1}$ yields
\begin{align}
\label{Mixed of g calculation}
    & \sup_{x'} \Big(\int_{\mathbb{R}^{d_2}} \Big|\int_{|x'-z'|>2^{j(1+\varepsilon)}} \frac{\mathcal{M}^{|\cdot|_2}(g(z', \cdot))(x'')}{(1 + |x'-z'| )^{N-2d_2-\epsilon_1}} \, dz' \Big|^2 dx'' \Big)^{1/2} \\
    &\nonumber \leq \sup_{x'} \Big(\int_{|x'-y'|>2^{j(1+\varepsilon)}} \frac{1}{(1 +  |x'-z'| )^{N-2d_2-\epsilon_1}} \Big(\int_{\mathbb{R}^{d_2}} |\mathcal{M}^{|\cdot|_2}(g(z', \cdot))(x'')|^2 \, dx'' \Big)^{1/2} dz' \Big) \\
    &\nonumber \leq \sup_{x'} \Big(\int_{|x'-y'|>2^{j(1+\varepsilon)}} \frac{\|g(z', \cdot)\|_{L_{x''}^2}}{(1 +  |x'-z'| )^{N-2d_2-\epsilon_1}} dz' \Big) \\
    &\nonumber \leq C 2^{-j(1+\varepsilon)(N-Q-\epsilon)} \sup_{x'}|\mathcal{M}^{|\cdot|_1}(\|g\|_{L_{x''}^2})(x')| \\
    &\nonumber \leq C 2^{-j(1+\varepsilon)(N-Q-\epsilon_1)} \|g\|_{L_{x''}^2 L_{x'}^{\infty}} ,
\end{align}
provided $N>Q+\epsilon_1$.

From Lemma \ref{Lemma: Integral of weight for general h} taking $h\equiv 1$, for $N>d_2$ we have the following estimate,
\begin{align}
\label{Computation of integral for h equals to 1 case}
    \int_{|x'-y'|\leq 2^{j(1+\varepsilon)}} \left( \int_{|x''-y''|>[2^{j(1+\varepsilon)},|x'|]} \frac{dx''}{(1+\varrho(x,y))^N} \right)\, dx' &\leq C 2^{-j(1+\varepsilon)(N-Q)} |B(y,1)| .
\end{align}

Hence, combining the estimates \eqref{Mixed of g calculation} and \eqref{Computation of integral for h equals to 1 case} and inserting into \eqref{Application of Holder in mixed norm for outside}, we conclude that
\begin{align*}
    \|\mathcal{B}^{\alpha}_{j,A_2,A_4}(f,g)\|_{L_{x''}^{2/3} L_{x'}^{1}} &\leq C 2^{j(2N+1/2+\epsilon_1)} 2^{-j(1+\varepsilon)(N-Q)} 2^{-j(1+\varepsilon)(N-Q-\epsilon_1)} \|f\|_{L^1} \|g\|_{L_{x''}^2 L_{x'}^{\infty}} \\
    &\leq C 2^{-j \delta} \|f\|_{L^1} \|g\|_{L_{x''}^2 L_{x'}^{\infty}} ,
\end{align*}
for some $\delta>0$, provided we choose $N>0$ sufficiently large and $\epsilon_1>0$ very small.

Therefore it only remains to estimate $\mathcal{B}^{\alpha}_{j,A_{1}, A_{1}}$, which follows a similar approach to the estimate of $\mathcal{B}^{\alpha}_{j,1}$ (see \ref{subsection: Estimate of Bj1 case}). We choose sequences $\{a_{n'}'\}_{n' \in \mathbb{N}}$ such that for $n'\neq m'$,
\begin{align*}
    |a_{n'}'-a_{m'}'|>2^{j(1+\varepsilon)}/10, \quad \text{and} \quad \sup_{a' \in \mathbb{R}^{d_1}} \inf_{n'}|a'-a_{n'}'| \leq 2^{j(1+\varepsilon)}/10 .
\end{align*}
For each $a_{n'}'$, we choose sequences $\{a_{n''}''\}_{n'' \in \mathbb{N}}$ such that for $n'' \neq m''$,
\begin{align*}
      |a_{n''}''-a_{m''}''|>[2^{j(1+\varepsilon)},|a_{n'}'|]/10 \quad  \text{and} \quad \sup_{a'' \in \mathbb{R}^{d_2}} \inf_{n''}|a''-a_{n''}''| \leq [2^{j(1+\varepsilon)},|a_{n'}'|]/10 .
\end{align*}
With the help of these sequences, define $\displaystyle{S_{n'}^{|\cdot|,j} := \Bar{B}^{|\cdot|}(a_{n'}', \tfrac{2^{j(1+\varepsilon)}}{10}) \setminus \cup_{m' < n'} \Bar{B}^{|\cdot|}(a_{m'}', \tfrac{2^{j(1+\varepsilon)}}{10})}$ and $\displaystyle{S_{n''}^{|\cdot|,j} := \Bar{B}^{|\cdot|}(a_{n''}'', \tfrac{[2^{j(1+\varepsilon)},|a_{n'}'|]}{10}) \setminus \cup_{m'' < n''} \Bar{B}^{|\cdot|}(a_{m''}'', \tfrac{[2^{j(1+\varepsilon)},|a_{n'}'|]}{10})}$. Similar to \eqref{Bounded overlapping property in j balls}, we also have the following bounded overlapping property:
\begin{align}
\label{bounded overlapping for mixed norms}
    \sup_{n'} \#\{m' : |a_{n'}'-a_{m'}'| \leq  2 \cdot 2^{j(1+\varepsilon)}\} \leq C\  \text{and} \  \sup_{n''} \#\{m'' : |a_{n''}''-a_{m''}''| \leq  2 \cdot 2^{j(1+\varepsilon)}\} \leq C .
\end{align}
Recalling \eqref{decomposition of kernel in mixed case}, we readily observe that
\begin{align*}
    &\supp{\mathcal{K}^{\alpha}_{j,A_1, A_1}} \subseteq \mathcal{D}_{j}^{|\cdot|} := \bigcup_{n' \in \mathbb{N}} \Big\{(x,y,z) \in (S_{n'}^{|\cdot|,j} \times \mathbb{R}^{d_2}) \times (\mathbb{R}^{d_1} \times \mathbb{R}^{d_2}) \times (\mathbb{R}^{d_1} \times \mathbb{R}^{d_2}) :  \\
    & \hspace{6cm} |x'-y'|\leq 2^{j(1+\varepsilon)}, |x''-y''| \leq 2 [2^{j(1+\varepsilon)},|a_{n'}'|], \\
    & \hspace{6.5cm} |x'-z'|\leq 2^{j(1+\varepsilon)}, |x''-z''| \leq 2 [2^{j(1+\varepsilon)},|a_{n'}'|] \Big\} ,
\end{align*}
which immediately implies
\begin{align*}
    \mathcal{D}_{j}^{|\cdot|} &\subseteq \bigcup_{\substack{n',n'',n_1',n_1'',,n_2',n_2'': \\ |a_{n'}'-a_{n_1'}'|\leq 2 \cdot 2^{j(1+\varepsilon)}, |a_{n''}''-a_{n_1''}''|\leq 4 \cdot [2^{j(1+\varepsilon)},|a_{n'}'|] \\
    |a_{n'}'-a_{n_2'}'|\leq 2 \cdot 2^{j(1+\varepsilon)}, |a_{n''}''-a_{n_2''}''|\leq 4 \cdot [2^{j(1+\varepsilon)},|a_{n'}'|]}} (S_{n'}^{|\cdot|,j} \times S_{n''}^{|\cdot|,j}) \times \left((S_{n_1'}^{|\cdot|,j} \times S_{n_1''}^{|\cdot|,j}) \times (S_{n_2'}^{|\cdot|,j} \times S_{n_2''}^{|\cdot|,j}) \right) .
\end{align*}
To this end, we introduce the short hand notation: for $i=1,2$,
\begin{align*}
    \sum_{n_i':} := \sum_{n_i':|a_{n'}'-a_{n_i'}'| \leq 2 \cdot 2^{j(1+\varepsilon)}} \quad \text{and} \quad \sum_{n_i'':} := \sum_{n_i'':|a_{n''}''-a_{n_i''}''| \leq 4 \cdot [2^{j(1+\varepsilon)},|a_{n'}'|]} .
\end{align*}
Therefore, we can decompose the operator $\mathcal{B}^{\alpha}_{j,A_{1}, A_{1}}$ as follows: 
\begin{align*}
    \mathcal{B}^{\alpha}_{j,A_{1}, A_{1}}(f,g)(x) &= \sum_{n',n''=0}^{\infty} \sum_{n_1':,n_1'':, n_2':,n_2'':} \chi_{S_{n'}^{|\cdot|,j} \times S_{n''}^{|\cdot|,j}}(x) \mathcal{B}^{\alpha}_{j,A_{1}, A_{1}}(f_{n_1', n_1''}^j,g_{n_2', n_2''}^j)(x) ,
\end{align*}
where $f_{n_1', n_1''}^j = f \chi_{S_{n_1'}^{|\cdot|,j} \times S_{n_1''}^{|\cdot|,j}}$ and $g_{n_2', n_2''}^j = g \chi_{S_{n_2'}^{|\cdot|,j} \times S_{n_2''}^{|\cdot|,j}}$.

Before moving forward, let us make the following claim. There exist some $\epsilon_1>0$, such that
\begin{align}
\label{Expression: Assumption for mixed norms}
    \|\chi_{S_{n'}^{|\cdot|,j} \times S_{n''}^{|\cdot|,j}} \mathcal{B}^{\alpha}_{j,A_{1}, A_{1}}(f_{n_1', n_1''}^j,g_{n_2', n_2''}^j)\|_{L_{x''}^{2/3} L_{x'}^{1}} & \leq C 2^{-j \alpha} 2^{j \epsilon_1} 2^{j/2} 2^{j d_2/2} \|f_{n_1', n_1''}^j\|_{L^{1}} \|g_{n_2', n_2''}^j\|_{L^{2}} .
\end{align}
Assuming the claim for a moment, we proceed to complete  the estimate of $\mathcal{B}^{\alpha}_{j,A_{1}, A_{1}}$. Disjointness of the sets $S_{n'}^{|\cdot|,j}$ and $S_{n''}^{|\cdot|,j}$ immediately gives
\begin{align*}
    & \|\mathcal{B}^{\alpha}_{j,A_{1}, A_{1}} (f,g)\|_{L_{x''}^{2/3} L_{x'}^{1}} = \Big\|\sum_{n',n''=0}^{\infty} \sum_{n_1':,n_1'':, n_2':,n_2'':} \chi_{S_{n'}^{|\cdot|,j} \times S_{n''}^{|\cdot|}} \mathcal{B}^{\alpha}_j (f_{n_1',n_1''}^j,g_{n_2',n_2''}^j) \Big\|_{L_{x''}^{2/3} L_{x'}^{1}} \\
    &= \sum_{n'=0}^{\infty} \Big\{\Big[ \int_{\mathbb{R}^{d_1}} \Big(\sum_{n''=0}^{\infty} \int_{\mathbb{R}^{d_2}} \Big| \sum_{n_1':,n_1'':, n_2':,n_2'':} \chi_{S_{n'}^{|\cdot|} \times S_{n''}^{|\cdot|}} \mathcal{B}^{\alpha}_{j,A_{1}, A_{1}} (f_{n_1',n_1''}^j,g_{n_2',n_2''}^j) \Big|^{2/3} \ dx'' \Big)^{3/2} \ dx' \Big]^{2/3} \Big\}^{3/2} .
\end{align*}
Applying \eqref{Triangle inequality for p less than 1 case} for the exponent $2/3$ followed by using triangle inequality for the exponent $3/2$, we can control the last quantity by
\begin{align*}
    & \sum_{n'=0}^{\infty} \Big\{\sum_{n''=0}^{\infty} \sum_{n_1':,n_1'':, n_2':,n_2'':} \Big[ \int_{\mathbb{R}^{d_1}} \Big( \int_{\mathbb{R}^{d_2}} \Big| \chi_{S_{n'}^{|\cdot|,j} \times S_{n''}^{|\cdot|}} \mathcal{B}^{\alpha}_{j,A_{1}, A_{1}} (f_{n_1',n_1''}^j,g_{n_2',n_2''}^j) \Big|^{2/3} \ dx'' \Big)^{3/2} \ dx' \Big]^{2/3} \Big\}^{3/2} .
\end{align*}
Consequently, using the claim \eqref{Expression: Assumption for mixed norms}, the above expression can be bounded by
\begin{align*}
    &  C 2^{-j \alpha} 2^{j \epsilon_1} 2^{j/2} 2^{j d_2/2} \sum_{n'=0}^{\infty} \Big\{\sum_{n''=0}^{\infty} \sum_{n_1':,n_1'':} \Big( \int_{\mathbb{R}^{d_1}} \int_{\mathbb{R}^{d_2}} | f_{n_1',n_1''}^j(y',y'')| \ dy''\ dy' \Big)^{2/3} \\
    & \hspace{7cm} \sum_{n_2':, n_2'':} \Big( \int_{\mathbb{R}^{d_1}} \int_{\mathbb{R}^{d_2}} | g_{n_2',n_2''}^j(z',z'')|^{2}  \ dz''\ dz' \Big)^{1/3} \Big\}^{3/2} .
\end{align*}

In addition, applying H\"older's inequality and bounded overlapping property \eqref{bounded overlapping for mixed norms}, the above expression can be further dominated by
\begin{align*}
    & C 2^{-j \alpha} 2^{j \epsilon_1} 2^{j/2} 2^{j d_2/2} \sum_{n'=0}^{\infty} \left\{\Big(\sum_{n''=0}^{\infty} \sum_{n_1':} \int_{\mathbb{R}^{d_1}} \sum_{n_1'':} \int_{\mathbb{R}^{d_2}} | f_{n_1',n_1''}^j(y',y'')| \ dy''\ dy' \Big)^{2/3} \right.  \\
    & \hspace{6cm} \left. \Big(\sum_{n''=0}^{\infty} \sum_{n_2':} \int_{\mathbb{R}^{d_1}} \sum_{n_2''} \int_{\mathbb{R}^{d_2}} | g_{n_2',n_2''}^j(z',z'')|^{2} \ dz'' \ dz' \Big)^{1/3} \right\}^{3/2} .
\end{align*}

Again, invoking bounded overlapping property \eqref{bounded overlapping for mixed norms}, we observe that the above quantity can be controlled by
\begin{align*}
    & C 2^{-j \alpha} 2^{j \epsilon_1} 2^{j/2} 2^{j d_2/2} 2^{j(1+\varepsilon) d_1/2} \Biggl\{\sum_{n'=0}^{\infty} \sum_{n_1':} \int_{S_{n_1'}^{|\cdot|,j}} \sum_{n''=0}^{\infty} \sum_{n_1'':} \int_{S_{n_1''}^{|\cdot|,j}} | f(y',y'')| \ dy''\ dy' \Biggr\}\\
    & \hspace{2cm} \sup_{n'} \Biggl\{\sum_{|a_{n'}'-a_{n_2'}'| \leq 2 \cdot 2^{j(1+\varepsilon)}} \sup_{z' \in B^{|\cdot|}(a_{n_2'}', 2^{j(1+\varepsilon)})} \Big( \sum_{n''=0}^{\infty} \sum_{n_2''} \int_{S_{n_2''}^{|\cdot|,j}} | g(z',z'')|^{2} \ dz'' \Big) \Biggr\}^{1/2} \\
    &\leq C 2^{-j \alpha} 2^{j \epsilon_1} 2^{j/2} 2^{j d_2/2} 2^{j(1+\varepsilon) d_1/2} \Big\{\int_{\mathbb{R}^{d_1}} \int_{\mathbb{R}^{d_2}} | f(y',y'')| \ dy''\ dy' \Big\} \sup_{z'} \Big\{ \int_{\mathbb{R}^{d_2}} | g(z',z'')|^{2} \ dz'' \Big\}^{1/2} \\
    % &\leq C 2^{-j \delta} \Biggl\{\int_{\mathbb{R}^{d_2}} \int_{\mathbb{R}^{d_1}} | f(x,u)| \ dx\ du \Biggr\} \Biggl\{\sup_{u} \left( \int_{\mathbb{R}^{d_1}} | g(x,u)|^{2} \ dx \right)^{\frac{1}{2}} \Biggr\} \\
    &\leq C 2^{-j \delta} \|f\|_{L_{x''}^{1}L_{x'}^{1}} \|g\|_{L_{x''}^{2} L_{x'}^{\infty}} ,
\end{align*}
where as $\alpha>(d+1)/2$, we can choose $\varepsilon, \epsilon_1>0$ sufficiently small such that $\delta=\alpha-(d+1)/2+d_1 \varepsilon/2+\epsilon_1>0$. This completes the estimate of $\mathcal{B}^{\alpha}_{j,A_{1}, A_{1}}$, upon assuming the claim.

Therefore, it remains only to prove the claim \eqref{Expression: Assumption for mixed norms}. We divide the proof into two cases: $|a_{n'}'|>\frac{11}{5} 2^{j(1+\varepsilon)}$ and $|a_{n'}'|\leq \frac{11}{5} 2^{j(1+\varepsilon)}$. Let us start addressing the claim \eqref{Expression: Assumption for mixed norms} accordingly.

\subsection{Case-I : \texorpdfstring{$|a_{n'}'|>\frac{11}{5} 2^{j(1+\varepsilon)}$}{}}
Application of H\"older's inequality gives
\begin{align*}
    &\|\chi_{S_{n'}^{|\cdot|,j} \times S_{n''}^{|\cdot|,j}} \mathcal{B}^{\alpha}_{j,A_1, A_1}(f_{n_1', n_1''}^j, g_{n_2',n_2''}^j)\|_{L_{x''}^{2/3} L_{x'}^{1}} \lesssim 2^{j d_2(1+\varepsilon)/2} |a_{n'}'|^{d_2/2} \Big\|\sum_{l \in \mathbb{Z}} \phi_{j,l} ^{\alpha}(\mathcal{L}) f_{n_1', n_1''}^j \psi_{l}(\mathcal{L}) g_{n_2',n_2''}^j\Big\|_{L^1} \\
    &\lesssim \sum_{l \in \mathbb{Z}} 2^{j d_2(1+\varepsilon)/2} |a_{n'}'|^{d_2/2} \| \phi_{j,l}^{\alpha}(\mathcal{L})f_{n_1', n_1''}^j \|_{L^2} \| \psi_l(\mathcal{L})g_{n_2',n_2''}^j \|_{L^2} .
\end{align*}

We have $|a_{n'}' - a_{n_i'}'|\leq 2 \cdot 2^{j(1+\varepsilon)}$ for $i=1,2$. Then as in subsection \ref{subsection: case I of infinity infinity}, one can easily see that $|a_{n'}'| \sim |a_{n_i'}'|$. Thus, invoking Proposition \ref{Weighted restriction estimate}, we find that the above expression is dominated by
\begin{align*}
    & C \sum_{l \in \mathbb{Z}} 2^{j d_2(1+\varepsilon)/2} |a_{n'}'|^{d_2/2} |a_{n_1'}'|^{-d_2/2} \|\phi_{j,l}^{\alpha}\|_{L^{\infty}} \| f_{n_1', n_1''}^j \|_{L^1} \| g_{n_2',n_2''}^j \|_{L^2} \\
    &\leq C 2^{j\epsilon_1} 2^{-j \alpha} 2^{j d_2/2} \| f_{n_1', n_1''}^j \|_{L^1} \| g_{n_2',n_2''}^j \|_{L^2} ,
\end{align*}
for some $\epsilon_1>0$.

\subsection{Case-II : \texorpdfstring{$|a_{n'}'|\leq \frac{11}{5} 2^{j(1+\varepsilon)}$}{}}
Similar to (\ref{Decomposition of phi into M1 cutoff}), we first decompose
\begin{align*}
    & \chi_{S_{n'}^{|\cdot|,j} \times S_{n''}^{|\cdot|,j}}(x) \mathcal{B}^{\alpha}_{j,A_1, A_1}(f_{n_1', n_1''}^j, g_{n_2',n_2''}^j)(x) \\
    &= C \chi_{S_{n'}^{|\cdot|,j} \times S_{n''}^{|\cdot|,j}}(x) \Big(\sum_{M_1=0 }^{j} + \sum_{M_1=j+1 }^{\infty} \Big) \sum_{l \in \mathbb{Z}} \phi_{j,l,M_1}^{\alpha}(\mathcal{L}) f_{n_1', n_1''}^j(x) \psi_{l}(\mathcal{L}) g_{n_2',n_2''}^j(x) =: E_1 + E_2 .
\end{align*}

\subsubsection{\textbf{Estimate of} \texorpdfstring{$E_{1}$}{}{}}
As in \ref{subsection: Estimate of J1 case}, we decompose the set $S_{n_1''}^{|\cdot|,j}$ into disjoint subsets $S_{n_1'',m_1''}^{|\cdot|, M_1,j}$ such that $S_{n_1'',m_1''}^{|\cdot|, M_1,j} \subseteq B^{|\cdot|}(b_{n_1'', m_1''}^{M_1}, 2^{(j+M_1)(1+\varepsilon)}/5)$, and for any $m_1'' \neq \Tilde{m}_1''$, we have the $|b_{n_1'', m_1''}^{M_1}-b_{n_1'', \Tilde{m}_1''}^{M_1}|>2^{(j+M_1)(1+\varepsilon)}/10$. For $\gamma>0$, we define $\Tilde{S}_{n_1'',m_1''}^{|\cdot|, M_1,j} := B^{|\cdot|}(b_{n_1'', m_1''}^{M_1}, 2^{(j+M_1)(1+\varepsilon)} 2^{\gamma j+1}/5)$, and decompose the function $f_{n_1', n_1''}^j$ as follows $f_{n_1', n_1''}^j = \sum_{m_1''=1}^{N_{M_1}} f_{n_1', n_1'',m_1''}^{M_1,j}$, where $f_{n_1', n_1'',m_1''}^{M_1,j} = f_{n_1', n_1''}^j \chi_{S_{n_1'',m_1''}^{|\cdot|, M_1,j}}$. Accordingly, we decompose $E_1$ into two parts as follows.
\begin{align*}
    & E_{1} = \sum_{M_1=0 }^{j} \sum_{m_1''=1}^{N_{M_1}} \sum_{l\in \mathbb{Z}} \chi_{S_{n'}^{|\cdot|,j} \times S_{n''}^{|\cdot|,j}}(x) \chi_{S_{n'}^{|\cdot|,j} \times \widetilde{S}_{n_1'',m_1''}^{|\cdot|, M_1,j}}(x) \phi_{j,l,M_1}^{\alpha}(\mathcal{L}, T) f_{n_1', n_1'',m_1''}^{M_1,j}(x)\, \psi_{l}(\mathcal{L}) g_{n_2',n_2''}^j(x) \\
    &+ \sum_{M_1=0 }^{j}  \sum_{m_1''=1}^{N_{M_1}} \sum_{l\in \mathbb{Z}} \chi_{S_{n'}^{|\cdot|,j} \times S_{n''}^{|\cdot|,j}}(x) (1-\chi_{S_{n'}^{|\cdot|,j} \times \widetilde{S}_{n_1'',m_1''}^{|\cdot|, M_1,j}})(x) \phi_{j,l,M_1}^{\alpha}(\mathcal{L}, T) f_{n_1', n_1'',m_1''}^{M_1,j}(x)\, \psi_{l}(\mathcal{L}) g_{n_2',n_2''}^j(x) \\
    &=: E_{11} + E_{12} .
\end{align*}

\subsubsection{\textbf{Estimate of} \texorpdfstring{$E_{12}$}{}{}}
Note that $E_{12}$ can be estimated in a similar manner to $J_{12}$ (see subsection \ref{Estimate of J12 for 1,2 case}), with the appropriate modifications.

\subsubsection{\textbf{Estimate of} \texorpdfstring{$E_{11}$}{}{}}
In this case, we have
\begin{align*}
    E_{11} &= C\sum_{M_1=0 }^{j} \sum_{l\in \mathbb{Z}} \Big\{ \sum_{m_1''=1}^{N_{M_1}} \chi_{S_{n'}^{|\cdot|,j} \times \widetilde{S}_{n_1'',m_1''}^{|\cdot|, M_1,j}}(x) \phi_{j,l,M_1}^{\alpha}(\mathcal{L}, T) f_{n_1', n_1'',m_1''}^{M_1,j}(x) \Big\} \{ \psi_{l}(\mathcal{L}) g_{n_2',n_2''}^j(x) \} .
\end{align*}
Applying the fact \eqref{Triangle inequality for p less than 1 case} and triangle inequality, we can see
\begin{align*}
    \|E_{11}\|_{L_{x''}^{2/3} L_{x'}^1}^{2/3} &\leq C\sum_{M_1=0 }^{j} \sum_{l \in \mathbb{Z}} \Big[\int_{\mathbb{R}^{d_1}} \Big( \int_{\mathbb{R}^{d_2}} \Big| \sum_{m_1''=1}^{N_{M_1}} \chi_{S_{n'}^{|\cdot|,j} \times \widetilde{S}_{n_1'',m_1''}^{|\cdot|, M_1,j}}(x) \phi_{j,l,M_1}^{\alpha}(\mathcal{L}, T) f_{n_1', n_1'',m_1''}^{M_1,j}(x) \\
    &\hspace{6cm} \psi_{l}(\mathcal{L}) g_{n_2',n_2''}^j(x) \Big|^{2/3} dx'' \Big)^{3/2} dx' \Big]^{2/3} .
\end{align*}
Now, by initially applying H\"older's inequality with respect to $x''$-variable, and subsequently with respect to $x'$-variable, the right hand side of the above expression is dominated by
\begin{align*}
      & C\sum_{M_1=0 }^{j} \sum_{l \in \mathbb{Z}} \Big[\int_{\mathbb{R}^{d_1}} \Big(\int_{\mathbb{R}^{d_2}} \Big| \sum_{m_1''=1}^{N_{M_1}} \chi_{S_{n'}^{|\cdot|,j} \times \widetilde{S}_{n_1'',m_1''}^{|\cdot|, M_1,j}}(x) \phi_{j,l,M_1}^{\alpha}(\mathcal{L}, T) f_{n_1', n_1'',m_1''}^{M_1,j}(x) \Big| dx'' \Big)^{2} dx' \Big]^{\frac{1}{2} \cdot \frac{2}{3}} \\
      & \hspace{6cm} \Big[\int_{\mathbb{R}^{d_1}} \int_{\mathbb{R}^{d_2}} | \psi_{l}(\mathcal{L}) g_{n_2',n_2''}^j(x) |^2 dx'' dx' \Big]^{\frac{1}{2} \cdot \frac{2}{3}} 
\end{align*}
Note that $|\widetilde{S}_{n_1'',m_1''}^{|\cdot|, M_1,j}| \lesssim 2^{(j+M_1)(1+\varepsilon) d_2/2} 2^{C \gamma j}$. Application of triangle inequality along with H\"older's inequality, we see that the above quantity is further bounded by
\begin{align*}
    & C \sum_{M_1=0 }^{j} \sum_{l \in \mathbb{Z}} \Big[\sum_{m_1=1}^{N_{M_1}} 2^{(j+M_1)(1+\varepsilon) d_2/2} 2^{C \gamma j} \| \phi_{j,l,M_1}^{\alpha}(\mathcal{L}, T) f_{n_1', n_1'',m_1''}^{M_1,j} \|_{L^2} \Big]^{2/3} \| \psi_{l}(\mathcal{L}) g_{n_2',n_2''}^j \|_{L^2}^{2/3} .
\end{align*}
With the help of the above estimate, together with Proposition \ref{Lemma: Martini_Mullar_Weighted_Plancherel} and \eqref{sum of l for 1,2 case}, we can conclude that
\begin{align*}
     \|E_{11}\|_{L_{x''}^{2/3} L_{x'}^1}^{2/3} &\lesssim \sum_{M_1=0 }^{j} \sum_{l \in \mathbb{Z}} \Big[2^{(j+M_1) d_2/2} 2^{C \gamma j} 2^{-M_1 d_2/2} \|\phi_{j,l}^{\alpha}\|_{L^{\infty}} \sum_{m_1=1}^{N_{M_1}} \| f_{n_1', n_1'',m_1''}^{M_1,j} \|_{L^1} \Big]^{2/3} \| g_{n_2',n_2''}^j \|_{L^2}^{2/3} \\
     &\leq C 2^{j \epsilon_1} 2^{-j 2\alpha /3} 2^{j/3} 2^{j d_2/3} \| f_{n_1', n_1''}^{j} \|_{L^1}^{2/3} \| g_{n_2',n_2''}^j \|_{L^2}^{2/3} ,
\end{align*}
for some $\epsilon_1>0$.

Thus, we arrive at
\begin{align*}
    \|E_{11}\|_{L_{x''}^{2/3} L_{x'}^1} &\leq C 2^{j \epsilon_1} 2^{-j \alpha} 2^{j/2} 2^{j d_2/2} \| f_{n_1',n_1''}^j \|_{L^1} \| g_{n_2',n_2''}^j \|_{L^2} .
\end{align*}

\subsubsection{\textbf{Estimate of} \texorpdfstring{$E_{2}$}{}{}}
This estimate resembles the estimate for $J_2$ (see \ref{Estimate of J2 for 1,2 case}). Utilizing H\"older's inequality, Proposition \ref{Lemma: Martini_Mullar_Weighted_Plancherel}, and the fact from \eqref{Convergence of sum of l using phi} leads to
\begin{align*}
    \|E_2\|_{L_{x''}^{2/3} L_{x'}^{1}} & \leq C 2^{j 2d_2(1+\varepsilon)/2} \Big\|\sum_{l \in \mathbb{Z}} \sum_{M_1=j+ 1 }^{\infty} \phi_{j,l,M_1} ^{\alpha}(\mathcal{L}, T) f_{n_1', n_1''}^j \, \psi_{l}(\mathcal{L}) g_{n_2',n_2''}^j\Big\|_{L^1} \\
    &\leq C \sum_{l \in \mathbb{Z}} \sum_{M_1=j+ 1 }^{\infty} 2^{j d_2(1+\varepsilon)} \| \phi_{j,l, M_1}^{\alpha}(\mathcal{L}, T)f_{n_1', n_1''}^j \|_{L^2} \| \psi_l(\mathcal{L})g_{n_2',n_2''}^j \|_{L^2} \\
    &\leq C \sum_{l \in \mathbb{Z}} \sum_{M_1=j+ 1 }^{\infty} 2^{j d_2(1+\varepsilon)} 2^{-M_1 d_2/2} \|\phi_{j,l}^{\alpha}\|_{L^{\infty}} \| f_{n_1', n_1''}^j \|_{L^1} \| g_{n_2',n_2''}^j \|_{L^2} \\
    &\leq C 2^{j\epsilon_1} 2^{-j \alpha} 2^{j d_2/2} \| f_{n_1', n_1''}^j \|_{L^1} \| g_{n_2',n_2''}^j \|_{L^2} ,
\end{align*}
for some $\epsilon_1>0$.

This completes the proof the claim \eqref{Expression: Assumption for mixed norms}.

\section*{Acknowledgments}
The second author gratefully acknowledges the financial support provided by the NBHM Post-doctoral fellowship, DAE, Government of India. The third author would like to acknowledge the support of the Prime Minister's Research Fellows (PMRF) supported by Ministry of Education, Government of India.

% \bibliographystyle{amsalpha}
% \bibliography{References}

\providecommand{\bysame}{\leavevmode\hbox to3em{\hrulefill}\thinspace}
\providecommand{\MR}{\relax\ifhmode\unskip\space\fi MR }
\providecommand{\MRhref}[2]{%
  \href{http://www.ams.org/mathscinet-getitem?mr=#1}{#2}
}
\providecommand{\href}[2]{#2}

\end{document}